\documentclass[reqno]{amsart}
\usepackage[top=1.1in, bottom=1in, left=1in, right=1in]{geometry}
\usepackage{amsfonts}
\usepackage{amssymb}
\usepackage{amsthm}
\usepackage{amsmath}
\usepackage{mathrsfs}
\usepackage{color}
\usepackage{enumerate}
\usepackage[numbers,sort&compress]{natbib}
\usepackage{pdfsync}
\usepackage{esint}
\usepackage{graphicx}
\usepackage{float}
\usepackage{caption}
\usepackage{subfigure}

\allowdisplaybreaks

\usepackage{tikz} % Use TikZ for drawing
\usetikzlibrary{calc}
\usetikzlibrary{intersections}
\usetikzlibrary{patterns}
\usepackage[colorlinks,
            linkcolor=black,       %%修改此处为你想要的颜色
            anchorcolor=blue,      %%修改此处为你想要的颜色
            citecolor=blue,        %%修改此处为你想要的颜色，例如修改blue为red
            ]{hyperref}
\makeatother

\numberwithin{equation}{section}

\newtheorem{theorem}{Theorem}[section]
\newtheorem{definition}[theorem]{Definition}

\newtheorem{lemma}[theorem]{Lemma}
\newtheorem{remark}[theorem]{Remark}

\newtheorem{proposition}[theorem]{Proposition}
\newtheorem{corollary}[theorem]{Corollary}
\numberwithin{equation}{section}

\newcommand*{\Id}{\ensuremath{\mathrm{Id}}}
\newcommand*{\supp}{\ensuremath{\mathrm{supp\,}}}

\newcommand{\aint}{{\fint}}

\newcommand{\T}{{\mathbb{T}}}

\newcommand{\rr}{\mathring{R}}
\newcommand{\tr}{\mathring{\wt R}}
\newcommand{\ru}{\mathring{R}_q}

\newcommand{\p}{\partial}
\renewcommand{\P}{\mathbb{P}}

\renewcommand{\div}{{\mathrm{div}}}
\newcommand{\curl}{{\mathrm{curl}}}

\renewcommand{\u}{{u_q}}

\renewcommand{\d}{{\rm d}}

\newcommand{\g}{g_{(\tau)}}

\newcommand{\norm}[1]{\lVert#1\rVert}
\newcommand{\abs}[1]{|#1|}

\newcommand{\la}{\lambda_q}
\newcommand{\laq}{\lambda_{q+1}}

\newcommand{\rs}{r_{\perp}}
\newcommand{\rp}{r_{\parallel}}
\newcommand{\va}{\varepsilon}

\newcommand{\wo}{w_{q+1}^{(o)}}
\newcommand{\dqo}{w_{q+1}^{(o)}}

\newcommand{\thq}{\theta_{q+1}}
\newcommand{\mq}{m_{q+1}}

\newcommand{\dist}{\operatorname{dist}}

\newcommand{\R}{{\mathbb R}}
\newcommand{\lbb}{\overline{\lambda}}

\def\a{{\alpha}}
\def\vf{{\varphi}}
\def\lbb{\lambda}
\def\wt{\widetilde}
\def\9{{\infty}}
\def\ve{{\varepsilon}}
\def\na{{\nabla}}
\def\bbr{{\mathbb{R}}}

\def\({\left(}
\def\){\right)}

\begin{document}
	
\title[] {Sharp non-uniqueness for the 3D hyperdissipative Navier-Stokes equations:
above the Lions exponent}

\author{Yachun Li}
\address{School of Mathematical Sciences, CMA-Shanghai, MOE-LSC, and SHL-MAC,  Shanghai Jiao Tong University, China.}
\email[Yachun Li]{ycli@sjtu.edu.cn}
\thanks{}

\author{Peng Qu}
\address{School of Mathematical Sciences $\&$ Shanghai Key Laboratory for Contemporary Applied Mathematics, Fudan University, China.}
\email[Peng Qu]{pqu@fudan.edu.cn}
\thanks{}

\author{Zirong Zeng}
\address{School of Mathematical Sciences, Shanghai Jiao Tong University, China.}
\email[Zirong Zeng]{beckzzr@sjtu.edu.cn}
\thanks{}

\author{Deng Zhang}
\address{School of Mathematical Sciences, CMA-Shanghai, Shanghai Jiao Tong University, China.}
\email[Deng Zhang]{dzhang@sjtu.edu.cn}
\thanks{}



\begin{abstract}
	We study the 3D hyperdissipative Navier-Stokes equations on the torus,
	where the viscosity exponent $\alpha$ can be larger than the Lions exponent $5/4$.
	It is well-known that,
	due to Lions \cite{lions69},
	for any $L^2$ divergence-free initial data,
	there exist unique smooth Leray-Hopf solutions when $\alpha \geq 5/4$.
	We prove that even in this high dissipative regime,
	the uniqueness would fail in the supercritical spaces $L^\gamma_tW^{s,p}_x$,
	in view of the generalized Lady\v{z}enskaja-Prodi-Serrin condition.
	The non-uniqueness is proved in the strong sense
	and, in particular, yields the sharpness at two endpoints
	$(3/p+1-2\alpha, \infty, p)$
and $(2\alpha/\gamma+1-2\alpha, \gamma, \infty)$.
	Moreover, the constructed solutions are allowed to coincide with the unique
	Leray-Hopf solutions near the initial time
	and, more delicately,
	admit the partial regularity outside a fractal set of singular times
	with zero Hausdorff $\mathcal{H}^{\eta_*}$ measure,
	where $\eta_*>0$ is any given small positive constant.
	These results also provide the sharp non-uniqueness
	in the supercritical Lebesgue and Besov spaces.
	Furthermore, the strong vanishing viscosity result is obtained
	for the hyperdissipative Navier-Stokes equations.
\end{abstract}

\maketitle

{
\tableofcontents
}

\section{Introduction and main results}

\subsection{Background} \label{Subsec-intro}

We consider the
three-dimensional hyperdissipative Navier-Stokes equations on the torus $\T^3:=[-\pi,\pi]^3$,
\begin{equation}\label{equa-NS}
	\left\{\aligned
	&\p_tu +\nu(-\Delta)^{\alpha} u+(u\cdot \nabla )u + \nabla P=0,  \\
	& \div u = 0,
	\endaligned
	\right.
\end{equation}
where  $u=(u_1,u_2,u_3)^\top(t,x)\in \R^3 $
and $P=P(t,x)\in \R$
represent the velocity field and pressure of the fluid, respectively,
$\nu> 0$ is the viscous coefficient,
$\alpha\in [1,2)$,
and $(-\Delta)^{\alpha}$ is the fractional Laplacian defined
via the Fourier transform on the flat torus
$$
\mathcal{F}((-\Delta)^{\alpha}u)(\xi)=|\xi|^{2\alpha}\mathcal{F}(u)(\xi),\ \ \xi\in \mathbb{Z}^3.
$$
In particular,
\eqref{equa-NS} are the classical Navier-Stokes equations  (NSE for short)  when $\alpha =1$,
and the Euler equations when the viscosity vanishes, i.e., $\nu=0$,

\begin{equation}\label{equa-Euler}
	\left\{\aligned
	&\p_tu +(u\cdot \nabla )u + \nabla P=0,  \\
	& \div u = 0.
	\endaligned
	\right.
\end{equation}

In the groundbreaking paper \cite{leray1934},
Leray constructed the weak solutions to NSE in the space
$L^\infty_tL^2_x\cap L^2_t\dot{H}_x^1$,
which obey the energy inequality
\begin{align}\label{nsenergy}
	\|u(t)\|_{L^2}^2
	+ 2 \nu \int_{t_0}^t \|(-\Delta)^\frac \alpha 2 u(s)\|_{L^2}^2 \d s
	\leq \|u(t_0)\|_{L^2}^2
\end{align}
with $\alpha =1$, for any $t>0$ and a.e. $t_0\geq 0$.
This class of weak solutions is now referred to as Leray-Hopf weak solutions,
due to the important contributions by Hopf \cite{hopf1951}
in the case of bounded domains.
Moreover,
Leray \cite{leray1934} proved that
for every such weak solution,
there exists a closed set $S\subseteq \mathbb{R}^+$
of measure zero, such that
the solution is smooth on
$\mathbb{R}^3\times (\mathbb{R}^+\setminus S)$,
and the $1/2$ Hausdorff measure $\mathcal{H}^{1/2}(S)=0$.

Since then,
there has been a vast amount of literature on the uniqueness,
regularity and
global existence of solutions to NSE in wider spaces.
Until now, the uniqueness of Leray-Hopf solutions still remains
a challenging problem.

The scaling consideration usually suggests a heuristic way
to find suitable functional spaces for the solvability of partial differential equations.
It provides a useful classification of subcritical,
critical and supercritical spaces.
A general philosophy is that equations are well-posed
in the subcritical spaces,
while solutions may exhibit ill-posedness phenomena in the supercritical spaces.
We refer to the papers \cite{K00,K17} of Klainerman
for comprehensive discussions.
For the interested readers,
we refer to \cite{CCT03,KPV01,BT08,XZ22}
and the references therein
for the norm-inflation and discontinuity of solution map
for nonlinear Schr\"odinger equations, KdV equations
and nonlinear wave equations in supercritical spaces.

For the current hyperdissipative NSE \eqref{equa-NS},
it is invariant under the scaling
\begin{align}  \label{scaling-hyperNSE}
   u(t,x) \mapsto \lambda^{2\alpha-1} u (\lambda^{2\alpha}t, \lambda x),\ \
   P(t,x) \mapsto \lambda^{4\alpha-2} P (\lambda^{2\alpha}t, \lambda x).
\end{align}
This suggests the {\it critical space} $\mathbb{X}$ for \eqref{equa-NS}
if the corresponding norm of solutions is
invariant under the scaling \eqref{scaling-hyperNSE}.

One typical critical space is $\mathbb{X}=C_tL^2_x$
if $\alpha = 5/4$.
In the (sub)critical regime where $\alpha \geq 5/4$,
a remarkable result proved by Lions \cite{lions69} is that,
for any divergence-free $L^2$ initial data,
the hyperdissipative NSE
\eqref{equa-NS} admits unique  smooth Leray-Hopf solutions.
See also the global strong solvability by
Mattingly-Sinai \cite{MS99}.
The well-posedness of \eqref{equa-NS}
also holds for $\alpha$ slightly below $5/4$
due to Tao \cite{tao09}.
Moreover, Katz-Pavlovi\'c \cite{KP02} proved that
the Hausdorff dimension of the singular set at the time of first blow-up is at most $5-4\alpha$.

In contrast, for the supercritical regime where $\alpha <5/4$,
in the breakthrough work \cite{bv19b} Buckmaster-Vicol
first proved the non-uniqueness of finite energy weak solutions to NSE (i.e. $\alpha =1$),
based on the convex integration scheme.
The approach of convex integration was introduced to 3D Euler equations
in the pioneering papers by De Lellis and Sz\'ekelyhidi \cite{dls09, dls10}
and has been proven very successful
in the fluid community.
In particular, a recent milestone is the resolution of the Onsager conjecture,
developed in \cite{B15,bdis15,bdls16,dls14,dls13}
and finally settled by Isett \cite{I18} and Buckmaster-De Lellis-Sz\'ekelyhidi-Vicol \cite{bdsv19}.

The crucial ingredient introduced by Buckmaster-Vicol \cite{bv19b} is the $L^2_x$-based intermittent spatial building blocks,
which in particular permit to control the hard dissipativity term $(-\Delta)u$ in NSE.
By making the full use of the spatial intermittency,
Luo-Titi \cite{lt20} proved the non-uniqueness of weak solutions in $C_tL^2_x$
to hyperdissipative NSE \eqref{equa-NS},
whenever the exponent $\alpha$ is less than the Lions exponent,
i.e., $\alpha <5/4$.
Furthermore, in the recent work \cite{bcv21},
Buckmaster-Colombo-Vicol constructed the non-unique weak solutions to \eqref{equa-NS}
when $\alpha \in [1,5/4)$, which are smooth
outside a singular set in time with Hausdorff dimension less than one.
The intermittent convex integration also has been applied to various other models.
We refer, e.g., to \cite{lq20} for 2D hypoviscous NSE,
\cite{luo19} for stationary NSE,
and \cite{CDR18,DR19} for the non-uniqueness of Leray solutions to hypodissipative NSE.
See the surveys \cite{dls09,bv21,bv19r,dls17} for other interesting applications.
We also refer to another method by Jia and \v{S}ver\'ak \cite{js14,js15}
for the non-uniqueness of Leray-Hopf solutions under a certain assumption
for the linearized Navier-Stokes operator,
and the very recent work \cite{ABC21}
for the non-uniqueness of Leray solutions
of the forced NSE.

Hence, in view of the works \cite{lions69,lt20,bcv21},
$\alpha=5/4$ is exactly the {\it critical threshold of viscosity} for the well-posedness
in $C([0,T];L^2)$
for hyperdissipative NSE \eqref{equa-NS}.

Another type of critical spaces extensively used is the mixed Sobolev space
$\mathbb{X}=L^\gamma_t\dot{W}^{s,p}_x$,
where the exponents $(s,\gamma, p)$ satisfy
\begin{align}   \label{critical-LPS-hyperNSE}
    \frac{2\alpha}{\gamma} + \frac{3}{p} = 2\alpha -1 +s.
\end{align}
In the case where $\gamma =\infty$
we may also consider $\mathbb{X}=C_t\dot{W}^{\frac 3p+1-2\alpha,p}_x$.
In particular,
the mixed Lebesgue space $L^\gamma_tL^p_x$
(or, more generally, Strichartz space frequently used
for dispersive equations, like the Schr\"odinger equations
and wave equations) is critical for the classical NSE,
when the exponents $(\gamma, p)$ satisfy the well-known
{\it Lady\v{z}enskaja-Prodi-Serrin condition}
\begin{align} \label{critical-LPS-NSE}
	\frac 2 \gamma + \frac 3p =1.
\end{align}

Due to the weak-strong uniqueness (\cite{prodi59,serrin62,L67,SvW84}),
the extra integrability in (sub)critical spaces $L^\gamma_tL^p_x$
with $2/ \gamma +  3 / p \leq 1$, $p\in[3,\infty)$,
suffices to guarantee the uniqueness
in the class of Leray-Hopf solutions to NSE.
The regularity in the delicate endpoint case $L^\infty_tL^3_x$ was solved by
Escauriaza-Seregin-\v{S}ver\'ak  \cite{iss03}.
See also \cite{R02} for the weak-strong uniqueness
in $L^\gamma_tW^{s,p}_x$
when $(s,\gamma,p)$ satisfies \eqref{critical-LPS-hyperNSE}
with $\alpha=1$,
\cite{GIP03} for the case of critical Besov spaces,
and \cite{LR16} for quite general Prodi-Serrin uniqueness criterion.
Furthermore,
due to the works of \cite{FJR72,FLRT00,LM01},
any weak solution to NSE  (in the distributional sense,
see Definition \ref{Def-Weak-Sol} below)
in the (sub)critical spaces $L^\gamma_tL^p_x$
is automatically the unique regular Leray-Hopf solution,
see \cite[theorem 1.3]{cl20.2} for the precise statements on the torus.

There are also many uniqueness results for the hyperdissipative NSE,
under the generalized Lady\v{z}enskaja-Prodi-Serrin condition \eqref{critical-LPS-hyperNSE}
or in the critical spaces.
We refer to, for instance,
\cite{Z07} for the mixed space $L^\gamma_tL^p_x$
when $2\alpha/\gamma+3/p\leq 2\alpha-1$,
and \cite{W06} for the Besov space $\dot{B}^{1-2\alpha+\frac dp}_{p,q}$.

In contrast to the positive side,
many questions remain open in the supercritical regime.
In the recent remarkable paper \cite{cl20.2},
Cheskidov-Luo proved the sharp non-uniqueness of NSE
near the endpoint $(s,\gamma, p)=(0,2,\infty)$
of the Lady\v{z}enskaja-Prodi-Serrin condition \eqref{critical-LPS-NSE}.
The proof in particular exploits the temporal intermittency
in the convex integrations scheme.
See also \cite{cl21,cl22} for the application of temporal intermittency
to transport equations,
and \cite{lzz21} for the case of MHD equations.
The non-uniqueness in \cite{cl20.2} is indeed proved in the {\it strong} sense that
every weak solution is non-unique,
and the Hausdorff dimension of the corresponding singularity set
in time can be less than any given small constant $\ve>0$.
It is also conjectured by Cheskidov-Luo  \cite{cl20.2}
that the non-uniqueness of weak solutions
shall be valid in the full range of the supercritical regime
$2/\gamma + 3/p >1$.
More recent progress has been made in \cite{cl21.2}
for the other endpoint
$(s,\gamma, p)=(0,\infty, 2)$ for the 2D NSE.

It is worth noting that,
the endpoint case $(s,\gamma,p)=(3/p+1-2\alpha, \infty, p)$
corresponds exactly to the critical space $C_tL^2_x$ for
equation \eqref{equa-NS} when $\alpha =5/4$ and $p=2$.

The significance of the endpoint $(s,\gamma,p)=(3/p+1-2\alpha,\infty,p)$
can be also seen from
its close relationship to more general critical Besov and Triebel-Lizorkin spaces.
Specifically,
for the classical NSE when $\alpha =1$,
one has the embedding of critical spaces:
\begin{align} \label{embed-critical-NSE}
	L^3 \hookrightarrow \dot{B}^{\frac 3p-1}_{p | 2\leq p<\infty, \infty}
	\hookrightarrow BMO^{-1} (=\dot{F}^{-1}_{\infty,2})
	\hookrightarrow \dot{B}^{-1}_{\infty,\infty}.
\end{align}
The solvability of NSE in these critical spaces has
attracted significant interests in literature.
It is usually obtained by the mild formulation of equations,
dating back to Kato and Fujita \cite{FK64,K84}.
One well-known critical space is $BMO^{-1}$,
due to Koch and Tataru \cite{kt01}.
See the monographs \cite{C04,LR16,M99} for more details.

It was a long standing problem whether
NSE is well-posed in the largest critical space $\dot{B}^{1-2\alpha}_{\infty,\infty}$
(\cite{C04,M99}).
Quite surprisingly,
the negative answer was provided by
Bourgain-Pavlovi\'c \cite{BP08},
by showing a phenomenon of norm-inflation instability in $\dot{B}^{-1}_{\infty,\infty}$
for NSE.
Germin \cite{G08} also proved that the solution map associated to NSE is not $C^2$
in the space $\dot{B}^{-1}_{\infty,q}$, $q>2$.
Afterwards, the norm-inflation in $\dot{B}^{-1}_{\infty,q}$ for $q\geq 1$
was proved by Yoneda \cite{Y10} and Wang \cite{W15}.
The ill-posedness phenomena also exhibit for the hyperdissipative NSE.
There exist discontinuous Leray-Hopf solutions in the critical space
$B^{1-2\alpha}_{\infty, \infty}$
with $\alpha \in [1,5/4)$,
due to Cheskidov-Shvydkoy \cite{CS12}, with arbitrarily small initial data.
Cheskidov-Dai \cite{CD14} also proved the norm-inflation instability
in $\dot{B}^{-s}_{\infty,q}$
for all $s\geq \alpha\geq 5/4$, $q\in (2,\infty]$.

Enlightened by the above progresses,
we consider the following three non-uniqueness questions:
\begin{enumerate}
 \item[$\bullet $] In the highly dissipative regime $\alpha \geq 5/4$
  where the global solvability of Leray-Hopf solutions was known due to Lions \cite{lions69},
  would it be possible to find non-unique and non-Leray-Hopf weak solutions
  even with the same initial data of Leray-Hopf solutions ?
 \item[$\bullet$]
       As conjectured in the NSE context \cite{cl20.2},
	   in view of the generalized Lady\v{z}enskaja-Prodi-Serrin condition \eqref{critical-LPS-hyperNSE},
       do there exist non-unique weak solutions to the hyperdissipative NSE \eqref{equa-NS}
	   in the supercritical spaces  $L^\gamma_tW_x^{s,p}$,
	   where $
	   {2\alpha}/{\gamma} + {3}/{p} > 2\alpha -1 +s$ ?
\item[$\bullet$]
       In view of the positive well-posedness results in critical spaces,
	   e.g., \cite{K84,kt01,W06},
	   are there non-unique weak solutions
	   to hyperdissipative NSE in the
	   supercritical Lebesgue, Besov, or Triebel-Lizorkin spaces ?
\end{enumerate}

It is worth noting that,
the global solvability of Leray-Hopf solutions when $\alpha \geq 5/4$
makes it significantly hard to construct non-unique weak solutions
to \eqref{equa-NS}.
Actually, it is not possible to construct non-unique weak solutions
as in \cite{bv19b,lt20,bcv21}
in the space $C_tL^2_x$,
since any weak solution in $C_tL^2_x$
is the unique Leray-Hopf solution
due to \cite{lions69}.

In the present work,
we give the positive answers to the first and third questions,
and to the second question at two endpoints.
These results are contained in the main result, i.e., Theorem \ref{Thm-Non-hyper-NSE}
below, concerning the non-uniqueness result
for every weak solution in the space $L^\gamma_tW^{s,p}_x$,
where $(s,\gamma,p)$ lies in the supercritical regimes
$\mathcal{A}_1$ and $\mathcal{A}_2$,
respectively, for $\alpha\in [5/4,2)$ and $\alpha \in [1,2)$.
See \eqref{A-regularity1} and \eqref{A-regularity2} below for the precise
formulations of $\mathcal{A}_i$, $i=1,2$.

To the best of our knowledge,
it is the first non-uniqueness result for the hyperdissipative NSE,
when the viscosity exponent $\alpha$ is beyond the Lions exponent $5/4$.
In particular,
the non-uniqueness results
in $L^\gamma_tW^{s,p}_x$ hold in the strong sense as in \cite{cl20.2}
and are sharp at two endpoints
$(3/p+1-2\alpha, \infty, p)$
and $(2\alpha/\gamma+1-2\alpha, \gamma, \infty)$,
in view of the generalized Lady\v{z}enskaja-Prodi-Serrin condition \eqref{critical-LPS-hyperNSE}.

It also provides the non-unique weak solutions to \eqref{equa-NS}
in the spaces $C_t\mathbb{X}$,
where $\mathbb{X}$ can be the supercritical
Lebesgue, Besov and Triebel-Lizorkin spaces.
In particular, in view of the well-posedness results
in \cite{W06} and Theorem \ref{Thm-GWP-HNSE-Lp} below,
the non-uniqueness results are sharp in the
Lebesgue and Besov spaces.

Furthermore,
the delicate phenomenon exhibited here is that,
even in the hyperdissipative case $\alpha \geq 5/4$,
albeit the unique smooth Leray-Hopf solutions to \eqref{equa-NS},
there indeed exist weak solutions
in any small $L^\gamma_tW^{s,p}_x$-neighborhood of Leray-Hopf solutions,
which coincide with Leray-Hopf solutions near the initial time,
are smooth outside a null set in time,
and have the zero Hausdorff $\mathcal{H}^{\eta_*}$ measure of
the singular set,
where $\eta_*>0$ can be any given small constant.
This fine structure of the temporal singular set
is exploited by using the gluing technique,
which was first developed to solve the Onsager conjecture
\cite{B15,I18,bdsv19}
and has been recently implemented in the context of NSE \cite{bcv21,cl20.2}.

The last result of the present work is concerned with
the viscosity vanishing result.
Namely, given any weak solution to Euler equations \eqref{equa-Euler}
in the space $H^{\wt \beta}_{t,x}$,
where $\wt \beta >0$,
we show that
it is a strong vanishing viscosity limit in $H^{\wt \beta}_{t,x}$ of a sequence of weak solutions
to the hyperdissipative NSE \eqref{equa-NS} with $\alpha \in [1,2)$.
Hence, it extends the viscosity vanishing result in the NSE case
to the hyperdissipative NSE.

The construction of non-unique weak solutions,
inspired by the recent works \cite{bcv21,bv19b,cl20.2,cl21.2},
is based on the approach of intermittent convex integration,
which features both the spatial and temporal intermittency.
The fundamental spatial building blocks are the intermittent jets
for the endpoint case $(3/p+1-2\alpha, \infty, p)$,
and the concentrated Mikado flows for the other endpoint case
$(2\alpha/\gamma+1-2\alpha, \gamma, \infty)$.
In both cases,
the extra temporal intermittency  shall be exploited
in an almost optimal way,
in order to control the high dissipativity,
time derivative errors and
oscillation errors,
and simultaneously, to respect the supercritical regularity.
As we shall see below,
in the very high dissipativity regime
where $\alpha$ is close to $2$,
the suitable temporal intermittency roughly equals to 3D,
and respectively, 4D spatial intermittency
in the supercritical regimes $\mathcal{A}_1$ and $\mathcal{A}_2$.  \\

{\bf Notations.} To simplify the notations, for $p\in [1,\infty]$ and $s\in \R$, we denote
\begin{align*}
	L^p_t:=L^p(0,T),\quad L^p_x:=L^p(\T^3),\quad H^s_x:=H^s(\T^3), \quad W^{s,p}_x:=W^{s,p}(\T^3),
\end{align*}
where $W^{s,p}_x$ is the usual Sobolev space
and $H^s_x=W^{s,2}_x$.
Moreover, $L^\gamma_tL^p_x$ denotes the usual Banach space
$L^\gamma(0,T;L^p(\T^3))$,
$p, \gamma\in [1,\infty]$.
Let
\begin{align*}
	\norm{u}_{W^{N,p}_{t,x}}:=\sum_{0\leq m+|\zeta|\leq N} \norm{\p_t^m \na^{\zeta} u}_{L^p_{t,x}}, \ \
     \norm{u}_{C_{t,x}^N}:=\sum_{0\leq m+|\zeta|\leq N}
	  \norm{\p_t^m \na^{\zeta} u}_{C_{t,x}},
\end{align*}
where $\zeta=(\zeta_1,\zeta_2,\zeta_3)$ is the multi-index
and $\na^\zeta:= \partial_{x_1}^{\zeta_1} \partial_{x_2}^{\zeta_2} \partial_{x_3}^{\zeta_3}$.
In particular, we write $L^p_{t,x}:= L^p_tL^p_x$ for brevity.
Given any Banach space $X$,
$C([0,T];X)$ denotes  the space of continuous functions from $[0,T]$ to $X$,
equipped with the norm $\|u\|_{C_tX}:=\sup_{t\in [0,T]}\|u(t)\|_X$.

We also use the Besov space ${B}_{p, q}^{s} (\T^3)$
endowed with the norm
\begin{align*}
   \|f\|_{{B}_{p, q}^{s}(\T^3)}= ( \sum_{j \geq 0}\big|2^{js} \| \Delta_j f\|_{L^p(\T^3)}\big|^q )^{\frac1q},
\end{align*}
where $(s,p,q)\in (-\infty, \infty) \times [1,\infty] \times [1,\infty]$,
$\{\Delta_j\}_{j\in \mathbb{Z}}$ is the Littlewood-Paley decomposition
of the unity.
Let ${F}_{p, q}^{s}(\T^3)$ denote the Triebel-Lizorkin space,
endowed with the norm
\begin{align*}
   \|f \|_{{F}_{p, q}^{s}(\T^3)}
   =\|\big(\sum_{j\geq 0}| 2^{js} \Delta_{j}f|^q \big)^{\frac{1}{q}}\|_{L^{p}(\T^3)},
\end{align*}
where $(s,p,q)\in (-\infty, \infty) \times [1,\infty) \times [1,\infty]$.
The homogeneous Besov $\dot{B}_{p, q}^{s}(\T^3)$
and Triebel-Lizorkin spaces $\dot{F}_{p, q}^{s}(\T^3)$
are defined similarly where the summation is over $j\in \mathbb{Z}$.
We refer to \cite{ST87} for more details.

For any $A\subseteq [0,T]$, $\ve_*>0$,
the neighborhood of $A$ in $[0,T]$ is defined by
\begin{align*}
	N_{\va_*}(A):=\{t\in [0,T]:\ \exists\,s\in A,\ s.t.\ |t-s|\leq \va_*\}.
\end{align*}
We also use the notation $a\lesssim b$, which means that $a\leq C b$ for some constant $C>0$.

\subsection{Main results} \label{Subsec-Main}

Before formulating  the main results,
let us first present the notion of weak solutions
in the distributional sense to equation \eqref{equa-NS}.

\begin{definition} \label{Def-Weak-Sol} (Weak solutions)
Given any weakly divergence-free datum $u_0 \in L^2(\mathbb{T}^3)$,
we say that $u\in L^2([0,T]\times \mathbb{T}^3)$
is a weak solution for the hyperdissipative Navier-Stokes equations \eqref{equa-NS}
if $u$ is divergence-free for a.e. $t\in [0,T]$,
and
\begin{align*}
	\int\limits_{\mathbb{T}^3} u_0 \vf(0,x) dx
	= - \int_0^T \int\limits_{\mathbb{T}^3}
	  u(\partial_t \vf - \nu (-\Delta)^\alpha \vf + (u\cdot \na) \vf) dx dt
\end{align*}
for any divergence-free test function
$\vf\in C_0^\infty([0,T)\times\mathbb{T}^3)$.
\end{definition}

We focus on the following two supercritical regimes,
whose borderlines contain two endpoints
of the generalized Lady\v{z}enskaja-Prodi-Serrin condition \eqref{critical-LPS-hyperNSE}.
More precisely,
in the case $\a\in [5/4,2)$
we consider the supercritical regime $\mathcal{A}_1$ given by
		\begin{align} \label{A-regularity1}
		\mathcal{A}_1:=\bigg\{ (s,\gamma,p)\in [0,3)\times [1, \infty] \times[1,\infty]: 0\leq s< \frac{4\a-5}{\gamma}+ \frac{3}{p}+1-2\a \bigg\},
	\end{align}
and in the case $\a\in [1,2)$ we consider
supercritical regime $\mathcal{A}_2$ given by
	\begin{align} \label{A-regularity2}
		\mathcal{A}_2:=\bigg\{ (s,\gamma,p)\in [0,3)\times [1, \infty]\times[1,\infty]: 0\leq s<  \frac{2\a}{\gamma}+\frac{2\a-2}{p}+1-2\a \bigg\}.
	\end{align}

The supercritical regimes $\mathcal{A}_1$ and $\mathcal{A}_2$ in the case $s=0$
can be seen in Figure $1$ below.
\begin{figure}[H]
	\centering
	\begin{tikzpicture}[scale=4,>=stealth,align=center]
	%\footnotesize
	\small
	
	%\node[circle, draw=red, line width=2pt, dashed, draw opacity=0.5] (a) at (0,0){A};
	
	\draw[->] (0,0) --node[pos=1,left] {$\frac{1}{p}$} node[pos=0,below]{$0$} (0,1.35);
	\draw[->] (0,0) --node[pos=1,below] {$\frac{1}{\gamma}$} (1.35,0);
	%\draw[pattern = north east lines, pattern color = red!25] (0,0) -- (0,4/5) -- (4/3,0) -- cycle;
	%\draw[blue,thick] (0,4/5) -- (4/3,0);
	%\node at (1/3,1/5) {Uniqueness \\ by LPS};
	%\node[name = A] at (1/2,1/2) {$(\frac{1}{2}, \frac{1}{2})$};
	
	%\draw (0,1) --node[pos=0,left] {$1$} (1,1);
	%\draw (1,0) --node[pos=0,below] {$1$} (1,1);	
	%\draw (0,0.7) --node[pos=0,left] {$\frac{2\alpha -1}{3} $} node[pos=1,below] {$\frac{2\alpha-1}{2\alpha}$} (0.7,0);
	
	\coordinate [label=below:$1$] (gamma1) at (1,0);
	\coordinate [label=left:$1$] (p1) at (0,1);
	\coordinate (11) at (1,1);
	%\draw [name path = l1] (gamma1) -- (11);
	%\draw [name path = l2] (p1) -- (11);

	\coordinate [label=left:$\frac{2\alpha -1}{3} $] (A) at (0,6/7); % (1.2-0.5) / 0.5 = 0.5 / ?, ? = 0.25/0.7 = 5/14, 5/14+ 1/2 = 6/7
	\coordinate [label=below:$\frac{2\alpha-1}{2\alpha}$] (B) at (6/7,0);
	\draw (A) --node[pos=0.5,below left]{$\;\nearrow$\\gLPS\;} (B);
	\coordinate [label=below:$ \frac{2\alpha-1}{4\alpha-5} $] (A1) at (1.2,0);
	\coordinate [label=left:$ \frac{2\alpha-1}{2\alpha-2} $] (B1) at (0,1.2);

	%\path [name intersections = {of= l1 and l3}] coordinate  (A2) at (intersection-1);
	%\path [name intersections = {of= l2 and l4}] coordinate  (B2) at (intersection-1);
	
	\draw[white, pattern =north west lines, pattern color = blue!50] (A) -- (p1) -- (11) -- (gamma1) -- (B) -- (0.5,0.5) -- cycle;
	%\draw[fill, color = blue!75!green!25!white] (A) -- (p1) -- (11) -- (gamma1) -- (B) -- (0.5,0.5) -- cycle;

	\draw[dashed, name path = l3] (A) --node[pos=0.65,above right] { \; $\mathcal{A}_1$\\ $\swarrow$ \;} (A1);
	\draw[dashed, name path = l4] (B) --node[pos=0.65,above right] { \; $\mathcal{A}_2$\\ $\swarrow$ \;} (B1);
	\draw (A) -- (p1) -- (11) -- (gamma1) -- (B);
	\fill(0.5,0.5) circle (0.5pt) node[above right] {$(\frac{1}{2}, \frac{1}{2})$} ;
	
	\coordinate [label=left:{Endpoint space \\ $C_t L_x^{\frac{3}{2\alpha-1}} \rightarrow $}  ] (M1) at (-0.2,13/14);
	\coordinate [label=below:{$ \uparrow $ \\ Endpoint space  $L_t^{\frac{2\alpha}{2\alpha-1}}L^\infty_x $}  ] (M1) at (6/7,-0.1);
		
\end{tikzpicture}
\caption{The case $\alpha \in [\frac{5}{4},2), s=0$}
\label{fig:1}
\end{figure}
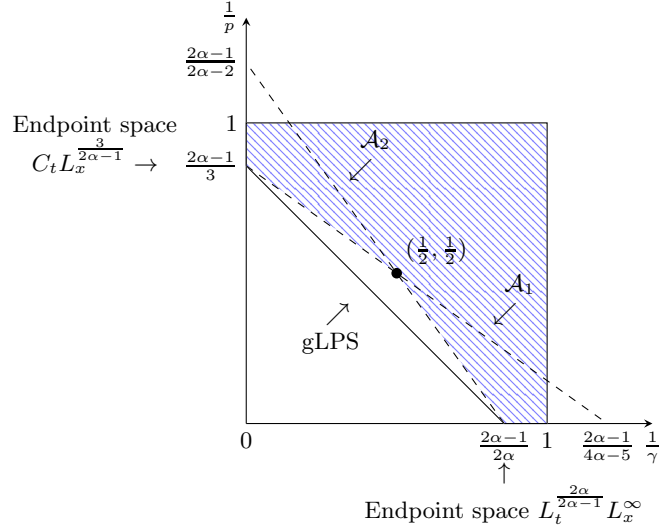

The main result of this paper is formulated in Theorem \ref{Thm-Non-hyper-NSE} below,
which in particular gives the non-uniqueness in the
shaded part in Figure $1$,
including any small neighborhood near two endpoints.

\begin{theorem} \label{Thm-Non-hyper-NSE}
Let $\tilde{u}$ be any smooth, divergence-free and mean-free vector field on $[0,T]\times \T^3$.
Then, there exists $\beta'\in(0,1)$,
such that for any $\va_*, \eta_*>0$
and for any $(s,p,\gamma)\in \mathcal{A}_1$ or $(s,p,\gamma)\in \mathcal{A}_2$,
respectively, if $\alpha \in [5/4,2)$ or $\alpha\in[1,2)$,
there exist a velocity field $u$
and a set
\begin{align*}
	\mathcal{G} = \bigcup\limits_{i=1}^\infty (a_i,b_i) \in [0,T],
\end{align*}
such that the following hold:
\begin{enumerate}[(i)]
	\item Weak solution: $u$ is a weak solution
		to \eqref{equa-NS} with the initial datum $\wt u(0)$
		and has zero spatial mean.
	\item Regularity: $u \in H^{\beta'}_{t,x} \cap L^\gamma_tW^{s,p}_x$,
	and
	\begin{align*}
		u|_{\mathcal{G}\times \mathbb{T}^3} \in C^\infty (\mathcal{G}\times \mathbb{T}^3).
	\end{align*}
	Moreover,
	if there exists $t_0\in (0,T)$ such that
	$\wt u$ is the solution to \eqref{equa-NS} on $[0,t_0]$,
	then $u$ agrees with $\wt u$ near $t=0$.
	\item The Hausdorff dimension of the singular set
	$\mathcal{B} = [0,T]/\mathcal{G}$ satisfies
	\begin{align*}
		   d_{\mathcal{H}}(\mathcal{B}) <\eta_*.
	\end{align*}
	In particular, the singular set $\mathcal{B}$
	has zero Hausdorff $\mathcal{H}^{\eta_*}$ measure,
	i.e.,  $\mathcal{H}^{\eta_*}(\mathcal{B})=0$.
	\item Small deviations of temporal support:
	 $$\supp_t u  \subseteq N_{\va_*}(\supp_t \tilde{u}).$$
	\item Small deviations on average:
	$$\|u-\tilde{u}\|_{L^1_tL^2_x}+\|u-\tilde{u}\|_{L^\gamma_tW^{s,p}_x}\leq \va_*.$$
\end{enumerate}
\end{theorem}

The first direct consequence of Theorem \ref{Thm-Non-hyper-NSE}
is the following strong non-uniqueness of weak solutions to \eqref{equa-NS}
in the hyperdissipative case where $\alpha \in [5/4,2)$.

\begin{corollary} \label{Cor-Strong-Nonuniq}
(Strong non-uniqueness)
Let $\alpha\in [5/4,2)$.
Then, for any weak solution $\wt u$ to \eqref{equa-NS},
there exists a different weak solution $u \in L^\gamma_tW^{s,p}_x$ to \eqref{equa-NS}
with the same initial data,
where $(s,\gamma,p) \in \mathcal{A}_1\cup \mathcal{A}_2$.

Moreover, for every divergence-free $L^2_x$ initial data,
there exist infinitely many weak solutions in $L^\gamma_{t}W^{s,p}_x$
to \eqref{equa-NS}
which are smooth almost everywhere in time.
\end{corollary}

Another interesting consequence of Theorem \ref{Thm-Non-hyper-NSE}
is the non-uniqueness in the supercritical Lebesgue,
Besov and Triebel-Lizorkin spaces.

\begin{corollary} \label{Cor-Nonuniq-Supercri}
	(Non-uniqueness in supercritical spaces)
 Let $\alpha \in [5/4,2)$.
 Then, there exist non-unique weak solutions to \eqref{equa-NS} in the supercritical spaces $C_t\mathbb{X}$,
 where $\mathbb{X}$ can be one of the following three types of spaces:
\begin{enumerate}
	\item[(i)]  $L^p$, $1\leq p<\frac{3}{2\alpha-1}$;
	\item[(ii)] ${B}^{s}_{p,q}$, $-\infty<s<\frac 3p +1-2\alpha$,
                $1< p<\infty$, $1\leq q\leq \infty$;
    \item[(iii)] ${F}^{s}_{p,q}$, $-\infty<s<\frac 3p +1-2\alpha$,
	            $1< p<\infty$, $1\leq q\leq \infty$.
\end{enumerate}
\end{corollary}

At last, we have the vanishing viscosity result
which extends the corresponding result for NSE in \cite{bv19b} to the hyperdissipative NSE \eqref{equa-NS}.

\begin{theorem} \label{Thm-hyperNSE-Euler-limit}
	(Strong vanishing viscosity limit)
	Let $\alpha \in (1,2)$
	and $u\in H^{\wt{\beta}}_{t,x}([-2T,2T]\times \T^3)$ be any mean-free weak solution
	to the Euler equation \eqref{equa-Euler},
	where $\wt \beta >0$.
	Then, there exist $\beta' \in (0, \wt \beta)$ and a sequence of weak solutions
	$u^{(\nu_{n})}\in H^{\beta'}_{t,x} $
	to \eqref{equa-NS},
    where  $\nu_{n}$
    is the viscosity coefficient,
	such that
	as $\nu_{n}\rightarrow 0$,
	\begin{align}\label{convergence}
		u^{(\nu_{n})}\rightarrow u \quad\text{strongly in}\  H^{\beta'}_{t,x}.
	\end{align}
\end{theorem}

\subsection{Comments on main results.}

In the following let us present some comments on the main results.

{\bf (i) Strong non-uniqueness for the high dissipativity above the Lions exponent.}
It is folklore that one has global solvability in the high dissipative case
when $\alpha \geq 5/4$.
Actually, in view of the works  \cite{lions69,lt20,bcv21},
$\alpha =5/4$ is the critical threshold for the well-posedness of
solutions in $C_tL^2_x$ to \eqref{equa-NS}.
That is, weak solutions in $C_tL^2_x$
are unique if $\alpha \geq 5/4$, while non-unique if $\alpha <5/4$.

Quite surprisingly,
Theorem \ref{Thm-Non-hyper-NSE} shows that,
even in the high dissipative regime $\alpha \geq 5/4$,
the uniqueness fails in the spaces $L^\gamma_tW^{s,p}_x$
where $(s,\gamma,p)$ lies in the supercritical regimes
$\mathcal{A}_1 \cup \mathcal{A}_2$, defined in \eqref{A-regularity1} and \eqref{A-regularity2}, respectively.
The non-uniqueness even exhibits in the strong sense that,
any solution in $L^\gamma_t W^{s,p}_x$
is non-unique.

In particular, in the case where $\alpha =5/4$,
Corollary \ref{Cor-Nonuniq-Supercri} $(i)$ yields
the non-uniqueness of weak solutions in $C_tL^p_x$ to \eqref{equa-NS}
for any $p<2$,.
Thus, in view of the well-posedness results \cite{lions69},
the non-uniqueness of Corollary \ref{Cor-Nonuniq-Supercri} $(i)$ is sharp in $C_tL^2_x$.

{\bf (ii) Sharp non-uniqueness at two endpoints of generalized Lady\v{z}enskaja-Prodi-Serrin condition.}
In the remarkable paper \cite{cl20.2},
Cheskidov-Luo first proved the sharp non-uniqueness at the endpoint
$(s,\gamma,p)=(0,2,\infty)$.
That is,
for any $\gamma<2$, there exist non-unique solutions
in $L^\gamma_tL^\infty_x$ to NSE ($\alpha=1$)
in all dimensions $d\geq 2$.
It is also conjectured that the non-uniqueness shall be valid
in the whole supercritical regime determined by the
Lady\v{z}enskaja-Prodi-Serrin condition \eqref{critical-LPS-NSE}.
The non-uniqueness for the other endpoint case $(s,\gamma,p)=(0,\infty,2)$
has been recently achieved in \cite{cl21.2} for the 2D NSE.

In view of the generalized Lady\v{z}enskaja-Prodi-Serrin condition \eqref{critical-LPS-hyperNSE}
and the well-posedness in the (sub) critical cases \cite{Z07},
Theorem \ref{Thm-Non-hyper-NSE} provides the sharp non-uniqueness
for the hyperdissipative NSE \eqref{equa-NS}
at two endpoints, i.e.,
$(3/p+1-2\alpha,\infty,p)$ for $\alpha \in [5/4,2)$,
and $(2\alpha/\gamma+1-2\alpha,\gamma,\infty)$ for $\alpha \in (1,2)$.
This in particular extends the results in
\cite{cl21.2} and \cite{cl20.2},
respectively, to the 3D hyperdissipative NSE where
$\alpha\in [5/4,2)$ and $\alpha \in (1,2)$.

We would expect the non-uniqueness
for the remaining supercritical regimes when $\alpha\in[5/4,2)$,
and for the supercritical regime near the
endpoint $(3/p+1-2\alpha, \infty, p)$ when $\alpha \in [1,5/4)$.
This seems out of the reach of present method,
due to the $L^2_{t,x}$-criticality of space-time convex integration method.
As a matter of fact, as pointed out in \cite{cl20.2,cl21.2},
the temporal intermittency allows to
raise the temporal integrability exponent $\gamma >2$,
yet at the cost of reducing the spatial integrability exponent $p<2$.
We note that, in the endpoint case $(0,\infty, {3}/{(2\alpha-1)})$
when $\alpha <5/4$,
both the temporal and spatial integrability exponents would be larger than two.

{\bf (iii) Sharp non-uniqueness in the supercritical Lebesgue and Besov spaces.}
There have been extensive results on the well-posedness
in the (sub)critical spaces,
see, e.g. \cite{C04,LR16,M99} for the NSE in the critical spaces:
\begin{align} \label{critical-space-NSE}
	L^3
	\hookrightarrow \dot{B}^{\frac 3p-1}_{p|2\leq p<\infty,\infty}
	\hookrightarrow BMO^{-1} (=\dot{F}^{-1}_{\infty,2}).
\end{align}

Corollary \ref{Cor-Nonuniq-Supercri} appears to be the first non-uniqueness
result for hyperdissipative NSE \eqref{equa-NS}
in the space $C_t\mathbb{X}$,
where $\mathbb{X}$ can be the supercritical
Lebesgue, Besov and Triebel-Lizorkin spaces.

Let us mention that,
the well-posedness of NSE in the critical space $L^3_x$
was proved in the famous paper by Kato \cite{K84}.
The mild formulation strategy proposed in \cite{K84}
has been now frequently used to obtain the well-posedness
of NSE in various spaces.
For the hyperdissipative NSE \eqref{equa-NS},
we include the well-posedness result in the critical space
$L^{3/(2\alpha-1)}_x$ in the Appendix.
In particular,
this shows that the non-uniqueness in Corollary \ref{Cor-Nonuniq-Supercri} $(i)$
is sharp in the supercritical Lebesgue spaces.

Moreover, it has been proved in \cite{W06} that,
for $\alpha >1/2$, \eqref{equa-NS} is well-posed with small data
in $\dot{B}^{\frac 52-2\alpha}_{2,q}(\mathbb{R}^d)$
for $1<q\leq \infty$.
The proof also applies to the torus case.
Hence,
the non-uniqueness in Corollary \ref{Cor-Nonuniq-Supercri} $(ii)$
is sharp in the Besov spaces.

Furthermore,
for any $s<1-2\alpha$,
we may take $\eta>0$ small enough
such that $1-2\alpha -\eta>s$.
Then, by the embedding of Besov spaces
we have for any $1\leq q\leq \infty$,
\begin{align*}
    {B}^{\frac 3p+1-2\alpha-\eta}_{p,q}
	\hookrightarrow {B}^{1-2\alpha-\eta}_{p,q}
	\hookrightarrow {B}^{s}_{\9,q}. 
\end{align*}
Hence, by virtue of Corollary \ref{Cor-Nonuniq-Supercri} $(ii)$,
we also have the non-uniqueness of weak solutions in
${B}^{s}_{\9,q}$,
for any $s<1-2\alpha$, $1\leq q\leq \infty$.
This may also be seen as a complement
to the ill-posedness results in \cite{CD14},
where the norm-inflation instability was proved  for equation \eqref{equa-NS}
with $\alpha \geq 5/4$
in the Besov spaces ${B}^{s}_{\infty,q}$,
for any $s\leq -\alpha$, $2<q\leq \infty$.

{\bf (iv) Partial regularity of weak solutions.}
In the pioneering paper \cite{leray1934},
Leray proves that
the Leray-Hopf solutions to NSE
are smooth outside a closed singular set of times,
which has zero Hausdorff $\mathcal{H}^{1/2}$ measure.
This provides another possible way to tackle the global existence problem.
In particular,
following the works of Scheffer \cite{S76,S77},
Caffarelli-Kohn-Nirenberg \cite{CKN82}
proved a space-time regularity version
and showed the existence of global Leray-Hopf solutions
which have singular sets in $\mathbb{R}^3\times \mathbb{R}^+$
of zero Hausdorff $\mathcal{H}^1$ measure.
See also the simplified proofs in \cite{Lin98,Vasseur07}.
For hyperdissipative NSE with $\alpha \in (1,5/4]$,
Katz-Pavlovi\'c \cite{KP02} proved that
the Hausdorff dimension of the singular set at the time of first blow-up is at most $5-4\alpha$.
Recently, Colombo-De Lelli-Massaccesi \cite{CDM20}
proved a stronger version of the Katz-Pavlovi\'c result,
and showed the existence of Leray-Hopf solutions
which have singular space-time sets
of zero Hausdorff $\mathcal{H}^{5-4\alpha}$ measure,
thus extending the Caffarelli-Kohn-Nirenberg theorem
to hyperdissipative NSE.

Theorem \ref{Thm-Non-hyper-NSE} shows that,
in the high dissipative regime $\alpha \in [5/4,1)$,
for any small $\eta_*>0$,
there exist weak solutions to \eqref{equa-NS}
in any small $L^\gamma_tW^{s,p}_x$-neighborhood of Leray-Hopf solutions,
$(s,\gamma,p)\in \mathcal{A}_1\cup\mathcal{A}_2$,
which coincide with the Leray-Hopf solutions near $t=0$,
and have singular sets of times with
zero Hausdorff $\mathcal{H}^{\eta_*}$ measure.

The proof of partial regularity in time takes advantage of the gluing technique,
which was developed in \cite{B15,bdls16,I18,bdsv19} to solve the famous Onsager conjecture for 3D Euler equations.
The gluing technique to singular set of weak solutions to NSE was first
implemented in \cite{bcv21},
where the Hausdorff dimension of the constructed solution is
strictly less than one.
The results of \cite{bcv21} also
imply the strong uniqueness of weak solutions for NSE in dimensions $d=3,4$.
This technique was later used by Cheskidov-Luo \cite{cl20.2}
to obtain the strong uniqueness of weak solutions in the endpoint case $(s,\gamma,p)=(0,2,\infty)$,
which have small Hausdorff dimension of the singular sets in time.
It has been also used in
\cite{CDR18,DR19} to prove the non-unique Leray-Hopf solutions
for hypodissipative NSE when $\alpha <1/3$.

{\bf (v) Non-uniqueness for MHD equations above the Lions exponent.}
The non-uniqueness problem for
magnetohydrodynamic equations (MHD for short)
has attracted increasing interests in recent years.
We refer to \cite{bbv20,fls21,fls21.2} for the recent progresses for the ideal MHD
and the relationship to the Taylor conjecture.

One delicate point here is that,
as pointed in \cite{bbv20},
the geometry of MHD equations restricts the oscillation directions
and so limits the spatial intermittency.
Hence, it is hard to control the viscosity and resistivity
of MHD equations when the exponent is larger than one.

In the recent work \cite{lzz21},
the non-uniqueness has been proved for MHD equations,
where the viscosity and resistivity exponents are
allowed to be larger than one, yet below the Lions exponent $5/4$,
based on the construction of building blocks
which are adapted to the geometry of MHD
and feature both the temporal and spatial intermittency.

We would expect that the refined building blocks and parameters in this paper
permit to obtain the non-uniqueness for MHD above the Lions exponent $5/4$.
The strong non-uniqueness with fine smoothness outside a small
fractal set in time  would also be expected.

\section{Outline of the proof}

Our proof is mainly inspired by the intermittent convex integration method
developed in \cite{bcv21,bv19b,bv19r,cl21.2,cl20.2}.
It is based on the iterative construction of approximate solutions to the
hyperdissipative Navier-Stokes-Reynolds system, namely,
for each integer $q\geq 0$,
\begin{equation}\label{equa-nsr}
	\left\{\aligned
	&\p_t \u+\nu(-\Delta)^{\alpha} \u+ \div(\u\otimes\u)+\nabla P_q=\div \rr_q,  \\
	&\div \u = 0,
	\endaligned
	\right.
\end{equation}
where the Reynolds stress $\ru$ is a symmetric traceless $3\times 3$ matrix.

In order to exploit the fine temporal singular set of approximate solutions,
we adapt the notion of the well-prepared solutions from \cite{cl20.2} here.

\begin{definition} (Well-preparedness)
 Let $\eta\in (0,\eta_*)$. We say that
 the smooth solution $(u_q,\rr_q)$ to \eqref{equa-nsr} on $[0,T]$ is well-prepared
 if there exist a set $I$ and a length scale $\theta>0$,
 such that $I$ is a union of at most $\theta^{-\eta}$ many closed intervals of length scale $5\theta$ and
  \begin{align*}
\rr_q(t,x)=0 \quad \text{if} \quad \operatorname{dist}(t,I^c)\leq \theta.
  \end{align*}
\end{definition}

Two important quantities to measure the size of the relaxation solutions $(u_q, \mathring{R}_q)$,
$q\in \mathbb{N}$,
are the frequency parameter $\lbb_q$ and the amplitude parameter $\delta_{q+2}$:
\begin{equation}\label{la}
	\la=a^{(b^q)}, \ \
	\delta_{q+2}=\lambda_{q+2}^{-2\beta}.
\end{equation}
Here $a\in \mathbb{N}$ is a large integer to be determined later,
$\beta>0$ is the regularity parameter,
$b\in 2\mathbb{N}$ is a large integer of multiple $2$ such that
\begin{align}
   b>\frac{1000}{\varepsilon\eta_*}, \ \
   0<\beta<\frac{1}{100b^2},  \label{b-beta-ve}
\end{align}
where for the given $(s,p,\gamma)\in \mathcal{A}_1$, $\varepsilon\in \mathbb{Q}_+$ is sufficiently small such that
\begin{equation}\label{e3.1}
	\varepsilon\leq\frac{1}{20}\min\{2-\alpha,\,\frac{4\a-5}{\gamma}+\frac{3}{p}-(2\a-1)-s \}\quad \text{and}\quad b\ve\in\mathbb{N},
\end{equation}
and for the given $(s,p,\gamma)\in \mathcal{A}_2$, $\varepsilon>0$ is sufficiently small such that
\begin{equation}\label{ne3.1}
	\varepsilon\leq\frac{1}{20}\min\{2-\alpha,\,\frac{2\alpha}{\gamma}+\frac{2\a-2}{p}-(2\a-1)-s \}\quad \text{and}\quad b(2-\a-8\ve)\in \mathbb{N}.
\end{equation}

The idea is then to prove the vanishing of Reynolds stress in an appropriate space
as $q$ tends to infinity.
Thus, intuitively,
the limit of $u_q$ is expected to solve the original equation \eqref{equa-NS}.
This procedure is quantified in the following iterative estimates:
\begin{align}
    & \|\u\|_{L^\9_tH^3_x} \lesssim  \lambda_{q}^{5}, \label{uh3} \\
    & \|\p_t \u\|_{L^\9_tH^2_x} \lesssim  \lambda_{q}^{8}, \label{upth2} \\
	& \|\rr_q\|_{L^\9_tH^3_x} \lesssim   \lambda_{q}^{9},\label{rh3}	\\
	& \|\rr_q\|_{L^\9_tH^4_x} \lesssim   \lambda_{q}^{10},\label{rh4}	\\
& \|\rr_q\|_{L^{1}_{t,x}} \leq  \la^{-\ve_R}\delta_{q+1}, \label{rl1}
\end{align}
where the implicit constants are independent of $q$
and $\ve_R>0$ is a small parameter such that
\begin{align*}
	\ve_R< \frac{\ve}{10}.
\end{align*}

\begin{remark}
We note that,
the approximate solutions $(u_q, \rr_q)$
are measured in the more regular spaces $L^\infty_tH^N_x$,
$N=2,3,4$, and have larger frequency upper bounds
than those in \cite{lzz21}.
This is in part due to the full oscillation and concentration
in space and time,
in order to achieve the sharp non-uniqueness in the endpoints cases.
Moreover,
it is also imposed here to be
compatible with the gluing stage in Section \ref{Sec-Concen-Rey},
in order to exploit the fine singular set of times.
\end{remark}

The crucial iteration results of the relaxation solutions $(u_{q}, \mathring{R}_{q})$
are formulated below.

\begin{theorem} [Main iteration]\label{Prop-Iterat}
Let $(s,p,\gamma)\in \mathcal{A}_1$ for $\alpha \in [5/4,2)$,
or $(s,p,\gamma) \in \mathcal{A}_2$ for $\alpha \in [1,2)$.
Then, there exist $\beta\in (0,1)$,
 $M^*>0$ large enough and $a_0=a_0(\beta, M^*)$,
 such that for any integer $a\geq a_0$,
the following holds:

Suppose that
$(\u, \ru )$ is a well-prepared solution to \eqref{equa-nsr}
for the set $I_q$ and the length scale $\theta_q$
	and satisfies  \eqref{uh3}-\eqref{rl1}.
	Then, there exists another well-prepared solution $(u_{q+1}, \mathring{R}_{q+1} )$
	to \eqref{equa-nsr} for some set $I_{q+1}\subseteq I_q$,
	$0,T\notin I_{q+1}$, and the length scale $\theta_{q+1}<\theta_q/2$,
	and $(u_{q+1}, \mathring{R}_{q+1} )$ satisfies \eqref{uh3}-\eqref{rl1} with $q+1$ replacing $q$.
	In addition, we have
	\begin{align}
		&\|u_{q+1}-u_{q}\|_{L^{2}_{t,x}} \leq  M^*\delta_{q+1}^{\frac{1}{2}}, \label{u-B-L2tx-conv}\\
        &\|u_{q+1}-u_{q}\|_{L^1_tL^{2}_{x}} \leq  \delta_{q+2}^{\frac{1}{2}},  \label{u-B-L1L2-conv}\\
		&\norm{ u_{q+1} - u_q }_{L^\gamma_tW^{s,p}_x} \leq  \delta_{q+2}^{\frac{1}{2}},\label{u-B-Lw-conv}
	\end{align}
and
\begin{align}
&\supp_t (u_{q+1}, \rr_{q+1})
\subseteq N_{\delta_{q+2}^{\frac12}}( \supp_t  (u_{q}, \rr_{q})).\label{suppru}
\end{align}
\end{theorem}

The proof of the main iteration theorem will occupy most parts of the present paper.
It relies crucially on the
gluing procedure and the approach of space-time intermittent convex integration.

\subsection{Gluing stage}

The first stage is to concentrate the Reynolds stress
into a smaller region,
which eventually enables us to concentrate
the singular times into a null set with small Hausdorff dimension.

More precisely, given
a well-prepared solution $(\u,\rr_q)$ to \eqref{equa-nsr} at level $q$,
we divide the $[0,T]$ into $\mq$ many sub-intervals $[t_i,t_{i+1}]$ with length $ T/\mq$,
where $t_i:= i T/\mq $ and $\mq$ may depend on $(\u,\rr_q)$.
Then, we solve the following hyperdissipative Navier-Stokes equations
on each small interval $[t_i,t_{i+1}+\thq]$:
\begin{equation}\label{equa-nsvi}
	\left\{\aligned
	& \p_tv_i +\nu(-\Delta)^{\alpha} v_i+(v_i\cdot \nabla )v_i + \nabla P_i=0,  \\
	& \div v_i = 0,\\
    & v_i|_{t=t_i}=\u(t_i),
	\endaligned
	\right.
\end{equation}
where $\thq=(T/\mq)^{1/\eta}$,
$\eta$ satisfies that
$\eta \simeq \eta^*$
with $\eta_*$ as in Theorem~\ref{Thm-Non-hyper-NSE}
and
\begin{align}\label{ne2.14}
\thq^{-30} \simeq m_{q+1}^{\frac{30}{\eta}}\ll \laq^{\ve},
\end{align}
where $\ve$ is as in \eqref{e3.1}-\eqref{ne3.1}.
In particular, by Lemmas \ref{mae-endpt1} and \ref{mae-endpt2},
\eqref{ne2.14} indicates that
the amplitudes of velocity perturbations
oscillate with much weaker frequency
than those of temporal and spatial building blocks.

Due to the classical local well-posedness theory,
there exists a unique local smooth solution $v_i$ to \eqref{equa-nsvi}
if $\theta_{q+1}$ is small enough (or, $\mq$ is sufficiently large).
Then, in order to construct a global approximate solution to \eqref{equa-nsr}
and concentrate the Reynolds stress into smaller subintervals,
we glue these local solutions $v_i$ together
with a partition of unity $\{\chi_i\}_i$
(see \eqref{def-chi1}-\eqref{def-chi3} below):
\[
\wt u_q:=\sum_{i=0}^{\mq-1} \chi_i v_i,
\]
which satisfies the equation
\begin{align}\label{equa-wtu}
\partial_{t} \wt u_q+ \nu(-\Delta)^\a \wt u_q +\operatorname{div}(\wt u_q \otimes \wt u_q) +\nabla \wt p=\div \tr_q,
\end{align}
for some pressure $\wt p$,
and the new Reynolds stress is of form
\begin{align}\label{def-nr}
\tr_q =\partial_t\chi_i\mathcal{R}(v_{i}-v_{i-1}) -\chi_{i}(1-\chi_{i})((v_{i}-v_{i-1})\mathring\otimes (v_{i}-v_{i-1})),
\ \ t\in [t_i, t_{i+1}],
\end{align}
where  $\mathcal{R}$ is the inverse-divergence operator
given by \eqref{calR-def},
$0\leq i\leq m_{q+1}-1$
(we let $v_{-1}\equiv 0$).

Note that,
by the definition of $\{\chi_i\}_i$,
the new Reynolds stress $\tr_q$ is supported
on a $\thq$-neighborhood of $t_i$ for each $0\leq i<\mq$.
This yields that the singular set of $\wt u_q$
can be covered by $\mq$ many small intervals of length scale $\thq$,
thereby having the small Hausdorff dimension.
Moreover, as pointed out in \cite{bcv21},
since $\u$ is already a smooth solution to \eqref{equa-NS} on a majority of $[0,T]$,
namely the $\theta_q$-neighborhood of
the complement of some small set $I_q^c$,
if $t_{i-1}$ and $t_i$ both lie in this region,
one has $\wt u_q=v_{i-1}=v_i=\u$ on $\supp (\chi_i\chi_{i-1})$.
Thus, we can define an index set $\mathcal{C}$
(see \eqref{def-indexsetb} below) to
extract those regions where $\wt u_q$ is not necessarily
an exact solution to \eqref{equa-NS},
such that the bad set $I_{q+1}$ of $(\wt u_q, \tr_q)$
is contained in $I_q$.

Last but not least,
due to the stability estimates of the local solutions to \eqref{equa-nsvi},
the new Reynolds stress shares almost the same decay rate
with the old one in the $L^1_{t,x}$ space
(see \eqref{nrl1} below).
In other words, the procedure of concentrating the Reynolds stress error
only costs a loss of $\la^{-3\ve_R/4}$ decay rate,
which is acceptable in the next stage of convex integration (see also \cite[p.8]{cl20.2}).

\subsection{Space-time convex integration stage}

The next stage is to construct the key velocity perturbations,
particularly to fulfill the interactive objectives,
i.e., estimates \eqref{uh3}-\eqref{rl1}.
We will treat the two supercritical regimes $\mathcal{A}_1$
and $\mathcal{A}_2$ separately.

{\bf $\bullet$ Endpoint case $(s,\gamma, p)=(3/p+1-2\alpha,\infty, p)$.}
For the endpoint $(s,\gamma, p)=(3/p+1-2\alpha,\infty, p)$,
$\alpha \in[5/4,2)$,
we choose the intermittent jets $W_{(k)}$,
developed in \cite{bcv21},
as the main spatial building blocks (see \eqref{snwd} below).
The intermittent jets are indexed by four parameters $(\rs,\rp,\lambda,\mu)$,
where $\rs$ and $\rp$ parameterize the concentration of the flows,
$\lambda$ is the frequency parameter,
and $\mu$ is the temporal oscillation parameter.
One main feature of the intermittent jet is the almost 3D intermittency, i.e.,
\begin{align}
\|W_{(k)}\|_{L^\9_tL_x^1}\lesssim \lambda^{-\frac32+},
\end{align}
which succeeds in controlling the dissipativity
$(-\Delta)^\alpha$ when $\alpha \leq 5/4$,
see \cite{bcv21}.

In order to control the high viscosity $(-\Delta)^\alpha$
when $\alpha$ is beyond the Lions exponent $5/4$,
we need to oscillate $W_{(k)}$ in time by using the temporal concentration functions
(see \eqref{gk} below),
which are indexed by two parameters $(\tau, \sigma)$
and provide the additional intermittency.

The crucial constrains to run the convex integration mechanism
are listed in the following:
\begin{subequations}\label{constset}
\begin{align}
 \lambda^s\rs^{\frac{2}{p}-1}\rp^{\frac1p-\frac12}\tau^{\frac12-\frac{1}{\gamma}} &\ll 1 \quad\ (w_{q+1}^{(p)}\in L^\gamma_tW^{s,p}_x)   \label{setpw} \\
 \mu\rs^{2}\rp^{-\frac12}\tau^{-\frac12}&\ll 1 \quad\ (\text{Time derivative error for}\ w_{q+1}^{(p)})  \label{setpt} \\
 \lambda^{2\a-1}\rs\rp^{\frac12}\tau^{-\frac12}&\ll 1  \quad\ (\text{Hyperdissipativity error for}\ w_{q+1}^{(p)}) \label{setdeltap} \\
 \lambda^{2\a-1}\mu^{-1}&\ll 1  \quad\ (\text{Hyperdissipativity error for}\ w_{q+1}^{(t)}) \label{setdeltat} \\
 \lambda^{-1}\rs^{-1 }&\ll 1 \quad\ (\text{Oscillation error for}\ w_{q+1}^{(p)}) \label{setrosc1} \\
\mu^{-1} \sigma\tau&\ll 1 \quad\ (\text{Oscillation error for}\ w_{q+1}^{(t)}) \label{setrosc2}
\end{align}
\end{subequations}

It is important here that,
in order to ensure the validity of these constrains,
one shall exploit the temporal intermittency
in an almost optimal way.
It turns out that,
the suitable temporal intermittency will roughly equal to
$(4\alpha-5)$-dimensional
spatial intermittency.
In particular,
the temporal intermittency almost achieves the 3D
spatial intermittency when $\alpha$ is close to $2$.
We show that there do exist six admissible parameters $(\rs,\rp,\lambda,\mu,\tau,\sigma)$
and give the
precise choice in \eqref{larsrp} below.

{\bf $\bullet$ Endpoint case $(2\alpha/\gamma+1-2\alpha, \gamma, \infty)$.}
Regarding the other endpoint $(2\alpha/\gamma+1-2\alpha, \gamma, \infty)$,
$\alpha\in [1,2)$,
one may be inclined to use the building blocks in the previous endpoint case,
which, unfortunately, leads to the emptiness of
the admissible parameters for the constrains \eqref{setpw}-\eqref{setrosc2}.
This is mainly due to the presence of the $\mu t$ term in the intermittent jets,
which gives rise to the restrictions \eqref{setdeltat} and \eqref{setrosc2}
that contradict with each other.

Inspired by the work \cite{cl20.2},
we use the concentrated Mikado flows instead as the spatial building blocks
(see \eqref{snwd-endpt2} below),
which are indexed by two parameters $(\rs,\lambda)$.
We note that the $\mu t$ term and
the concentration parameter $\rp$
are not involved in the concentrated Mikado flows.
This permits to reduce the constrains to
\begin{subequations}\label{constset2}
\begin{align}
 \lambda^s\rs^{\frac{2}{p}-1}\tau^{\frac12-\frac{1}{\gamma}} &\ll 1  \quad\ (w_{q+1}^{(p)}\in L^\gamma_tW^{s,p}_x)  \label{setpw.2} \\
 \sigma\lambda^{-1}\rs\tau^{\frac12}&\ll 1 \quad\ (\text{Time derivative error for}\ w_{q+1}^{(p)}) \label{setpt.2} \\
 \lambda^{2\a-1}\rs\tau^{-\frac12}&\ll 1 \quad\ (\text{Hyperdissipativity error for}\ w_{q+1}^{(p)})  \label{setdeltap.2} \\
 \lambda^{-1}\rs^{-1 }&\ll 1 \quad\ (\text{Oscillation error for}\ w_{q+1}^{(p)})  \label{setrosc1.2}
\end{align}
\end{subequations}
The absence of the large parameter $\mu$ permits more flexibility for the
choice of parameters $(\rs, \lambda, \tau, \sigma)$.
It turns out that,
there do exist four admissible parameters to fulfill the constrains \eqref{setpw.2}-\eqref{setrosc1.2}.
Even though the Mikado flows give at most 2D intermittency,
the suitable temporal building blocks would provide
much more intermittency,
which almost reaches 4D spatial intermittency when $\alpha$ is close to $2$.
The precise choice of the four parameters is given in \eqref{larsrp-endpt2} below.

Let us mention that,
unlike the intermittent jets,
the concentrated Mikado flows proved 2D intermittency,
which is insufficient to handle the endpoint case $(3/p+1-2\alpha,\infty,p)$.
Thus, the building blocks for the endpoint $(2\alpha/\gamma+1-2\alpha, \gamma, \infty)$
are not applicable to the previous endpoint case.
It would be very interesting to construct the building blocks
in a unified manner
for both endpoints and even for the rest of critical scaling values in \eqref{critical-LPS-hyperNSE},
which still remains open.
Nevertheless, once the velocity perturbations constructed,
we prove Theorems \ref{Thm-Non-hyper-NSE}
and \ref{Thm-hyperNSE-Euler-limit} by using unified arguments.

The rest structure of this paper is organized as follows.
In Section \ref{Sec-Concen-Rey} we use the gluing technique to construct
new approximate solutions.
In particular, the new Reynolds stress concentrates
on smaller temporal supports.
Then, Sections \ref{Sec-Flow-Endpt1} and \ref{Sec-Rey-Endpt1} are mainly devoted to the
endpoint point case $(3/p+1-2\alpha,\infty,p)$.
More precisely, we first construct the velocity perturbations and
prepare the important algebraic identities and analytic estimates in
Section \ref{Sec-Flow-Endpt1}.
Then,
we treat the Reynolds stress in Section \ref{Sec-Rey-Endpt1}.
The other endpoint case $(2\alpha/\gamma+1-2\alpha, \gamma, \infty)$
is mainly treated in Section \ref{Sec-Endpt2}.
At last, the proofs of main results are contained in Section \ref{Sub-Proof-Main}.
Section \ref{Set-App}, i.e., the Appendix, contains some preliminary results used in the proof.

\section{Concentrating the Reynolds error}  \label{Sec-Concen-Rey}

This section is devoted to construct a new smooth solution $(\wt u_q, \tr_q)$ to \eqref{equa-nsr}.
In particular, the new Reynolds stress concentrates on much smaller intervals,
while still keeping the rapid decay in the space $L^1_{t,x}$.

For this purpose,
we divide the time interval $[0,T]$ into $m_{q+1}$ many subintervals $[t_i, t_{i+1}]$,
and denote by $\theta_{q+1}$ the length scale of bad sets supporting the new Reynolds stress.
The two parameters $m_{q+1}$ and $\theta_{q+1}$ are chosen in the following way
\begin{align}\label{def-mq-thetaq}
T/\mq=\la^{-12}, \quad
\theta_{q+1}:=(T/m_{q+1})^{1/\eta} \simeq \lambda_q^{-\frac{12}{\eta}},
\end{align}
where $\eta$ is a small constant such that
$$0<\frac{\eta_*}{2}<\eta<\eta_*<1.$$
Without loss of generality, we assume $\mq$ is an integer to
such that the time interval is perfectly divided.

We also recall from \cite{dls13} the inverse-divergence operator $\mathcal{R}$,
defined by
 \begin{align} \label{calR-def}
 	& (\mathcal{R} v)^{kl} := \partial_k \Delta^{-1} v^l + \partial_l \Delta^{-1} v^k - \frac{1}{2}(\delta_{kl} + \partial_k \partial_l \Delta^{-1})\div \Delta^{-1} v,
 \end{align}
where $v$ is mean-free, i.e., $\int_{\mathbb{T}^3} v dx =0$.
Note that, the inverse-divergence operator $\mathcal{R}$
maps mean-free functions to symmetric and trace-free matrices.
Moreover, one has the algebraic identity
\begin{align*}
	\div \mathcal{R}(v) = v.
\end{align*}

Proposition \ref{prop-nunr} below is the main result of this section,
which provides a new well-prepared solution
$(\wt u_q, \tr_q)$ to \eqref{equa-nsr}
with concentrated support set $I_{q+1}$ and
much smaller length scale $\theta_{q+1}$.

\begin{proposition}   \label{prop-nunr}
Consider $(s,p,\gamma)\in \mathcal{A}_1$ for $\alpha \in [5/4,2)$,
or $(s,p,\gamma) \in \mathcal{A}_2$ for $\alpha \in [1,2)$.
Let $(u_q,\rr_q)$ be a well-prepared smooth solution to \eqref{equa-nsr}
for some set $I_q$ and a length scale $\theta_q$.
Then, there exists another well-prepared solution $(\wt u_q,\tr_q)$ to \eqref{equa-nsr}
for some set $I_{q+1}\subseteq I_q$, $0,T\notin I_{q+1}$
and the length scale $\theta_{q+1}(<\theta_q/2)$,
satisfying:
 \begin{align}
    & \tr_q(t,x)=0 \quad \text{if} \quad \operatorname{dist}(t,I_{q+1}^c)\leq {3}\thq/2,\label{suppnr}\\
    & \|\wt u_q\|_{L^\infty_tH^3_x}\lesssim \lambda_{q}^{5},\label{nuh3}\\
    & \|\wt u_q -\u \|_{L^{\infty}_tL^2_x} \lesssim  \la^{-3},\label{uuql2}\\
	& \|\tr_q\|_{L^{1}_{t,x}} \leq  \la^{-\frac{\ve_R}{4}}\delta_{q+1}, \label{nrl1}\\
    & \|\partial_t^M \nabla^N \tr_q\|_{L^\infty_tH^3_x} \lesssim \thq^{-M-1}\mq^{-\frac{N}{2\alpha}}\lambda_q^{5}
	\lesssim \thq^{-M-N-1} \lambda_q^{5},  \label{nrh3}
\end{align}
where the implicit constants are independent of $q$.
\end{proposition}

\subsection{Stability estimates}

We recall the regularity estimates of the strong solutions
in \cite{bcv21} for $\a\in [1,5/4)$,
the proof there also applies to the case where $\a\in [1,2)$.

\begin{proposition} (\cite{bcv21}) \label{Prop-LWP-Hyper-NLSE}
Let $\a\in [1,2)$, $v_0\in H_x^3(\T^3)$ be a mean free function and consider the
Cauchy problem for \eqref{equa-NS} with initial data $v|_{t=t_0}=v_0$. If
\begin{align} \label{t*-t0-v0H3}
	 0<t_*-t_{0}\leq \frac{c}{\|v_0\|_{H^3_x}}
\end{align}
for some universal constant $c\in (0,1]$,
then there exists a unique strong solution $v$ to \eqref{equa-NS} on
$[t_0,t_{*}]$ satisfying
\begin{align}
\sup _{t \in\left[t_0,t_{*}\right]}\|v(t)\|_{L^{2}_x}^{2}
+2 \int_{t_{0}}^{t_{*}}\|v(t)\|_{\dot{H}_x^{\alpha}}^{2} d t
& \leq\left\|v_{0}\right\|_{L^{2}_x}^{2}, \label{vl2}\\
\sup _{t \in\left[t_0,t_{*}\right]}\|v(t)\|_{H^{3}_x} &\leq 2\left\|v_{0}\right\|_{H^{3}_x}. \label{vh3}
\end{align}
Moreover, if
\begin{align}  \label{con-pdvh3}
 0< t_{*}-t_0 \leq  \frac{ c}{ \norm{v_0}_{H^3_x}  (1+ \norm{v_0}_{L^2_x})^{\frac{1}{2\alpha-1}}},
\end{align}
then it holds that for any $N\geq 0$ and $M \in\{0,1\}$,
\begin{align}
\sup_{t\in (t_0,t_{*}]}  |t-t_0|^{\frac{N}{2\alpha}+M}
\norm{\partial_t^M \na^N v(t)}_{H_x^3}   \lesssim  \norm{v_0}_{H^3_x}\,,
\label{pdvh3}
\end{align}
where the implicit constant depends on $\alpha, N, M$.
\end{proposition}

On each subinterval $[t_i, t_{i+1}+\theta_{q+1}]$,
$0\leq i\leq m_{q+1}-1$,
we solve the following hyperdissipative Navier-Stokes equations
\begin{equation}\label{equa-nsuq}
	\left\{\aligned
	&\p_tv_i +\nu(-\Delta)^{\alpha} v_i+(v_i\cdot \nabla )v_i + \nabla P_i=0,  \\
	& \div \ v_i = 0,\\
    & v_i|_{t=t_i}=u_q(t_i).
	\endaligned
	\right.
\end{equation}
The existence of strong solutions to \eqref{equa-nsuq} can be guaranteed by the classical local well-posedness
theory of the hyperdissipative Navier-Stoke equations.
Actually, by \eqref{uh3} and \eqref{upth2},
the conditions \eqref{t*-t0-v0H3} and \eqref{con-pdvh3}
are verified with $t_*, t_0$ replaced by $t_{i+1}+\theta_{q+1}$ and $t_i$,
respectively. Thus, there exists a unique solution $v_i$ to \eqref{equa-nsuq}
on $[t_i, t_{i+1}+\theta_{q+1}]$.

Let $w_i:=u_q-v_i$.
Then $w_i: [t_i, t_{i+1}]\times \T^3\rightarrow \R^3$ satisfies the following equations
\begin{equation}\label{equa-vi}
\left\{\begin{array}{l}
\partial_{t} w_{i}+(-\Delta)^\a w_{i}+\div(v_i\otimes w_i+w_i\otimes u_q )+\nabla p_i =\div \rr_q, \\
\operatorname{div} w_{i}=0, \\
w_{i}|_{t=t_i}=0,
\end{array}\right.
\end{equation}
for some $p_i:[t_i,t_{i+1}]\times \T^3\rightarrow \R$.

Lemma \ref{lem-est-vi} below contains the stability estimates
for the solutions $w_i$ to \eqref{equa-vi}.

\begin{lemma} (Stability estimates) \label{lem-est-vi}
Let $\a\in [1,2)$, $1<\rho\leq 2$ and $(u_q,\rr_q)$ be the well-prepared solution to \eqref{equa-nsr}.
Then, there exists a universal constant $C$,
depending only on $\a$ and $\rho$,
such that the following estimates hold
\begin{align}
&\|w_i\|_{L^\9([t_i,t_{i+1}+\theta_{q+1}];L^\rho_x)}\leq C\int_{t_i}^{t_{i+1}+\theta_{q+1}}\| \abs{\nabla} \rr_q(s)\|_{L^\rho_x} ds,\label{est-vilp}\\
&\|w_i\|_{L^\9([t_i,t_{i+1}+\theta_{q+1}];H^3_x)}\leq C\int_{t_i}^{t_{i+1}+\theta_{q+1}}\| \abs{\nabla} \rr_q(s)\|_{H^3_x} ds,\label{est-vih3}\\
&\|\mathcal{R}w_i\|_{L^\9([t_i,t_{i+1}+\theta_{q+1}];L^\rho_x)}\leq C\int_{t_i}^{t_{i+1}+\theta_{q+1}}\| \rr_q(s)\|_{L^\rho_x} ds,  \label{est-rvi}
\end{align}
where $\mathcal{R}$ is the inverse divergence operator given by \eqref{calR-def}.
\end{lemma}

\begin{proof}
Without loss of generality, we may consider the case $t_i=0$. Let $\P_{H}$ denote the Helmholtz-Leray projector, i.e.,
$\P_{H}=\Id-\nabla\Delta^{-1}\div$.
In order to prove \eqref{est-vilp},
we apply the semigroup method and reformulate equation \eqref{equa-vi}
as follows
\begin{align}\label{e2.12}
  w_i(t)=\int_{0}^{t} e^{-(t-s)(-\Delta)^\a} \P_H \div (\rr_q-v_i\otimes w_i-w_i\otimes u_q )(s)\d s.
\end{align}
Applying the classical semigroup estimates (cf. \cite[(3.14)]{bcv21}) to \eqref{e2.12}
we get
\begin{align}\label{e2.13}
 \|w_i(t)\|_{L^\rho_x} & \leq \int_{0}^{t} \|e^{-(t-s)(-\Delta)^\a}  \P_H \div (\rr_q-v_i\otimes w_i-w_i\otimes u_q )(s)\|_{L^\rho_x}\d s \notag\\
  & \lesssim \int_{0}^{t} \|\abs{\nabla}\rr_q\|_{L^\rho_x}+(t-s)^{-\frac{1}{2\a}}\|(v_i\otimes w_i+w_i\otimes u_q )(s)\|_{L^\rho_x}\d s \notag \\
 & \leq C_* \int_{0}^{t} \|\abs{\nabla}\rr_q\|_{L^\rho_x}+(t-s)^{-\frac{1}{2\a}}(\|v_i(s)\|_{L^\9_x}+\|u_q(s)\|_{L^\9_x})\|w_i(s)\|_{L^\rho_x}\d s,
\end{align}
for some universal constant $C_*$ depending only on $\rho$ and $\a$.

We claim that for all $t\in [0,t_1+\theta_{q+1}]$,
\begin{align}\label{e2.14}
\|w_i(t)\|_{L^\rho_x}\leq 2 C_* \int_{0}^{t} \|\abs{\nabla}\rr_q(s)\|_{L^\rho_x}ds.
\end{align}

We prove \eqref{e2.14} via the bootstrap argument.
First note that,
\eqref{e2.14} is valid for $t=0$.
Moreover, if \eqref{e2.14} holds,
we prove that the same estimate also holds with
the constant $2C_*$
replaced by a smaller constant such as $3C_*/2$.
To this end, plugging \eqref{e2.14} into \eqref{e2.13} we get
\begin{align}
\norm{w_i(t)}_{L^\rho_x}
& \leq 2 C_* \int_{0}^{t} \|\abs{\nabla}\rr_q(s)\|_{L^\rho_x}ds\( \frac{1}{2} + C_*\left( \norm{v_i}_{L^\infty_{t,x}} + \norm{u_q}_{L^\infty_{t,x}} \right) \int_0^t (t-s)^{-\frac{1}{2\alpha}}ds\) \notag\\
&\leq   2 C_* \int_{0}^{t} \|\abs{\nabla}\rr_q(s)\|_{L^\rho_x}ds\( \frac{1}{2} + \frac{2\a C_*}{2\a-1}t^{1-\frac{1}{2\a}}\left( \norm{v_i}_{L^\infty_{t,x}} + \norm{u_q}_{L^\infty_{t,x}} \right) \) .
\label{e2.15}
\end{align}
Since $\a\in [1,2)$, for any $t_1$ small enough such that
\begin{align} \label{e2.16}
    \frac{2\alpha C_*}{2\alpha-1} (t_1+\theta_{q+1})^{1-\frac{1}{2\a}} \left( \norm{v_i}_{L^\infty_{t,x}} + \norm{u_q}_{L^\infty_{t,x}} \right) \leq \frac 14,
\end{align}
we have \eqref{e2.14} with the constant replaced by ${3}C_*/2$.
Concerning the left-hand-side of \eqref{e2.16},
using the Sobolev embedding $H^3_x\hookrightarrow L^\9_x$,
\eqref{uh3}, \eqref{def-mq-thetaq} and \eqref{vh3} we have
\begin{align}
  \frac{2\alpha C_*}{2\alpha-1}
  (t_1+\theta_{q+1})^{1-\frac{1}{2\a}}\(\norm{v_i}_{L^\infty_{t,x}} + \norm{u_q}_{L^\infty_{t,x}} \)
  &\leq 4CC_* (T/m_{q+1})^{1-\frac{1}{2\a}}\left( \norm{v_i}_{L^\infty_{t}H^3_x} + \norm{u_q}_{L^\infty_{t}H^3_x} \right)\notag\\
  &\leq 12CC_* (\la^{-12})^{\frac12+\frac{\a-1}{2\a}} \norm{u_q}_{L^\infty_{t}H^3_x} \notag\\
  &\leq C' (\la^{-12})^{\frac12+\frac{\a-1}{2\a}}\la^5
  \leq C' \la^{-1},
\end{align}
for some universal constant $C'$.
Thus, \eqref{e2.16} is verified
for $a$ sufficiently large such that $C'\la^{-1}\leq 1/4$.
This yields \eqref{e2.14} and so \eqref{est-vilp}.

Regarding \eqref{est-vih3},  since $H^3_x$ is an algebra,
using the classical semigroup estimates we get
\begin{align}\label{vih3}
\|w_i(t)\|_{H^3_x} & \leq \int_{0}^{t} \|e^{-(t-s)(-\Delta)^\a}  \P_H \div (\rr_q-v_i\otimes w_i-w_i\otimes u_q )(s)\|_{H^3_x}\d s \notag\\
 & \leq C_* \int_{0}^{t} \|\abs{\nabla}\rr_q\|_{H^3_x}+(t-s)^{-\frac{1}{2\a}}(\|v_i(s)\|_{H^3_x}+\|u_q(s)\|_{H^3_x})\|w_i(s)\|_{H^3_x}\d s.
\end{align}

Similarly to  \eqref{e2.14}, we claim that
for all $t\in [0,t_1+\theta_{q+1}]$,
\begin{align}\label{claim-vih3}
\|w_i(t)\|_{H^3_x}\leq 2 C_* \int_{0}^{t} \|\abs{\nabla}\rr_q(s)\|_{H^3_x}ds.
\end{align}

Actually, plugging \eqref{claim-vih3} into \eqref{vih3} we get
\begin{align}
\norm{w_i(t)}_{H^3_x}& \leq 2C_* \int_{0}^{t} \|\abs{\nabla}\rr_q(s)\|_{H^3_x}ds\( \frac{1}{2}
+ \frac{2\a C_*}{2\a-1}t^{1-\frac{1}{2\a}}\left( \norm{v_i}_{L^\infty_{t}H^3_{x}} + \norm{u_q}_{L^\infty_{t}H^3_{x}} \right) \).
\end{align}
In view of \eqref{uh3}, \eqref{def-mq-thetaq} and \eqref{vh3}, one has
\begin{align}\label{e2.23}
      \frac{2\a C_*}{2\a-1}t^{1-\frac{1}{2\a}}\left( \norm{v_i}_{L^\infty_{t}H^3_{x}} + \norm{ u_q}_{L^\infty_{t}H^3_{x}} \right)\notag
\leq &\, 6C_*(t_1+\theta_{q+1})^{\frac12+\frac{2\a-1}{2\a}} \norm{u_q}_{L^\infty_{t}H^3_x} \notag\\
\leq &\, C' (\la^{-12})^{\frac12+\frac{2\a-1}{2\a}} \la^{5}
\leq C' \la^{-1},
\end{align}
where $C'$ is a universal constant.
Thus, we see that the constant in \eqref{claim-vih3} can be improved by $3C_*/2$,
by assuming $a$ sufficiently large such that $C'\la^{-1}<1/4$,
which, via the bootstrap argument yields \eqref{claim-vih3}
for all $t\in [0,t_1+\theta_{q+1}]$, as claimed.
Thus, estimate \eqref{est-vih3} follows.

It remains to prove \eqref{est-rvi}.
Let
\begin{align}  \label{z-wi}
  z: = \Delta^{-1}\curl w_i.
\end{align}
Note that $\curl z=-w_i$,
as $w_i$ is divergence free.
By the boundedness of Calder\'{o}n-Zygmund operators in $L^\rho$ spaces for $1<\rho<\9$,
it holds that for any $t\in [t_i,t_{i+1}]$,
$$ \|\mathcal{R}w_i(t)\|_{L^\rho_x}\leq C_1 \|z(t)\|_{L^\rho_x}, $$
where $C_1$ is a universal constant. Moreover, $z$ satisfies the equation (see e.g. \cite{bcv21,I18} for more details),
\begin{align}
\partial_t z + v_i \cdot \nabla z  +(- \Delta)^\alpha  z
&= \Delta^{-1} \curl \div \rr_q +  \Delta^{-1} \curl \div \left(  (z\times \nabla) v_i  \right) \notag\\
&\quad +  \Delta^{-1} \nabla \div  \left( ( z \cdot \nabla) v_i\right)    + \Delta^{-1} \curl \div \left( ((z\times \nabla) u_q)^T \right),   \label{equa-z}
\end{align}
which can be reformulated as follows
\begin{align}
z(t)
&= \int_0^t e^{-(t-s)(-\Delta)^\alpha} \left( \Delta^{-1} \curl \div \rr_q  +   \Delta^{-1} \curl \div \left( ((z\times \nabla) u_q)^T \right)  - \div (v_i \otimes z)\right)(s) \d s \notag\\
&\qquad + \int_0^t e^{-(t-s)(-\Delta)^\alpha} \left(   \Delta^{-1} \curl \div \left( (z\times \nabla) v_i  \right) + \Delta^{-1} \nabla \div  \left( ( z \cdot \nabla) v_i\right)   \right)(s) \d s.
\label{euqa-zinte}
\end{align}
Then, by virtue of the boundedness of Calder\'{o}n-Zygmund operators in $L^\rho$, $\rho\in(1,2)$,
and the classical semigroup estimates,
we derive
\begin{align}
\norm{z(t)}_{L^\rho_x} &\leq C_* \bigg(\int_0^t \norm{\rr_q(s)}_{L^\rho_x}\d s  + \left(  \norm{\nabla v_i }_{L^\infty_{t,x}}
+   \norm{\nabla u_q}_{L^\infty_{t,x}} \right) \int_0^t \norm{z(s)}_{L^\rho_x} \d s\notag\\
&\qquad\quad + \norm{v_i}_{L^\infty_{t,x}} \int_0^t (t-s)^{-\frac{1}{2\alpha}}\norm{z(s)}_{L^\rho_x} \d s\bigg),
\label{e2.20}
\end{align}
where $C_*$ is a universal constant depending only on $\rho$ and $\a$.

As in the case of \eqref{e2.14},
we claim that for any $t\in [0,t_1+\theta_{q+1}]$,
\begin{align}
&\|z(t)\|_{L^\rho_x}\leq 2C_*\int_{0}^{t}\| \rr_q(s)\|_{L^\rho_x} \d s. \label{e2.21}
\end{align}
For this purpose, inserting \eqref{e2.21} into \eqref{e2.20} we get
\begin{align}
\norm{z(t)}_{L^\rho_x}& \leq 2 C_* \int_{0}^{t} \| \rr_q(s)\|_{L^\rho_x}\d s\( \frac{1}{2}
+ C_*t\left( \norm{\nabla v_i}_{L^\infty_{t,x}} + \norm{\nabla u_q}_{L^\infty_{t,x}} \right)+ \frac{2\a C_*}{2\a-1}t^{1-\frac{1}{2\a}}\norm{v_i}_{L^\infty_{t,x}}\) .
\label{e2.22}
\end{align}
Then, by \eqref{uh3}, \eqref{def-mq-thetaq}, \eqref{vh3} and the Sobolev embedding,
\begin{align}\label{e2.23}
&C_*t\left( \norm{\nabla v_i}_{L^\infty_{t,x}}
+ \norm{\nabla u_q}_{L^\infty_{t,x}} \right)
+ \frac{2\a C_*}{2\a-1}t^{1-\frac{1}{2\a}}\norm{v_i}_{L^\infty_{t,x}}\notag\\
\leq &\, CC_*\((t_1+\theta_{q+1})\left( \norm{v_i}_{L^\infty_{t}H^3_x}
    + \norm{u_q}_{L^\infty_{t}H^3_x} \right)
	+ 2(t_1+\theta_{q+1})^{1-\frac{1}{2\a}}\norm{v_i}_{L^\infty_{t}H^3_x}\)\notag\\
\leq &\, CC_*(3(t_1+\theta_{q+1})+4(t_1+\theta_{q+1})^{1-\frac{1}{2\a}})\norm{u_q}_{L^\infty_{t}H^3_x}\notag\\
 \leq &\, C'\(\la^{-12}+(\la^{-12})^{\frac12+\frac{2\a-1}{2\a}}\) \la^{5}
 \leq 2C'\la^{-1},
\end{align}
which along with \eqref{e2.22} yields that
for $a$ sufficiently large such that $2C'\la^{-1}<1/4$,
the constant in \eqref{e2.21} can be replaced by $ 3C_*/2$.
This yields \eqref{e2.21} for all $t\in [0,t_1+\theta_{q+1}]$, as claimed.
Thus, in view of \eqref{z-wi},
we prove \eqref{est-rvi}.
Therefore, the proof is complete.
\end{proof}

\subsection{Temporal gluing of local solutions}
From the previous section,
we see that $v_i$ is exactly the solution to the hyperdissipative Navier-Stokes equations \eqref{equa-NS}
on each subinterval $[t_i,t_{i+1}+\theta_{q+1}]$, $0\leq i\leq \mq-1$.
Then, we glue the local solutions together
in an appropriate way such that the glued solution $\sum_i \chi_i v_i$ is the exact solution to \eqref{equa-NS}
in a majority part of the time interval $[0,T]$,
while the Reynolds stress error has smaller disjoint supports in time.

More precisely, we let $\{\chi_i\}_{i=0}^{\mq-1}$ be a $C_0^\9$ partition of unity on $[0,T]$  such that
$$
0\leq  \chi_i(t) \leq 1, \quad\text{for} \quad t\in[0,T],
$$
for $0<i< \mq-1$,
\begin{align}\label{def-chi1}
\chi_{i}= \begin{cases}1 & \text { if } t_{i}+\thq \leq t \leq t_{i+1}, \\
 0 & \text { if } t\leq t_{i}, \text { or } t \geq t_{i+1}+\thq,\end{cases}
\end{align}
and for $i=0$,
\begin{align}\label{def-chi2}
\chi_{i}= \begin{cases}1 & \text { if } 0 \leq t \leq t_{i+1},
\\ 0 & \text { if } t \geq t_{i+1}+\thq,\end{cases}
\end{align}
and for $i= \mq-1$,
\begin{align}\label{def-chi3}
\chi_{i}= \begin{cases}1 & \text { if } t_{i}+\thq \leq t \leq T, \\
  0 & \text { if } t \leq t_{i}.\end{cases}
\end{align}
Furthermore, we assume that $\chi_i$, $0\leq i\leq \mq-1$,
satisfy the following bounds,
\begin{align}\label{est-chi}
\|\partial_t^M \chi_i\|_{L^{\9}_t}\lesssim \thq^{-M},
\end{align}
where the implicit constant is independent of $\thq$ and $i, M\geq 0$.

Now, let
\begin{align}\label{def-wtu}
\wt u_q:=\sum_{i=0}^{\mq-1} \chi_i v_i.
\end{align}
Note that, $\wt u_q: [0,T]\times \T^3\rightarrow \R^3$ is divergence and mean free.
Moreover, for $t\in [t_i, t_{i+1}]$, we have
\[
\wt u_q=(1-\chi_i)v_{i-1}+\chi_i v_i,
\]
and $\wt u_q$ satisfies
\begin{align}\label{equa-wtu}
\partial_{t} \wt u_q+ \nu (-\Delta)^\a \wt u_q +\operatorname{div}(\wt u_q \otimes \wt u_q) +\nabla \wt p=\div \tr_q,
\end{align}
where
the new Reynolds stress is of form
\begin{align}\label{def-nr}
\tr_q =\partial_t\chi_i\mathcal{R}(v_{i}-v_{i-1}) -\chi_{i}(1-\chi_{i})((v_{i}-v_{i-1})\mathring\otimes (v_{i}-v_{i-1})),
\end{align}
and the pressure $\wt p:[0,1] \times \mathbb{T}^{3} \rightarrow \mathbb{R}$ is given by
\begin{align}\label{def-wp}
\wt p= \chi_{i} p_{i}+(1-\chi)p_{i-1}-\chi_i(1-\chi_i)\left(|v_{i}-v_{i-1}|^2-\int_{\T^3}|v_{i}-v_{i-1}|^2\d x\right).
\end{align}

\subsection{Proof of Proposition~\ref{prop-nunr}}
Using the definition of $\chi_i$ and $v_i$, \eqref{uh3} and \eqref{vh3} we get
\begin{align}\label{ver-nuh3}
\|\wt u_q\|_{L^\9_tH^3_x}& \leq \|\sum_i \chi_i v_i\|_{L^\9_tH^3_x}\notag\\
&\leq \sup_i \(\|(1-\chi_i)v_{i-1}\|_{L^\9_t(\supp(\chi_i\chi_{i-1});H^3_x)}+\|\chi_i v_i\|_{L^\9_t(\supp(\chi_i);H^3_x)}\)\notag\\
&\leq 4\|u_q\|_{L^\9_tH^3_x}\lesssim \la^{5},
\end{align}
which yields \eqref{nuh3}.

Regarding the estimate \eqref{uuql2},
using \eqref{rh3}, \eqref{def-mq-thetaq} and \eqref{est-vilp}
we get
\begin{align}\label{ne2.42}
\|\wt u_q-u_q\|_{L^\9_tL^2_{x}}& \leq \|\sum_i \chi_i(v_i-u_q)\|_{L^\9_tL^2_{x}} \notag\\
&\leq \sup_i\(\|v_i-u_q\|_{L^\9_t(\supp(\chi_i);L^2_x)}+\|v_{i-1}-u_q\|_{L^\9_t(\supp(\chi_i\chi_{i-1});L^2_x)}\)\notag\\
& \lesssim \sup_i|\supp(\chi_i)|\| \abs{\nabla} \rr_q\|_{L^\9_tL^2_x}\notag\\
&\lesssim \mq^{-1}\|\rr_q\|_{L^\9_tH^3_x}\lesssim \mq^{-1}\la^9 \lesssim \la^{-3},
\end{align}
Hence, estimate \eqref{uuql2} is verified.

Concerning the $L^1$-estimate of the new Reynolds stress, by \eqref{def-nr},
\begin{align}
\|\tr_q\|_{L_{t,x}^{1}} & \leq\|\partial_t\chi_i\mathcal{R}(v_{i}-v_{i-1})\|_{L_{t,x}^1 }+\|\chi_{i}(1-\chi_{i})((v_{i}-v_{i-1})\mathring\otimes (v_{i}-v_{i-1}))\|_{L_{t,x}^{1}} \notag\\
&=:J_1+J_2.\label{e2.33}
\end{align}
Regarding the estimate of $J_1$, we choose
\begin{align*}
1<\rho<\frac{4\ve_R+36+8\beta b}{\ve_R+36+8\beta b},
\end{align*}
where $b,\beta$ are given by \eqref{b-beta-ve} and
$\ve_R$ is given by \eqref{rl1},
to get
\begin{align}
 J_1 \leq&  \sum_i \|\p_t \chi_i\|_{L^1_t}\|\mathcal{R}(v_i-v_{i-1})\|_{L^\9_tL^\rho_x}  \notag \\
 \lesssim& \sum_i \|\mathcal{R}(w_i-w_{i-1})\|_{L^\9_t(\supp(\chi_i\chi_{i-1});L^\rho_x)},
\end{align}
which along with the Gagliardo-Nirenberg inequality,
\eqref{rh3}, \eqref{rl1} and \eqref{est-rvi} yields that
\begin{align}   \label{est-j1}
    J_1 \lesssim& \sum_i\int_{t_i}^{t_{i+1}} \|\rr_q(s)\|_{L^\rho_x}\d s \notag\\
	\lesssim&  \|\rr_q\|_{L^1_{t,x}}^{1-\frac{2(\rho-1)}{3\rho}}\|\rr_q\|_{L^\9_tH^3_x}^{\frac{2(\rho-1)}{3\rho}} \notag \\
    \lesssim& \la^{-\ve_R+(\ve_R+9)\frac{2(\rho-1)}{3\rho}}\delta_{q+1}^{1-\frac{2(\rho-1)}{3\rho}} \leq \la^{-\frac{3\ve_R}{8}}\delta_{q+1},
\end{align}
where in the last step we chose $a$ sufficiently large
and used $\la^{-\ve_R/8}$ to absorb the implicit constant.

Concerning the estimate of $J_2$, using \eqref{est-vilp} we obtain
\begin{align}
  J_2 & \leq \sum_i \|\chi_{i}-\chi_{i}^{2} \|_{L^1_t} \|  ((v_{i}-v_{i-1})\mathring\otimes (v_{i}-v_{i-1}))\|_{L^\9_tL^1_x} \notag \\
  &\lesssim \sum_i |\supp_{t} \chi_i(1-\chi_i)|\|\chi_{i}-\chi_{i}^{2} \|_{L^\9_t} \|  w_{i}-w_{i-1}\|_{L^\9_t(\supp(\chi_i\chi_{i-1});L^2_x)}^2 \notag \\
  & \lesssim \sum_i \thq \(\int_{t_i}^{t_{i+1}}\||\nabla|\rr_q(s)\|_{L^2_x} \d s\)^2  \notag\\
  &\lesssim \thq \||\nabla|\rr_q\|_{L^1_tL^2_x} ^2,
\end{align}
then using \eqref{rh3}, \eqref{rl1} and the fact that $\la^{-1}\ll \delta_{q+1}^{1/9}$ and $0<\eta<1$ we lead to
\begin{align}  \label{est-j2}
	J_2 \lesssim \mq^{-\frac{1}{\eta}} \|\rr_q\|_{L^1_{t,x}}^{\frac{8}{9}}\|\rr_q\|_{L^1_tH^3_x}^{\frac{10}{9}}
  \lesssim \la^{-12/\eta}\la^{-\frac89\ve_R}\delta_{q+1}^{\frac89} \la^{10}\leq \la^{-\frac12\ve_R}\delta_{q+1},
\end{align}
where we also chose $a$ sufficient large and used $\la^{-\ve_R/4}$ to absorb the implicit constant.

Thus, we conclude from \eqref{est-j1}-\eqref{est-j2} that
\begin{align}\label{est-trq}
\|\tr_q\|_{L^1_{t,x}}\leq \la^{-\frac{\ve_R}{4}}\delta_{q+1},
\end{align}
and so \eqref{nrl1} is verified.

Regarding the estimate \eqref{nrh3},
by \eqref{def-nr}, for $t\in[t_i,t_{i+1}]$, $0\leq i\leq \mq-1$,
\begin{align}\label{e2.45}
\|\partial_t^M \nabla^N \tr_q\|_{L^\infty_tH^3_x} & \leq\| \partial_t^M \nabla^N(\partial_t\chi_i\mathcal{R}(v_{i}-v_{i-1}))\|_{L^\infty_tH^3_x }\notag\\
&\quad +\|\partial_t^M \nabla^N(\chi_{i}(1-\chi_{i})((v_{i}-v_{i-1})\mathring\otimes (v_{i}-v_{i-1})))\|_{L^\infty_tH^3_x}.
\end{align}
For the first term on the right-hand-side of \eqref{e2.45},
by \eqref{uh3}, \eqref{vh3}  and the fact that $\thq^{-1} \geq \mq$,
\begin{align}\label{e2.46}
&\norm{\partial_t^M \nabla^N (\partial_t \chi_i \mathcal{R}(v_i-v_{i-1})) }_{L^\infty_tH^3_x} \notag\\
\lesssim &\sum_{M_1+M_2=M}\norm{\partial_t^{M_1+1} \chi_i}_{L^\infty_t} \left(\norm{\partial_t^{M_2} \nabla^N v_i}_{L^\infty_t(\supp(\chi_i);H^3_x)} + \norm{\partial_t^{M_2} \nabla^N v_{i-1}}_{L^\infty_t(\supp(\chi_{i-1});H^3_x)}\right) \notag\\
\lesssim& \sum_{M_1+M_2=M} \thq^{-M_1-1}\mq^{\frac{N}{2\alpha} + M_2} \lambda_q^{5} \lesssim \thq^{-M-1} \mq^{\frac{N}{2\alpha} } \lambda_q^{5} \,,
\end{align}
where the implicit constants are independent of $i$ and $q$.
In order to control the second term of \eqref{e2.45}, using \eqref{uh3}, \eqref{vh3}, \eqref{pdvh3}
and the fact that $\thq^{-1} \geq \mq \geq \lambda_q^{-5}$ we obtain
\begin{align}\label{e2.47}
&\norm{\partial_t^M \nabla^N ( \chi_i (1-\chi_i) (v_i-v_{i-1}) \mathring\otimes(v_i-v_{i-1}))}_{L^\infty_tH^3_x} \notag\\
\lesssim & \sum_{M_1+M_2=M}  \norm{\partial_t^{M_1}( \chi_i (1-\chi_i)}_{L^\9_t}\norm{\partial_t^{M_2} \nabla^N(v_i-v_{i-1}) \mathring\otimes(v_i-v_{i-1}))}_{L^\infty_t(\supp(\chi_i\chi_{i-1});H^3_x)} \notag\\
\lesssim & \sum_{M_1+M_2=M}  \thq^{-M_1} \norm{\partial_t^{M_2} \nabla^N ( (v_i-v_{i-1}) \otimes(v_i-v_{i-1}))}_{L^\infty_t(\supp(\chi_i\chi_{i-1});H^3_x)} \notag\\
 \lesssim & \sum_{M_1+M_2=M}  \thq^{-M_1}\mq^{\frac{N}{2\alpha}+M_2 } \lambda_q^{10}
 \lesssim \thq^{-M-1} \mq^{\frac{N}{2\alpha} }\lambda_q^{5} ,
\end{align}
where the implicit constants are independent of $i$ and $q$.
Combing \eqref{e2.46} and \eqref{e2.47} we prove \eqref{nrh3}.

It remains to prove \eqref{suppnr} and the well-preparedness of the
$(\wt u_q,\tr_q)$.
By the choice of $m_q$ and $\theta_q$
in \eqref{def-mq-thetaq},
\begin{align}\label{thetaq1}
\thq= \lambda_{q}^{-\frac{12}{\eta}}\ll \frac12\lambda_{q-1}^{-\frac{12}{\eta}}=\frac12\theta_q.
\end{align}
Since $(u_q,\rr_q)$ is a well-prepared solution to \eqref{equa-nsr}
for the set $I_{q}$ and the length scale $\theta_q$,
we note that $u_q$ is an exact solution to \eqref{equa-NS} on the $\theta_{q}$-neighborhood of $I^c_q$.
In particular, if $t_{i-1}$ and $t_{i}$ both lie in this region,
we have that $\wt u_q =v_{i-1}=v_{i}=u_q$ on the overlapped region $\supp \chi_{i-1}\chi_{i}$,
thus $\wt u_q$ is an exact solution there.

Based on the above argument, we define the index set $\mathcal{C}$ by
\begin{align}\label{def-indexsetb}
\mathcal{C}:=\left\{ i\in \mathbb{Z}: 1\leq i\leq m_{q+1}-1\  \text{and}\  \rr_q\not\equiv 0\ \text{on}\ [t_{i-1},t_{i}+\thq]\cap [0,T] \right\},
\end{align}
and choose $I_{q+1}$ in the way
\begin{align}   \label{Iq1-C-def}
  I_{q+1} := \bigcup_{i\in \mathcal{C}} \left[t_i-2\thq,t_i+3\thq \right].
\end{align}

We claim that for any $q\geq 0$,
\begin{align}\label{iq1}
I_{q+1}\subseteq I_q.
\end{align}

To this end, it is equivalent to show that for any $t_*\in I_{q}^c$,
it holds that $t_*\in I_{q+1}^c$.
We argue by contradiction.
Suppose that $t_*\in I_{q+1}$, then
\begin{align}\label{e3.56}
   \rr_q \not\equiv 0\quad \text{on}\quad [t_*-2(T/\mq),t_*+2(T/\mq)+\thq].
\end{align}
Since $t_*\in I_{q}^c$ and $2T/\mq+\thq \ll \theta_q/2$,
we infer that
\begin{align*}
  \dist (t,I_{q}^c)\leq \frac{\theta_q}{2}, \ \
  \forall t\in [t_*-2(T/\mq),t_*+2(T/\mq)+\thq],
\end{align*}
which along with the well-preparedness of $(u_q,\rr_q)$ yields that
\begin{align*}
  \rr_q(t)=0\quad \text{for all}\quad  t\in [t_*-2(T/\mq),t_*+2(T/\mq)+\thq].
\end{align*}
This leads to the contradiction with \eqref{e3.56}.
Thus, \eqref{iq1} is proved.

We next prove \eqref{suppnr}, i.e.,
\begin{align*}
	\tr_q(t,x)=0 \quad \text{if} \quad \operatorname{dist}(t,I_{q+1}^c)\leq {3}\thq/2,
\end{align*}

For this purpose,
we take any $t\in [0,T]$ such that $\dist(t, I_{q+1}^c)\leq {3}\thq/2$.

If $t\in I_{q+1}$,
by \eqref{Iq1-C-def},
$t\in [t_i-2\thq,t_i+3\thq]$ for some $i\in \mathcal{C}$.
Then, we have
\begin{align}\label{e3.57}
t\in [t_i-2\thq,t_i-\frac{\thq}{2}]\quad \text{or}\quad t\in [t_i+\frac{3\thq}{2},t_i+3\thq].
\end{align}
For $t\in [t_i-2\thq,t_i-\thq/2]$,
since  $\chi_{i-1}(t) =1$ and $\partial_t \chi_{i-1}(t)=0$,
we infer from \eqref{def-nr} with $i-1$ replacing $i$ that
$\tr_q(t) = 0$.
Similarly arguments also apply to the case where $t\in [t_i+3\thq/2,t_i+3\thq]$.

If $t\in I_{q+1}^c$, then there exists $0\leq j\leq \mq-1$, such that $t\in [t_j,t_{j+1}]$.
If $t\in [t_j+\thq,t_{j+1}]$,
since $\chi_{j}(t) =0$, $\partial_t\chi_j(t) = 0$,
we still have from \eqref{def-nr} with $j$ replacing $i$ that
$\tr_q(t) =0$.
For $t\in [t_j,t_{j}+\thq]$,
we have that $j\notin \mathcal{C}$,
otherwise $t\in [t_j-2\thq,t_j+3\thq]\subseteq I_{q+1}$.
Hence, it follows from \eqref{def-indexsetb} that
\begin{align*}
\rr_q\equiv0 \quad \text{on}\quad [t_{j-1},t_j+\thq].
\end{align*}
This means that $u_q$ solves solution to \eqref{equa-NS} on $[t_{j-1},t_j+\thq]$.
Thus, by the uniqueness in Proposition  \ref{Prop-LWP-Hyper-NLSE},
\begin{align*}
v_{j-1}=v_j=u_q\ \ on\ [t_j, t_j+\theta_{q+1}].
\end{align*}
Plugging this into \eqref{def-nr} with $j$ replacing $i$
we obtain $\tr_q(t) =0 $
and thus finish the proof of \eqref{suppnr}.

Therefore, the proof of Proposition~\ref{prop-nunr} is complete.
\hfill $\square$

\section{Velocity perturbations in the supercritical regime $\mathcal{A}_1$}  \label{Sec-Flow-Endpt1}

In this section, we mainly construct the crucial velocity perturbations
whose borderline in particular includes the endpoint case
where $(s,\gamma, p)=(3/p+1-2\alpha, \infty, p)$,
$\alpha \in [5/4, 2)$.

The fundamental building blocks are indexed by six parameters
$\rs$, $\rp$, $\lambda$, $\mu$, $\tau$ and $\sigma$,
chosen in the following way
\begin{equation}\label{larsrp}
	\rs := \lambda_{q+1}^{-1+2\varepsilon},\ \rp := \lambda_{q+1}^{-1+4\varepsilon},\
	 \lambda := \lambda_{q+1},\ \mu:=\lambda_{q+1}^{2\a-1+2\varepsilon}, \
      \tau:=\lambda_{q+1}^{4\a-5+11\varepsilon}, \ \sigma:=\lambda_{q+1}^{2\varepsilon},
\end{equation}
where $\varepsilon$ is a sufficiently small constant satisfying \eqref{e3.1}.

We note that,
the parameters are chosen in this way,
in order to control both the strong dissipativity
and oscillation errors as shown in the constrains \eqref{setpw}-\eqref{setrosc2}.

\subsection{Spatial building blocks.}
We use the intermittent jets, first introduced in \cite{bcv21},
as the basic spatial building blocks.
More precisely,
we let $\Phi : \mathbb{R}^2 \to \mathbb{R}$ be a smooth cut-off function supported on a ball of radius $1$
and normalize $\Phi$ such that $\phi := - \Delta\Phi$ satisfies
\begin{equation}\label{e4.91}
	\frac{1}{4 \pi^2}\int_{\mathbb{R}^2} \phi^2(x)\d x = 1.
\end{equation}
Moreover, let $\psi: \mathbb{R} \rightarrow \mathbb{R}$ be a smooth and mean-zero function, satisfying
\begin{equation}\label{e4.92}
	\frac{1}{2 \pi} \int_{\mathbb{R}} \psi^{2}\left(x\right) \d x=1, \quad \supp\psi\subseteq [-1,1].
\end{equation}

The corresponding rescaled cut-off functions are defined by
\begin{equation*}
	\phi_{\rs}(x) := {\rs^{-1}}\phi\left(\frac{x}{\rs}\right), \quad
	\Phi_{\rs}(x):=   {\rs^{-1}} \Phi\left(\frac{x}{\rs}\right),\quad
	\psi_{\rp}\left(x\right) := {r_{\|}^{- \frac 12}} \psi\left(\frac{x}{r_{\|}}\right).
\end{equation*}
With this scaling, $\phi_{\rs}$ is supported in the ball of radius $\rs$ in $\R^2$ and $\psi_{\rp}$ is supported in the ball of radius $\rp$ in $\bbr$. By an abuse of notation,
we periodize $\phi_{\rs}$, $\Phi_{\rs}$ and $\psi_{\rp}$ so that
$\phi_{\rs}$, $\Phi_{\rs}$ are treated as periodic functions defined on $\mathbb{T}^2$ and $\psi_{\rp}$ is treated as a periodic function defined on $\mathbb{T}$.

\vspace{1ex}

Let  $\Lambda \subset \mathbb{S}^2 \cap \mathbb{Q}^3$ be the wavevector set as in the Geometrical Lemma ~\ref{geometric lem 2},
and let $(k,k_1,k_2)$ be the orthonormal bases for every $k\in \Lambda$.

The \textit{intermittent jets} are defined by
\begin{equation*}
	W_{(k)} :=  \psi_{\rp}(\lambda \rs N_{\Lambda}(k_1\cdot x+\mu t))\phi_{\rs}( \lambda \rs N_{\Lambda}k\cdot (x-\alpha_k),\lambda \rs N_{\Lambda}k_2\cdot (x-\alpha_k))k_1,\ \  k \in \Lambda.
\end{equation*}

Here, $N_{\Lambda}$ is given by \eqref{NLambda}.
The parameters $\rp$ and $\rs$ measure the concentration effect of the intermittent jets,
and $\mu$ is the temporal oscillation parameter.
The shifts $\a_k\in \R^3$ are chosen suitably  such that $W_{(k)}$ and $W_{(k')}$ have disjoint supports if $k\neq k'$.
The existence of such $\a_k$ can be guaranteed by taking $\rs$ sufficiently small
(see, e.g., \cite{bv19r}).

For brevity, we set
\begin{equation}\label{snp}
	\begin{array}{ll}
		&\psi_{(k_1)}(x) :=\psi_{\rp}(\lambda \rs N_{\Lambda}(k_1\cdot x+\mu t)), \\
		&\phi_{(k)}(x) := \phi_{\rs}( \lambda \rs N_{\Lambda}k\cdot (x-\a_k),\lambda \rs N_{\Lambda}k_2\cdot (x-\a_k)), \\
		&\Phi_{(k)}(x) := \Phi_{\rs}( \lambda \rs N_{\Lambda}k\cdot (x-\a_k),\lambda \rs N_{\Lambda}k_2\cdot (x-\a_k)),
		\end{array}
\end{equation}
and thus
\begin{equation}\label{snwd}
	W_{(k)} = \psi_{(k_1)}\phi_{(k)} k_1,\quad  k\in \Lambda.
\end{equation}

Because $W_{(k)}$ is not divergence-free, we also need the corrector
\begin{equation}
	\begin{aligned}
		\label{corrector vector}
		\wt W_{(k)}^c := \frac{1}{\lambda^2N_{ \Lambda }^2} \nabla\psi_{(k_1)}\times\curl(\Phi_{(k)} k_1)
	\end{aligned}
\end{equation}
and let
\begin{align} \label{Vk-def}
	W^c_{(k)} := \frac{1}{\lambda^2N_{\Lambda}^2 } \psi_{(k_1)}\Phi_{(k)} k_1.
\end{align}
Then, by straightforward computations,
\begin{equation}\label{wcwc}
	  W_{(k)} + \wt W_{(k)}^c
	=\curl \curl \left(\frac{1}{\lambda^2N_{\Lambda}^2 } \psi_{(k_1)}\Phi_{(k)} k_1\right)
	=\curl \curl W^c_{(k)},
\end{equation}
which yields that
\begin{align} \label{div-Wck-Wk-0}
	\div (W_{(k)}+ \wt W^c_{(k)}) =0.
\end{align}

Lemma \ref{buildingblockestlemma} below contains the key estimates of the intermittent jets.

\begin{lemma} [Estimates of intermittent jets, \cite{bv19r}] \label{buildingblockestlemma}
	For $p \in [1,\infty]$, $N,\,M \in \mathbb{N}$, we have
	\begin{align}
		&\left\|\nabla^{N} \partial_{t}^{M} \psi_{(k_1)}\right\|_{C_t L^{p}_{x}}
		\lesssim r_{\|}^{\frac 1p- \frac 12}\left(\frac{r_{\perp} \lambda}{r_{\|}}\right)^{N}
          \left(\frac{r_{\perp} \lambda \mu}{r_{\|}}\right)^{M}, \label{intermittent estimates} \\
		&\left\|\nabla^{N} \phi_{(k)}\right\|_{L^{p}_{x}}+\left\|\nabla^{N} \Phi_{(k)}\right\|_{L^{p}_{x}}
		\lesssim r_{\perp}^{\frac 2p- 1}  \lambda^{N}, \label{intermittent estimates2}
	\end{align}
	where the implicit constants are independent of $\rs,\,\rp,\,\lambda$ and $\mu$. Moreover, it holds that
	\begin{align}
		&\displaystyle\left\|\nabla^{N} \partial_{t}^{M} W_{(k)}\right\|_{C_t  L^{p}_{x}}
          +\frac{r_{\|}}{r_{\perp}}\left\|\nabla^{N} \partial_{t}^{M} \wt W_{(k)}^{c}\right\|_{C_t L^{p}_{x}}
          +\lambda^{2}\left\|\nabla^{N} \partial_{t}^{M} W_{(k)}^c\right\|_{C_t L^{p}_{x}}\displaystyle \nonumber \\
		&\qquad \lesssim r_{\perp}^{\frac 2p- 1} r_{\|}^{\frac 1p- \frac 12} \lambda^{N}
          \left(\frac{r_{\perp} \lambda \mu}{r_{\|}}\right)^{M}, \ \ k\in \Lambda, \label{ew}
	\end{align}
where the implicit constants are independent of $\rs,\,\rp,\,\lambda$ and $\mu$.
\end{lemma}

\subsection{Temporal building blocks.}
As already seen in the previous sections,
the space $C_tL^2_x$ is not a suitable candidate
to construct non-unique solutions for the hyperdissipative NSE \eqref{equa-NS}
when $\alpha \geq 5/4$,
due to the well-posedness results \cite{lions69}.

The keypoint in the current case
is to exploit the temporal intermittency
through suitable temporal building blocks.
These temporal building blocks are indexed by two additional parameters $\tau$ and $\sigma$,
which particularly parameterize the concentration and oscillation in time.
In particular,
in view of the choice \eqref{larsrp},
the strength of temporal oscillation is proportionate to the viscosity of the fluid,
and almost provides the 3D spatial intermittency
when the high dissipativity exponent $\alpha$ is close to $2$.

More precisely, as in \cite{cl20.2},
we choose $g\in C_c^\infty([0,T])$ as a cut-off function such that
\begin{align*}
	\aint_{0}^T g^2(t) \d t=1,
\end{align*}
and then rescale $g$ by
\begin{align}\label{gk1}
	g_\tau(t)=\tau^{\frac 12} g(\tau t),
\end{align}
where the parameter $\tau$ is given by \eqref{larsrp}. By an abuse of notation, we periodic $g_\tau$ such that it is treated as a periodize function defined on $[0,T]$.
Then, let
\begin{align} \label{hk}
	h_\tau(t):= \int_{0}^t \left(g_\tau^2(s)  - 1\right)\ ds,\ \ t\in [0,T],
\end{align}
and set
\begin{align}\label{gk}
	\g:=g_\tau(\sigma t),\ \
    h_{(\tau)}(t):= h_\tau(\sigma t),
\end{align}
when $\sigma$ is as in \eqref{larsrp}.
We note that, $h_{(\tau)}$ satisfies
\begin{align}   \label{pt-h-gt}
\p_t(\sigma^{-1} h_{(\tau)}) = \g^2-1=\g^2-\aint_{0}^T \g^2(t) \d t.
\end{align}
It will be used in the construction of the temporal corrector $\wo$,
in order to balance the high temporal frequency error in \eqref{mag oscillation cancellation calculation}.

We recall from \cite{cl21,lzz21}
the crucial estimates of $\g$ and $h_{(\tau)}$,
which are the contents of Lemma \ref{Lem-gk-esti} below.

\begin{lemma} [Estimates of temporal intermittency]   \label{Lem-gk-esti}
	For  $\gamma \in [1,\infty]$, $M  \in \mathbb{N}$,
	we have
	\begin{align}
		\left\|\partial_{t}^{M}\g \right\|_{L^{\gamma}_t} \lesssim \sigma^{M}\tau^{M+\frac12-\frac{1}{\gamma}},\label{gk estimate}
	\end{align}
	where the implicit constants are independent of $\sigma$ and $\tau$.
	Moreover, we have
	\begin{align}\label{hk-esti}
	\|h_{(\tau)}\|_{C_t}\leq 1.
	\end{align}
\end{lemma}

\begin{remark}
Let us mention that,
the temporal building blocks for MHD in \cite{lzz21} provide the almost 1D intermittency, namely,
$\tau\sim\lambda_{q+1}^{1-}$.
Here, for the hyperdissipative NSE when $\a$ is close to $2$,
we need to oscillate the temporal building blocks in
a much larger frequency, which provides the almost 3D spatial intermittency, i.e.,
$\tau\sim\lambda_{q+1}^{3-}$.
\end{remark}

\subsection{Velocity perturbations}  \label{Subsec-Velo-perturb}

Below we construct the velocity perturbation,
which mainly consist of the principal perturbation, the incompressibility corrector and two temporal correctors.

For this purpose,
let us first construct the amplitudes of perturbations,
which are important to provide the cancellation
between the low frequency part of the nonlinearity and the old Reynolds stress.\\

\paragraph{\bf Amplitudes}
Let $\chi: [0, \infty) \to \mathbb{R}$ be a smooth cut-off function satisfying
\begin{equation}\label{e4.0}
	\chi (z) =
	\left\{\aligned
	& 1,\quad 0 \leq z\leq 1, \\
	& z,\quad z \geq 2,
	\endaligned
	\right.
\end{equation}
and
\begin{equation}\label{e4.1}
	\frac 12 z \leq \chi(z) \leq 2z \quad \text{for}\quad z \in (1,2).
\end{equation}

Set
\begin{equation}\label{rhob}
	\varrho(t,x) := 2 \varepsilon_u^{-1} \la^{- \frac{\ve_R}{4}} \delta_{q+ 1}
	\chi\left( \frac{|\tr_q(t, x) |}{\la^{-\frac{\ve_R}{4}}\delta_{q+1} } \right),
\end{equation}
where $\varepsilon_u>0$ is the small constant as in the Geometric Lemma \ref{geometric lem 2}.
By \eqref{e4.0}, \eqref{e4.1} and \eqref{rhob},
\begin{equation}\label{rhor}
	\left|  \frac{\tr_q}{\varrho} \right|
     = \left| \frac{\tr_q}{2 \varepsilon_u^{-1}\la^{-\frac{\ve_R}{4}}  \delta_{q+ 1}\chi
          ( \la^{\frac{\ve_R}{4}}  \delta_{q+1} ^{-1} |\tr_q | )} \right| \leq \varepsilon_u,
\end{equation}
and for any $p\in[1,\infty]$,
\begin{align}
	\label{rhoblowbound}
	&\varrho\geq  \varepsilon_u^{-1} \la^{-\frac{\ve_R}{4}}\delta_{q+ 1},\\
	\label{rhoblp}
	&\norm{ \varrho }_{L^p_{t,x}} \lesssim  \varepsilon^{-1}_u ( \la^{-\frac{\ve_R}{4}} \delta_{q+1}  + \norm{\tr_q}_{L^p_{t,x}} ).
\end{align}
Moreover, by \eqref{nrh3}, \eqref{rhoblowbound} and the standard H\"{o}lder estimate (see \cite[(130)]{bdis15}), for $1\leq N\leq 9$,
\begin{align}
	& \norm{ \varrho }_{C_{t,x}} \lesssim  \thq^{-2} , \quad   \norm{ \varrho }_{C_{t,x}^N}  \lesssim \thq^{-3N}, \label{rhoB-Ctx.1}\\
	&\norm{ \varrho^{1/2}}_{C_{t,x}} \lesssim \thq^{-1}, \quad  \norm{  \varrho^{1/2} }_{C_{t,x}^N} \lesssim \thq^{-3N},  \label{rhoB-Ctx.2} \\
	&\norm{ \varrho^{-1}}_{C_{t,x}} \lesssim \theta_{q+1}^{-1}, \quad\norm{ \varrho^{-1} }_{C_{t,x}^N } \lesssim \thq^{-3N},  \label{rhoB-Ctx.3}
\end{align}
where the implicit constants are independent of $\la$, $\delta_{q+1}$ and $\thq$.

In order to guarantee the temporal support of the perturbations
to be compatible with that of the concentrated Reynolds stress $\tr_q$
in \S \ref{Sec-Concen-Rey},
we use the  smooth temporal cut-off function $f: [0,T]\rightarrow [0,1]$, satisfying
\begin{itemize}
	\item $0\leq f\leq 1$ and $f \equiv 1$ on $\supp_t \tr_q$;
	\item $\supp_t f\subseteq N_{\thq/2}(\supp_t \tr_q )$;
	\item $\|f \|_{C_t^N}\lesssim \thq^{-N}$,\ \  $1\leq N\leq 9$.
\end{itemize}

Now, we define the amplitudes of the perturbations by
\begin{equation}\label{akb}
	a_{(k)}(t,x):= \varrho^{\frac{1}{2} } (t,x) f (t)\gamma_{(k)}
       \left(\Id-\frac{\tr_q(t,x)}{\varrho(t,x)}\right), \quad k \in \Lambda,
\end{equation}
where $\gamma_{(k)}$ and $\Lambda$ are as in is the Geometric Lemma~\ref{geometric lem 2}.

Applying the Geometric Lemma~\ref{geometric lem 2}
and using the expression \eqref{akb}
we have the following algebraic identity,
which enables us to reduce the effect of the concentrated Reynolds stress
(see also \eqref{mag oscillation cancellation calculation} below)
\begin{align}\label{magcancel}
  \sum\limits_{ k \in  \Lambda} a_{(k)}^2 \g^2
 W_{(k)} \otimes W_{(k)}
	= & \varrho f^2 {\rm Id} -\tr_q
	+  \sum\limits_{ k \in \Lambda}  a_{(k)}^2\g^2\P_{\neq 0}(  W_{(k)} \otimes W_{(k)} ) \notag\\
	& + \sum_{k \in \Lambda}  a_{(k)}^2 (\g^2-1) \aint_{\T^3}W_{(k)}\otimes W_{(k)}\d x  ,
\end{align}
where $\P_{\neq 0}$ denotes the spatial projection onto nonzero Fourier modes.
We also have the analytic estimates of the amplitudes below,
the proof is similar to that of \cite{lzz21}.

\begin{lemma} [Estimates of amplitudes, \cite{lzz21}] \label{mae-endpt1}
For $1\leq N\leq 9$, $k\in \Lambda$, we have
	\begin{align}
		\label{e3.15}
		&\norm{a_{(k)}}_{L^2_{t,x}} \lesssim \delta_{q+1}^{\frac{1}{2}} ,\\
		\label{mag amp estimates}
		& \norm{ a_{(k)} }_{C_{t,x}} \lesssim \thq^{-1},\ \ \norm{ a_{(k)} }_{C_{t,x}^N} \lesssim \thq^{-7N},
	\end{align}
where the implicit constants are independent of $q$.
\end{lemma}

\paragraph{\bf Velocity perturbations}
We are now in stage to construct the velocity perturbations.
In the sequel, we will define their principal part, the incompressibility corrector
and two types of temporal correctors.

First, the principal part $w_{q+1}^{(p)}$ of the velocity perturbations is defined by
\begin{align}  \label{pv}
		w_{q+1}^{(p)} &:= \sum_{k \in \Lambda } a_{(k)}\g W_{(k)},
\end{align}
where $a_{(k)}$ is the amplitude given by \eqref{akb}
and $ W_{(k)}$, $g_{(\tau)}$ are the spatial and temporal building blocks,
respectively,
constructed in \eqref{snwd} and \eqref{gk}.

The important fact here is that, by \eqref{magcancel},
the effect of the concentrated Reynolds stress can be reduced by
the zero frequency part of $w_{q+ 1}^{(p)} \otimes w_{q+ 1}^{(p)}$:
\begin{align} \label{mag oscillation cancellation calculation}
 w_{q+ 1}^{(p)} \otimes w_{q+ 1}^{(p)} + \tr_{q}
	=& \varrho f^2 {\rm Id}+ \sum_{k \in \Lambda }  a_{(k)}^2 \g^2 \P_{\neq 0}(W_{(k)}\otimes W_{(k)})  \notag\\
	&+ \sum_{k \in \Lambda }  a_{(k)}^2 (\g^2-1) \aint_{\T^3}W_{(k)}\otimes W_{(k)}\d x.
\end{align}

Because the principal part of perturbation is not divergence free,
we need the incompressibility corrector defined by
	\begin{align}  \label{wqc-dqc}
		w_{q+1}^{(c)}
		&:=   \sum_{k\in \Lambda } \g\left(\curl (\nabla a_{(k)} \times W^c_{(k)})
		+ \nabla a_{(k)} \times \curl W^c_{(k)} +a_{(k)} \wt W_{(k)}^c \right) ,
	\end{align}
where $W^c_{(k)}$ and  $\wt W_k^c $  are given by \eqref{Vk-def} and \eqref{corrector vector}, respectively.
Note that, one has
	\begin{align} \label{div free velocity}
		&  w_{q+1}^{(p)} + w_{q+1}^{(c)}
		=\curl \curl \left(  \sum_{k \in \Lambda} a_{(k)} \g W^c_{(k)} \right),
	\end{align}
and thus,
\begin{align} \label{div-wpc-dpc-0}
	\div (w_{q+1}^{(p)} + w_{q +1}^{(c)})= 0.
\end{align}

Furthermore,
in order to balance
the high  spatial and temporal frequency errors in \eqref{mag oscillation cancellation calculation},
we introduce another two types of temporal correctors as follows:

{\bf $\bullet$ Temporal corrector to balance spatial oscillations.}
We define the temporal corrector $w_{q+1}^{(t)}$ by
	\begin{align}   \label{veltemcor}
		&w_{q+1}^{(t)} := -{\mu}^{-1}  \sum_{k\in \Lambda } \P_{H}\P_{\neq 0}(a_{(k)}^2\g^2\psi_{(k_1)}^2 \phi_{(k)}^2 k_1).
	\end{align}
It is introduced mainly to handle the high spatial frequency oscillations
in \eqref{mag oscillation cancellation calculation},
namely, by Leibniz's rule,
\begin{align} \label{utem}
	\partial_{t} w_{q+1}^{(t)}+    \sum_{k \in \Lambda}  \P_{\neq 0}
	\(a_{(k)}^{2}\g^2 \div(W_{(k)} \otimes W_{(k)})\)
	=&(\nabla\Delta^{-1}\div)  {\mu}^{-1}  \sum_{k \in \Lambda }  \P_{\neq 0} \partial_{t}
	\(a_{(k)}^{2}\g^2 \psi_{(k_1)}^{2} \phi_{(k)}^{2}  k_1\)   \nonumber  \\
	& - {\mu}^{-1}    \sum_{k \in \Lambda }   \P_{\neq 0}
\(\partial_{t}( a_{(k)}^{2}\g^2)  \psi_{(k_1)}^{2} \phi_{(k)}^{2} k_1\).
\end{align}

We note that,
$\partial_t w^{(t)}_{q+1}$ cancels the large spatial frequency of
$\div (W_{(k)} \otimes W_{(k)})$ on the left-hand-side above.
The remaining terms include the low spatial frequency part
$\partial_{t}(a_{(k)}^{2}\g^2)$ and the pressure term,
where the latter can be removed later by applying
the Helmholtz-Leray projector $\P_{H}$.

{\bf $\bullet$ Temporal corrector to balance temporal oscillations.}
Another type of the temporal corrector is defined by
\begin{align}
		& w_{q+1}^{(o)}:= -\sigma^{-1}\sum_{k\in\Lambda }\P_{H}\P_{\neq 0}\(h_{(\tau)}\aint_{\T^3} W_{(k)}\otimes W_{(k)}\d x\nabla (a_{(k)}^2) \)   ,\label{wo}
\end{align}
which is used to balance the high temporal frequency oscillations
in \eqref{mag oscillation cancellation calculation},
due to the presence of the temporal oscillation function $\g$.
More precisely, by \eqref{hk}, \eqref{wo} and the Leibniz rule,
\begin{align} \label{utemcom}
	&\partial_{t} w_{q+1}^{(o)}+
	\sum_{k\in \Lambda}
	\P_{\neq 0}\( (\g^2-1 )\aint_{\T^3}W_{(k)}\otimes W_{(k)}\d x \nabla(a_{(k)}^{2}) \) \nonumber  \\
	=&\left(\nabla\Delta^{-1}\div\right) \sigma^{-1}  \sum_{k \in \Lambda} \P_{\neq 0} \partial_{t}\(h_{(\tau)}\aint_{\T^3}W_{(k)}\otimes W_{(k)}\d x\nabla (a_{(k)}^2)\) \nonumber  \\
	&-\sigma^{-1}\sum_{k\in \Lambda}\P_{\neq 0}\(h_{(\tau)}\aint_{\T^3}W_{(k)}\otimes W_{(k)}\d x\p_t\nabla (a_{(k)}^2)\).
\end{align}

We note that,
moduling the harmless pressure term,
the right-hand-side contains the low frequency part
$\p_t\nabla (a_{(k)}^2)$,
which is acceptable in the convex integration scheme.
Hence, the high temporal frequencies in the term $(\g^2-1)$
have been cancelled with the help of the temporal corrector $\partial_{t} w_{q+1}^{(o)}$.

Now, we define the velocity perturbation $w_{q+1}$ at level $q+1$  by
	\begin{align}
		w_{q+1} &:= w_{q+1}^{(p)} + w_{q+1}^{(c)}+ w_{q+1}^{(t)}+\wo.
		\label{velocity perturbation}
	\end{align}
By the constructions above,
$w_{q+1}$ is mean-free and divergence-free.

The velocity field at level $q+1$ is then defined by
	\begin{align}
		& u_{q+1}:= \wt u_{q} + w_{q+1},
		\label{q+1 velocity}
	\end{align}
where $\wt u_q$ is defined in \eqref{def-wtu} by the gluing stage.

We summarize the crucial estimates of velocity perturbations in Lemma \ref{totalest} below.

\begin{lemma}  [Estimates of velocity perturbations] \label{totalest}
	For any $\rho \in(1,\infty), \gamma \in [1,\infty]$ and integers $0\leq N\leq 7$, we have the following estimates:
	\begin{align}
		&\norm{\na^N w_{q+1}^{(p)} }_{L^ \gamma_tL^\rho_x }  \lesssim \thq^{-1} \lbb^N\rs^{\frac{2}{\rho}-1}\rp^{\frac{1}{\rho}-\frac12}\tau^{\frac12-\frac{1}{ \gamma}},\label{uprinlp-endpt1}\\
		&\norm{\na^N w_{q+1}^{(c)} }_{L^\gamma_tL^\rho_x   } \lesssim \thq^{-1}\lbb^N\rs^{\frac{2}{\rho}}\rp^{\frac{1}{\rho}-\frac{3}{2}}\tau^{\frac12-\frac{1}{\gamma}}, \label{ucorlp-endpt1}\\
		&\norm{ \na^Nw_{q+1}^{(t)} }_{L^\gamma_tL^\rho_x   }\lesssim \thq^{-2}\lbb^N\mu^{-1}\rs^{\frac{2}{\rho}-2}\rp^{\frac{1}{\rho}-1}\tau^{1-\frac{1}{\gamma}} ,\label{dco rlp-endpt1}\\
		&\norm{\na^N \wo }_{L^\gamma_tL^\rho_x  }\lesssim \thq^{-7N-9}\sigma^{-1},\label{dcorlp-endpt1}
	\end{align}
where the implicit constants depend only on $N$, $\gamma$ and $\rho$. In particular, for integrals $1\leq N\leq 7$, we have,
\begin{align}
& \norm{ w_{q+1}^{(p)} }_{L^\9_tH^N_x }  + \norm{ w_{q+1}^{(c)} }_{L^\9_tH^N_x}+\norm{ w_{q+1}^{(t)} }_{L^\9_tH^N_x}+\norm{ \wo }_{L^\9_tH^N_x}
          \lesssim \lambda^{N+2},\label{principal h3 est-endpt1}\\
& \norm{\p_t w_{q+1}^{(p)} }_{L^\9_tH^N_x }  + \norm{\p_t w_{q+1}^{(c)} }_{L^\9_tH^N_x}+\norm{\p_t w_{q+1}^{(t)} }_{L^\9_tH^N_x}+\norm{\p_t \wo }_{L^\9_tH^N_x}
           \lesssim \lambda^{N+6},\label{pth2 est-endpt1}
\end{align}
where the implicit constants are independent of $\lbb$.
\end{lemma}

\begin{proof}
First,
using \eqref{ew}, \eqref{gk estimate}, \eqref{pv} and Lemma~\ref{mae-endpt1}
we have that for any $\rho \in (1,\infty)$,
	\begin{align*}
			 \norm{\nabla^N w_{q+1}^{(p)} }_{L^\gamma_tL^\rho_x }
			\lesssim&  \sum_{k \in \Lambda}
             \sum\limits_{N_1+N_2 = N}
            \|a_{(k)}\|_{C^{N_1}_{t,x}}\|\g\|_{L_t^\gamma}
			\norm{ \nabla^{N_2} W_{(k)} }_{C_tL^\rho_x } \notag \\
		\lesssim&	 \thq^{-1}\lbb^N\rs^{\frac{2}{\rho}-1}\rp^{\frac{1}{\rho}-\frac12}\tau^{\frac12-\frac{1}{\gamma}},
	\end{align*}
which verifies \eqref{uprinlp-endpt1}.

Moreover, by \eqref{b-beta-ve}, \eqref{def-mq-thetaq}, \eqref{ew},
\eqref{wqc-dqc} and Lemma \ref{mae-endpt1},
	\begin{align*}
		 \quad \norm{\na^N w_{q+1}^{(c)} }_{L^\gamma_tL^\rho_x}
		\lesssim&
		\sum\limits_{k\in \Lambda     }\|\g\|_{L^\gamma_t}  \sum_{N_1+N_2=N}
         \( \norm{ a_{(k)} }_{C_{t,x}^{N_1+2}} \norm{\na^{N_2} W^c_{(k)}}_{C_tW^{1,\rho}_x }
		 +  \norm{ a_{(k)} }_{C_{t,x}^{N_1}} \norm{ \na^{N_2}\wt W^c_{(k)}}_{C_tL^\rho_x }  \)  \nonumber   \\
		 \lesssim & \sum_{N_1+N_2=N} \tau^{\frac12-\frac{1}{\gamma}} ( \thq^{-7N_1-14}\lambda^{N_2-1}r_{\perp}^{\frac{2}{\rho} - 1} r_{\parallel}^{\frac{1}{\rho} - \frac{1}{2}}
           + \thq^{-7N_1-1}\lambda^{N_2}\rs^{\frac{2}{\rho} } \rp^{\frac{1}{\rho} - \frac{3}{2}})\notag \\
		 \lesssim& \tau^{\frac12-\frac{1}{\gamma}}(\thq^{-14}\lambda^{N-1}r_{\perp}^{\frac{2}{\rho} - 1} r_{\parallel}^{\frac{1}{\rho}
                - \frac{1}{2}}+\thq^{-1}\lambda^{N}\rs^{\frac{2}{\rho}} \rp^{\frac{1}{\rho} - \frac{3}{2}}) \notag \\
		 \lesssim&  \thq^{-1}\lambda^{N}\rs^{\frac{2}{\rho}} \rp^{\frac{1}{\rho} - \frac{3}{2}}\tau^{\frac12-\frac{1}{\gamma}}.
	\end{align*}
	Thus, we obtain \eqref{ucorlp-endpt1}.
	
	Concerning the temporal correctors,
	in view of \eqref{mag amp estimates},
	\eqref{veltemcor}, Lemmas \ref{buildingblockestlemma}, \ref{Lem-gk-esti}, \ref{mae-endpt1}
and the boundedness of operators $\P_{\not =0}$ and $\mathbb{P}_H$ in $L^\rho$,
$\rho \in (1,\infty)$, we infer that
\begin{align*}
			\norm{\na^N w_{q+1}^{(t)} }_{L^\gamma_tL^\rho_x}
			\lesssim & \,\mu^{-1}    \sum_{k \in \Lambda}
			\norm{ \g }_{L^{2\gamma}_t }^2 \sum_{N_1+N_2+N_3=N}
              \|\nabla^{N_1}(a_{(k)}^2)\|_{C_{t,x} }\norm{  \na^{N_2} (\psi_{(k_1)}^2) }_{C_tL^{\rho}_x }\norm{  \na^{N_3}(\phi_{(k)}^2) }_{L^{\rho}_x } \notag  \\
			\lesssim & \,\mu^{-1}  \tau^{1-\frac{1}{\gamma}}
             \sum_{N_1+N_2+N_3=N} \thq^{-7N_1-2} \lambda^{N_2} r_{\parallel}^{\frac{1}{\rho} -1}\lambda^{N_3} r_{\perp}^{\frac{2}{\rho}-2} \notag   \\
			\lesssim & \, \thq^{-2} \lambda^{N}\mu^{-1}r_{\perp}^{\frac{2}{\rho}-2} r_{\parallel}^{\frac{1}{\rho} -1}\tau^{1-\frac{1}{\gamma}},
\end{align*}
which yields \eqref{dco rlp-endpt1}.
	
In view of \eqref{wo}, \eqref{hk-esti} and Lemma \ref{mae-endpt1},
we also have
		\begin{align*}
			\norm{ \na^N \wo }_{L^\gamma_tL^\rho_x }
			\lesssim \sigma^{-1}\sum_{k \in \Lambda}\|h_{(\tau)}\|_{C_{t}} \|\nabla^{N+1} (a^2_{(k)})\|_{C_{t,x}}
			\lesssim  \thq^{-7N-9} \sigma^{-1}.
		\end{align*}

Regarding the $L^\9_tH^N_x$-estimate of velocity perturbations,
using estimates \eqref{uprinlp-endpt1}-\eqref{dcorlp-endpt1} we get
\begin{align*}
& \norm{ w_{q+1}^{(p)} }_{L^\9_tH^N_x }  + \norm{ w_{q+1}^{(c)} }_{L^\9_tH^N_x}+\norm{ w_{q+1}^{(t)} }_{L^\9_tH^N_x}+\norm{ \wo }_{L^\9_tH^N_x}\notag \\
\lesssim &\, \thq^{-1}\lambda^N \tau^{\frac12} +\thq^{-1}\lambda^N r_{\perp} r_{\parallel}^{-1}\tau^{\frac12}
	+ \thq^{-2} \lambda^N \mu^{-1} r_{\perp}^{-1} r_{\parallel}^{-\frac12} \tau + \thq^{-7N-9}\sigma^{-1} \notag\\
\lesssim &\,\thq^{-1}\lambda^{N+2\a-\frac52+6\ve} +\thq^{-1}\lambda^{N+2\a-\frac52+4\ve}+ \thq^{-2} \lambda^{N+2\a-\frac52+5\ve}+ \thq^{-7N-9} \lambda^{-2\ve}\notag \\ \lesssim&\, \lambda^{N+2},
\end{align*}
where the last step is due to \eqref{b-beta-ve} and \eqref{e3.1}.
This verifies \eqref{principal h3 est-endpt1}.

It remains to prove \eqref{pth2 est-endpt1}, using \eqref{b-beta-ve}, \eqref{larsrp} and Lemmas \ref{buildingblockestlemma}, \ref{Lem-gk-esti} and \ref{mae-endpt1}
we get
	\begin{align} \label{wprincipal h2 est}
		\norm{\p_t w_{q+1}^{(p)} }_{L^\9_tH^N_x }
		\lesssim&   \sum_{k \in \Lambda }
		\|a_{(k)}\|_{C_{t,x}^{N+1} }
		 \sum_{ M_1+M_2 =1} \norm{ \p_t^{M_1}\g}_{L^\9_t}\norm{ \p_t^{M_2}W_{(k)} }_{L^\9_tH^N_x}  \notag \\
        \lesssim& \sum_{ M_1+M_2 =1} \thq^{-7N-7}
         \sigma^{M_1} \tau^{M_1+\frac 12}  \lbb^{N} \(\frac{\rs \lbb \mu}{\rp}\)^{M_{2}} \notag \\
        \lesssim &\, \thq^{-7N-7}\lambda^{N+1}\mu \tau^{\frac12}
	\end{align}
	and
	\begin{align} \label{uc h2 est}
		&\quad \norm{\p_t w_{q+1}^{(c)} }_{L^\9_tH^N_x  } \notag \\
		& \lesssim   \sum_{k \in \Lambda }
		\|a_{(k)}\|_{C_{t,x}^{N+3}}
		  \sum_{  M_1+M_2 =1} \norm{ \g}_{C_{t}^{M_1}}
         \(\norm{\p_t^{M_2} W^c_{(k)} }_{L^\9_tH^N_x} + \norm{ \p_t^{M_2} \nabla W^c_{(k)} }_{L^\9_tH^N_x}
		+ \norm{\p_t^{M_2} \wt W^c_{(k)}}_{L^\9_tH^N_x } \)   \nonumber \\
        & \lesssim   \sum_{  M_1+M_2=1}
                     \thq^{-7N-21} \sigma^{M_1} \tau^{M_1+\frac 12}  \(\frac{\rs \lbb \mu}{\rp}\)^{M_2}
                    \(\lbb^{N-2}+\lbb^{N-1} + \frac{\rs}{\rp}\lbb^{N}\) \notag \\
		&    \lesssim \thq^{-7N-21}\lambda^{N+1}\mu r_{\perp} r_{\parallel}^{-1}\tau^{\frac12} .
	\end{align}

	Moreover, using the boundedness of $\mathbb{P}_H \mathbb{P}_{\not =0}$ in $H^N_x$
	and applying  Lemmas \ref{buildingblockestlemma}, \ref{Lem-gk-esti} and \ref{mae-endpt1}
	 we get
	\begin{align} \label{ut h2 est}
		\norm{\p_t w_{q+1}^{(t)} }_{L^\9_tH^N_x  }
         & \lesssim \mu^{-1} \sum_{k \in \Lambda }
           \|\p_t(a_{(k)}^2 g_{(\tau)}^2 \psi^2_{(k_1)} \phi^2_{(k)})\|_{L^\9_tH^N_x } \notag \\
		 & \lesssim \mu^{-1}  \sum_{k \in \Lambda }\|a_{(k)}^2\|_{C_{t,x}^{N+1} }
           \sum_{ M_1+M_2=1} \|\p_t^{M_1}\g^2\|_{L^\9_t}\sum_{0\leq N_1+N_2\leq N}\|\p_t^{M_2}\nabla^{N_1}\psi^2_{(k_1)}\|_{L^\9_tL^2_x }
           \| \nabla^{N_2}\phi^2_{(k)}\|_{L^\9_tL^2_x } \notag \\
        &    \lesssim  \sum_{ M_1+M_2=1} \thq^{-7N-9}\mu^{-1} \sigma^{M_1} \tau^{M_1+1}  \lbb^{N}r_{\perp}^{-1} r_{\parallel}^{-\frac12}  \(\frac{\rs \lbb \mu}{\rp}\)^{M_{2}} \notag\\
        &\lesssim \thq^{-7N-9}\lambda^{N+1}r_{\perp}^{-1} r_{\parallel}^{-\frac12} \tau .
	\end{align}
Arguing as above and using \eqref{b-beta-ve}, \eqref{larsrp},
\eqref{pt-h-gt} and Lemma~\ref{Lem-gk-esti}  we have
	\begin{align} \label{wo h2 est}
		\norm{\p_t w_{q+1}^{(o)} }_{L^\9_tH^N_x  }
        & \lesssim \sigma^{-1} \sum_{k \in \Lambda } \|\p_t (h_{(\tau)} \na (a_{(k)}^2) )\|_{L^\9_tH^N_x} \notag \\
		& \lesssim \sigma^{-1} \sum_{k \in \Lambda }\sum_{M_1+M_2=1} \|\p_t^{M_1}h_{(\tau)}\|_{C_{t} } \|\p_t^{M_2} \nabla (a^2_{(k)})\|_{C_{t,x}^{N}} \notag \\
		& \lesssim \sigma^{-1} \sum_{M_1+M_2=1}\sigma^{M_1}\tau^{M_1}\thq^{-9-7(M_2+N)}\lesssim  \thq^{-7N-9}\tau .
	\end{align}

Therefore, taking into account that $\thq^{-7N-14}\ll \lbb^{\ve}$
and $0<\ve\leq 1/20$,  we conclude that
\begin{align*}
& \norm{\p_t w_{q+1}^{(p)} }_{L^\9_tH^N_x }  + \norm{\p_t  w_{q+1}^{(c)} }_{L^\9_tH^N_x}+\norm{\p_t  w_{q+1}^{(t)} }_{L^\9_tH^N_x}+\norm{ \p_t \wo }_{L^\9_tH^N_x}\notag \\
\lesssim &\, \thq^{-7N-7}\lambda^{N+1}\mu \tau^{\frac12} +\thq^{-7N-21}\lambda^{N+1}\mu r_{\perp} r_{\parallel}^{-1}\tau^{\frac12}
	+ \thq^{-7N-9} \lambda^{N+1} r_{\perp}^{-1} r_{\parallel}^{-\frac12} \tau + \thq^{-7N-9}\tau \notag\\
\lesssim &\,\thq^{-7N-7}\lambda^{N+4\a-\frac52+8\ve} +\thq^{-7N-21}\lambda^{N+4\a-\frac52+6\ve}+ \thq^{-7N-9} \lambda^{N+4\a-\frac52+7\ve}+ \thq^{-7N-9} \lambda^{4\a-5+11\ve} \notag \\
\lesssim&\, \lambda^{N+6},
\end{align*}
which yields \eqref{pth2 est-endpt1}.
Therefore, the proof of Lemma~\ref{totalest} is complete.
\end{proof}

\subsection{Verification of inductive estimates for velocity perturbations}  \label{Subsec-induc-vel-mag}
As a direct consequence of these estimates,
we are now ready to verify the inductive estimates \eqref{uh3}, \eqref{upth2}, \eqref{u-B-L2tx-conv}-\eqref{u-B-Lw-conv}
for the velocity perturbations.

First, in order to
derive the decay of $L^2_{t,x}$-norms of the velocity perturbations,
since $a_{(k)}$ has compact supports on $[0,T]\times \T^3$, it can be regarded as a periodic function on $\T^4$.
We apply the $L^p$ decorrelation Lemma~\ref{Decorrelation1} with $f= a_{(k)}$, $g = \g\psi_{(k_1)}\phi_{(k)}$, $\sigma = \lambda^{2\ve}$
and then use \eqref{la}, \eqref{b-beta-ve} and Lemmas \ref{buildingblockestlemma}, \ref{Lem-gk-esti} and \ref{mae-endpt1} to get
\begin{align}
	\label{Lp decorr vel}
	\norm{w^{(p)}_{q+1}}_{L^2_{t,x}}
	&\lesssim \sum\limits_{k\in \Lambda}
       \Big(\|a_{(k)}\|_{L^2_{t,x}}\norm{ \g }_{L^2_{t}} \norm{ \psi_{(k_1)}\phi_{(k)}}_{C_tL^2_{x}} \notag\\
      &\qquad\qquad +\sigma^{-\frac12}\|a_{(k)}\|_{C^1_{t,x}}\norm{ \g }_{L^2_{t}} \norm{  \psi_{(k_1)}\phi_{(k)}}_{C_tL^2_{x}}\Big) \notag\\
	&\lesssim  \delta_{q+1}^{\frac{1}{2}}+\thq^{-7}\lambda^{-\ve}_{q+1}   \lesssim \delta_{q+1}^{\frac{1}{2}}.
\end{align}
Thus, in view of \eqref{b-beta-ve}, using \eqref{Lp decorr vel} and Lemma~\ref{totalest},
we bound the velocity perturbation by
\begin{align}  \label{e3.41.1}
	\norm{w_{q+1}}_{L^2_{t,x}} &\lesssim\norm{w_{q+1}^{(p)} }_{L^2_{t,x}} + \norm{ w_{q+1}^{(c)} }_{L^2_{t,x}} +\norm{ w_{q+1}^{(t)} }_{L^2_{t,x}}+\norm{ \wo }_{L^2_{t,x}}\notag \\
	&\lesssim \delta_{q+1}^{\frac{1}{2}} +\thq^{-1}r_{\perp} r_{\parallel}^{-1}
	+ \thq^{-2} \mu^{-1} r_{\perp}^{-1} r_{\parallel}^{-\frac12} \tau^\frac 12 + \thq^{-9}\sigma^{-1}\lesssim \delta_{q+1}^{\frac{1}{2}},
\end{align}
and
\begin{align}  \label{wql1.1}
	\norm{w_{q+1}}_{L^1_tL^2_x} &\lesssim\norm{w_{q+1}^{(p)} }_{L^1_tL^2_x} + \norm{ w_{q+1}^{(c)} }_{L^1_tL^2_x} +\norm{ w_{q+1}^{(t)} }_{L^1_tL^2_x}+\norm{ \wo }_{L^1_tL^2_x}\notag \\
	&\lesssim \thq^{-1}\tau^{-\frac12}+\thq^{-1} r_{\perp} r_{\parallel}^{-1} \tau^{-\frac12}+ \thq^{-2} \mu^{-1} r_{\perp}^{-\frac{1}{2}} r_{\parallel}^{-\frac12} + \thq^{-9}\sigma^{-1}\lesssim \lambda_{q+1}^{-\ve}.
\end{align}

Next, we verify the iterative estimates for $u_{q+1}$.
Since $\lbb_q^{-3} \ll \delta_{q+2}^{1/2}$,
using \eqref{uh3}, \eqref{nuh3},
\eqref{pdvh3},
\eqref{principal h3 est-endpt1} and \eqref{pth2 est-endpt1} we derive that
\begin{align}
	 \norm{u_{q+1}}_{L^\9_tH^3_x}
	\lesssim& \norm{\wt u_q}_{L^\9_tH^3_x}+\norm{w_{q+1}}_{L^\9_tH^3_x} \notag \\
	\lesssim&  \lambda_{q}^5+ \lambda_{q+1}^{5}\lesssim \lambda_{q+1}^5, \label{verifyuc1}
\end{align}
and
\begin{align}
	\norm{\p_t u_{q+1}}_{L^\9_tH^2_x}
	\lesssim& \norm{\p_t \wt u_q}_{L^\9_tH^2_x}+\norm{\p_t w_{q+1}}_{L^\9_tH^2_x} \notag \\
\lesssim& \sup_i   \|\p_t\(\chi_iv_i \) \|_{L^\9_tH^2_x}+ \lambda_{q+1}^{8}  \notag\\
\lesssim& \sup_i( \|\p_t\chi_i\|_{C_t}\| v_i  \|_{L^\9(\supp \chi_i; H^2_x)}
      +\|\chi_i\|_{C_t}\| \p_t v_i \|_{L^\9(\supp \chi_i;H^2_x)}) + \lambda_{q+1}^{8}  \notag\\
\lesssim&\, \thq^{-1}\la^{5}+\mq\la^5+ \lambda_{q+1}^8\lesssim \lambda_{q+1}^8. \label{verifyupth2}
\end{align}
Moreover, using \eqref{la}, \eqref{b-beta-ve}, \eqref{uuql2}
and \eqref{Lp decorr vel}-\eqref{wql1.1} we have
\begin{align}
	& \norm{u_{q} - u_{q+1}}_{L^2_{t,x}}  \leq \norm{ u_{q} -\wt u_{q}  }_{L^2_{t,x}} + \norm{\wt u_{q}  - u_{q+1}}_{L^2_{t,x}} \nonumber   \\
	&\qquad \qquad \qquad \  \lesssim  \norm{ u_q - \wt u_{q} }_{L^\9_{t}L^2_x}+ \norm{w_{q+1}}_{L^2_{t,x}}  \nonumber  \\
	&\qquad \qquad \qquad \  \lesssim  \la^{-3}+\delta_{q+1}^{\frac{1}{2}}   \leq M^*\delta_{q+1}^{\frac{1}{2}}, \label{e3.43}
\end{align}
for $M^*$ sufficiently large and
\begin{align}  \label{uql1l2.1}
	\norm{u_{q} - u_{q+1}}_{L^1_tL^2_x} 	
	\lesssim& \norm{ u_{q} -\wt u_{q} }_{L^\9_{t}L^2_x}+ \norm{w_{q+1}}_{L^1_tL^2_x} \nonumber \\
	\lesssim& \la^{-3}+\lambda_{q+1}^{-\ve} \leq \delta_{q+2}^{\frac{1}{2}},
\end{align}
where we also chose $a$ sufficiently large
such that the last inequalities of \eqref{verifyuc1} are valid.

Regarding the iteration estimate \eqref{u-B-Lw-conv}, we first claim that the Sobolev embedding
\begin{align}\label{sobolevem}
H^3_x\hookrightarrow W^{s,p}_x
\end{align}
holds for any $(s,p,\gamma)\in \mathcal{A}_1$.

To this end,
it suffices to consider the case of $\gamma=1$.
For the given $(s,p,\gamma)\in \mathcal{A}_1$,
it holds  that for $1\leq p\leq 2$,
\begin{align*}
0\leq s< 4\a-5+\frac{3}{p}+(1-2\a)\leq 2\a-1<3.
\end{align*}
Because $L^2_x\hookrightarrow L^p_x$ for $1\leq p\leq 2$,
we obtain $H^3_x \hookrightarrow W^{s,p}_x$.
Moreover, for $p>2$,
\begin{align*}
0\leq s< 4\a-5+\frac{3}{p} +(1-2\a)< 2\a-\frac52<\frac32.
\end{align*}
Taking into account $H^3_x\hookrightarrow W_x^{s,\9} \hookrightarrow W^{s,p}_x$ for $s<3/2$,
we thus obtain \eqref{sobolevem}.
Hence, by virtue of \eqref{rh4}, \eqref{est-vih3} and \eqref{sobolevem}, we have
\begin{align}\label{wtu-u}
 \norm{\wt u_q-u_q}_{L^\gamma_tW^{s,p}_x} & \lesssim \norm{\sum_i\chi_i(v_i-u_q)}_{L^\9_tH^3_x} \notag\\
 & \lesssim \sup_i \Big(\norm{\chi_i(v_i-u_q)}_{L^\9((\supp(\chi_i);H^3_x)}
    +\norm{(1-\chi_i)(v_{i-1}-u_q)}_{L^\9((\supp(\chi_i\chi_{i-1});H^3_x)}\Big)\notag\\
 	&\lesssim \sup_i |t_{i+1}+\theta_{q+1}-t_i| \||\nabla|\rr_q\|_{L^\9_tH^3_x}\lesssim \mq^{-1}\la^{10}\lesssim \la^{-2}.
\end{align}

Therefore, for $\a\in [5/4,2)$,
by Lemma \ref{totalest},
 \begin{align}\label{lw-est}
 	\norm{ u_{q+1} - u_q }_{L^\gamma_tW^{s,p}_x}
	 &\lesssim \norm{\wt u_q-u_q}_{L^\gamma_tW^{s,p}_x}
 	+\norm{w_{q+1}}_{L^\gamma_tW^{s,p}_x}   \notag\\
 	&\lesssim \la^{-2} + \thq^{-1}\lbb_{q+1}^{s}\rs^{\frac{2}{p}-1}\rp^{\frac{1}{p}-\frac12}\tau^{\frac12-\frac{1}{\gamma}}
 	+\thq^{-30}\sigma^{-1} \notag\\
   &\lesssim \lambda_{q}^{-2}
	+ \lambda_{q+1}^{s+2\a-1-\frac{3}{p}-\frac{4\a-5}{\gamma} +\ve(2+\frac8p-\frac{11}{\gamma}) }
	+ \lbb_{q+1}^{-\ve} ,
 \end{align}
where in the last inequality we also used \eqref{b-beta-ve}, \eqref{larsrp} and Lemma \ref{totalest}.
Taking into account \eqref{e3.1}
we obtain
\begin{align}\label{endpt1-condition}
s+2\a-1-\frac{3}{p}-\frac{4\a-5}{\gamma} +\ve(2+\frac8p-\frac{11}{\gamma})
\leq s+2\a-1-\frac{3}{p}-\frac{4\a-5}{\gamma} +10\ve
<-10\ve,
\end{align}
which yields that
\begin{align}
\norm{ u_{q+1} - u_q }_{L^\gamma_tW^{s,p}_x} \leq \delta_{q+2}^{\frac12}. \label{ne6.6}
\end{align}

Therefore, the iteration estimates \eqref{uh3}, \eqref{upth2}, \eqref{u-B-L2tx-conv}-\eqref{u-B-Lw-conv} are verified.

\section{Reynolds stress for the supercritical regime $\mathcal{A}_1$}   \label{Sec-Rey-Endpt1}

The aim of this section is to verify the inductive estimates \eqref{rh3} and \eqref{rl1}
for the new Reynolds stress $\mathring{R}_{q+1}$
in the supercritical regime $\mathcal{A}_1$ when $\alpha \in [5/4,2)$,
whose borderline in particular includes the endpoint case $(s,\gamma,p)=(3/p+1-2\alpha,\infty, p)$.

\subsection{Decomposition of Reynolds stress}
We derive from \eqref{q+1 velocity} and equation \eqref{equa-nsr}
of $(\wt u_q, \mathring{\wt R}_q)$ that
the new Reynolds stress $\mathring{R}_{q+1}$ satisfies
\begin{align}  \label{ru}
		\displaystyle\div\mathring{R}_{q+1} - \nabla P_{q+1}
		&\displaystyle = \underbrace{\partial_t (w_{q+1}^{(p)}+w_{q+1}^{(c)}) +\nu(-\Delta)^{\alpha} w_{q+1} +\div\big(\wt u_q \otimes w_{q+1} + w_{q+ 1} \otimes \wt u_q \big) }_{ \div\mathring{R}_{lin} +\nabla P_{lin} }   \notag\\
		&\displaystyle\quad+ \underbrace{\div (w_{q+1}^{(p)} \otimes w_{q+1}^{(p)}+  \tr_q)+\partial_t w_{q+1}^{(t)}+\partial_t \wo}_{\div\mathring{R}_{osc} +\nabla P_{osc}}  \notag\\
		&\displaystyle\quad+ \underbrace{\div\Big((w_{q+1}^{(c)}+ w_{q+1}^{(t)}+\wo)\otimes w_{q+1}+ w_{q+1}^{(p)} \otimes (w_{q+1}^{(c)}+ w_{q+1}^{(t)}+\wo) \Big)}_{\div\mathring{R}_{cor} +\nabla P_{cor}}.
\end{align}
Using the inverse divergence operator $\mathcal{R}$
we can choose the Reynolds stress at level $q+1$:
\begin{align}\label{rucom}
	\mathring{R}_{q+1} := \mathring{R}_{lin} +   \mathring{R}_{osc}+ \mathring{R}_{cor},
\end{align}
where the linear error
\begin{align}
	\mathring{R}_{lin} & := \mathcal{R}\(\partial_t (w_{q+1}^{(p)} +w_{q+1}^{(c)}  )\)
	+ \nu \mathcal{R} (-\Delta)^{\a} w_{q+1} + \mathcal{R}\P_H \div \(\wt u_q\mathring{\otimes} w_{q+1} + w_{q+ 1}
	\mathring{\otimes} \wt u_q\), \label{rup}
\end{align}
the oscillation error
\begin{align}\label{rou}
	\mathring{R}_{osc} :=& \sum_{k \in \Lambda } \mathcal{R} \P_H\P_{\neq 0}\left(\g^2 \P_{\neq 0}(W_{(k)}\otimes W_{(k)})\nabla (a_{(k)}^2)\right) \notag\\
	& -\mu^{-1}\sum_{k \in \Lambda }\mathcal{R} \P_H \P_{\neq 0}\(\p_t (a_{(k)}^2\g^2)\psi_{(k_1)}^2\phi_{(k)}^2k_1\)\notag\\
	&-\sigma^{-1}\sum_{k\in \Lambda}\mathcal{R} \P_H \P_{\neq 0}\(h_{(\tau)}\aint_{\T^3}W_{(k)}\otimes W_{(k)}\d x\, \p_t\nabla(a_{(k)}^{2})\),
\end{align}
and the corrector error
\begin{align}
	\mathring{R}_{cor} &
	:= \mathcal{R} \P_H \div \bigg( w^{(p)}_{q+1} \mathring{\otimes} (w_{q+1}^{(c)}+w_{q+1}^{(t)}+\wo)
      + (w_{q+1}^{(c)}+w_{q+1}^{(t)}+\wo) \mathring{\otimes} w_{q+1} \bigg). \label{rup2}
\end{align}
Moreover, one also has  (see, e.g., \cite{bcv21,lzz21})
\begin{align} \label{calRuPHdiv-Ru}
    \mathring{R}_{q+1}  = \mathcal{R}  \mathbb{P}_H \div \mathring{R}_{q+1}.
\end{align}

\subsection{Verification of $L^\9_tH^N_x$-estimates of Reynolds stress}
Regarding the $L^\9_tH^N_x$-estimates \eqref{rh3} and \eqref{rh4}
of the Reynolds stress,
using the identity \eqref{calRuPHdiv-Ru} and
equation \eqref{equa-nsr} for $(u_{q+1}, \mathring R_{q+1})$
we get that for $N=3,4$,
\begin{align}
\norm{\mathring R_{q+1}}_{L^\9_tH^N_x}&\leq  \norm{\mathcal R \P_H (\div \rr_{q+1})}_{L^\9_tH^N_x}\notag \\
&\lesssim \norm{\partial_t u_{q+1}+\div(u_{q+1}\otimes u_{q+1}) +\nu(-\Delta )^{\alpha}u_{q+1}}_{L^\9_tH^{N-1}_x}\notag \\
&\lesssim  \norm{\partial_t u_{q+1}}_{L^\9_tH^{N-1}_x}+\norm{u_{q+1}\otimes u_{q+1}}_{L^\9_tH^N_x} + \norm{u_{q+1}}_{L^\9_tH^{N+3}_x}\notag \\
&\lesssim   \norm{\partial_t u_{q+1}}_{L^\9_tH^{N-1}_x} +\norm{u_{q+1}}_{L^\9_tH^N_x} \norm{u_{q+1}}_{L^\infty_{t,x}} + \norm{u_{q+1}}_{L^\9_tH^{N+3}_x}.\label{ine-rq1h3}
\end{align}

We claim that for every $0\leq \wt N\leq 4$ and for all $q\geq 0$,
\begin{align}\label{uh6}
\norm{u_{q+1}}_{L^\9_tH^{\wt N+3}_x}\lesssim \lambda_{q+1}^{\wt N+5}, \quad
\norm{\p_t u_{q+1}}_{L^\9_tH^{\wt N-1}_x}\lesssim \lambda_{q+1}^{\wt N+5},
\end{align}
where the implicit constant is independent of $q$.

To this end,
in view of \eqref{def-mq-thetaq},  \eqref{nuh3},
\eqref{pdvh3}, \eqref{q+1 velocity} and Lemma~\ref{totalest},
we derive
\begin{align}\label{uhm3}
		\|u_{q+1}\|_{L^\9_tH^{\wt N+3}_x}& \leq \|\wt u_q\|_{L^\9_tH^{\wt N+3}_x}+\|w_{q+1}\|_{L^\9_tH^{\wt N+3}_x} \notag\\
&\lesssim \sup_i\|v_i\|_{L^\9(\supp (\chi_i); H^{\wt N+3}_x)} +\laq^{\wt N+5} \notag\\
&\lesssim \mq^{\frac{\wt N}{2\a}}\la^5+ \laq^{\wt N+5}\notag\\
& \lesssim \lambda_{q+1}^{\wt N+5},
\end{align}
and
\begin{align}\label{ptuhm}
		\|\p_t u_{q+1}\|_{L^\9_tH^{\wt N-1}_x}& \leq \|\p_t \wt u_q\|_{L^\9_tH^{\wt N-1}_x}+\|\p_t w_{q+1}\|_{L^\9_tH^{\wt N-1}_x} \notag\\
&\lesssim \sup_i\|\p_t (\chi_iv_i)\|_{L^\9(\supp( \chi_i); H^{\wt N-1}_x)} +\laq^{\wt N+5}\notag\\
&\lesssim \sup_i( \|\p_t\chi_i\|_{C_t}\| v_i  \|_{L^\9(\supp (\chi_i); H^{\wt N-1}_x)}
  +\|\chi_i\|_{C_t}\| \p_t v_i \|_{L^\9(\supp( \chi_i); H^{\wt N-1}_x)})+ \laq^{\wt N+5} \notag\\
& \lesssim \thq^{-1}\la^5+\mq\la^5+\lambda_{q+1}^{\wt N+5}\lesssim \lambda_{q+1}^{\wt N+5}.
\end{align}
Thus, we prove \eqref{uh6}, as claimed.

Concerning the $L^\9_{t,x}$ estimates of $u_{q+1}$,
using \eqref{nuh3},
the Sobolev embedding $H^2_x\hookrightarrow L^\9_x$
and Lemma~\ref{totalest} we have
\begin{align}\label{ul9}
\norm{u_{q+1}}_{L^\infty_{t,x}} & \leq \norm{\wt u_{q}}_{L^\infty_{t,x}}+ \norm{w_{q+1}}_{L^\infty_{t}H^2_x} \lesssim \norm{\wt u_{q}}_{L^\infty_{t}H^3_x}+ \laq^{4}
\lesssim\la^{5}+ \lambda_{q+1}^{4} \lesssim \lambda_{q+1}^{4},
\end{align}
where the implicit constant is independent of $q$.

Thus, inserting \eqref{verifyuc1}, \eqref{verifyupth2}, \eqref{uh6} and \eqref{ul9} into \eqref{ine-rq1h3}
we lead to
\begin{align*}
\norm{\mathring R_{q+1}}_{L^\9_tH^N_x}
&\lesssim  \lambda_{q+1}^{N+5}+ \lambda_{q+1}^{N+6} + \lambda_{q+1}^{N+5}
\lesssim \lambda_{q+1}^{N+6}
\end{align*}
for some universal constant.
Therefore, the $L^\9_tH^N_x$-estimates \eqref{rh3} and \eqref{rh4} of $\mathring R_{q+1}$ are verified.

\subsection{Verification of $L^1_{t,x}$-decay of Reynolds stress}
Below we mainly verify the delicate $L^1_{t,x}$-decay \eqref{rl1}
of Reynolds stress $\mathring{R}_{q+1}$ at level $q+1$.
Since the building blocks constructed in
\S \ref{Sec-Flow-Endpt1} are highly oscillated and concentrated in space and time,
we need to be careful when estimating
the time/space derivatives in the linear error $\mathring{R}_{lin}$
and the high frequency errors in the oscillation error $\mathring{R}_{osc}$.

Since the Calder\'{o}n-Zygmund operators are bounded in the space $L^\rho_x$, $1<\rho<\infty$,
we choose
\begin{align}\label{defp}
\rho: =\frac{3-8\varepsilon}{3-9\varepsilon}\in (1,2),
\end{align}
where $\ve$ is given by \eqref{e3.1}.
Note that,
\begin{equation}\label{setp}
	(3-8\varepsilon)(1-\frac{1}{\rho})=\ve,
\end{equation}
and
\begin{align}  \label{rs-rp-p-ve}
\rs^{\frac 2\rho-2}\rp^{\frac 1\rho-1} = \lambda^{\ve},\ \
   \rs^{\frac2\rho- 1}\rp^{\frac 1\rho-\frac 12} = \lambda^{-\frac32+5\varepsilon},\ \
   \rs^{\frac 2\rho} \rp^{\frac 1\rho - \frac 32} = \lambda^{-\frac 32+ 3\ve}.
\end{align}

We shall treat the linear error, the oscillation error
and the corrector, separately.

\paragraph{\bf (i) Linear error.}
First, in view of Lemmas \ref{buildingblockestlemma},
\ref{Lem-gk-esti} and \ref{mae-endpt1}, \eqref{e3.1}, \eqref{larsrp} and \eqref{rs-rp-p-ve}, we have
\begin{align}
	 & \| \mathcal{R}\partial_t( w_{q+1}^{(p)}+ w_{q+1}^{(c)})\|_{L_t^1L_x^\rho}  \nonumber \\
	\lesssim& \sum_{k \in \Lambda}\| \mathcal{R} \curl\curl\partial_t(\g a_{(k)} W^c_{(k)}) \|_{L_t^1L_x^\rho} \nonumber \\
	\lesssim& \sum_{k \in \Lambda}\Big(\| \g\|_{L^1_t}(\|  a_{(k)} \|_{C_{t,x}^2}\| W^c_{(k)} \|_{C_t W_x^{1,\rho}}
            +\|  a_{(k)} \|_{C_{t,x}^1}\| \p_t W^c_{(k)} \|_{C_t W^{1,\rho}_x}) \nonumber \\
	&\qquad\qquad+\| \p_t\g\|_{L_t^1}\| a_{(k)} \|_{C_{t,x}^1}\| W^c_{(k)} \|_{C_t W_x^{1,\rho}}\Big)\nonumber \\
    \lesssim& \tau^{-\frac12}(\thq^{-14}\rs^{\frac{2}{\rho}-1}\rp^{\frac{1}{\rho}-\frac{1}{2}}\lambda^{-1}
        +\thq^{-7}\rs^{\frac{2}{\rho} }\rp^{\frac{1}{\rho}-\frac{3}{2}}\mu)
        + \sigma\tau^{\frac12}\thq^{-7}\rs^{\frac{2}{\rho}-1}\rp^{\frac{1}{\rho}-\frac{1}{2}}\lambda^{-1}  \notag\\
	\lesssim& \thq^{-14}(\lambda^{-2\a-\frac \ve2 }+ \lambda^{-\frac \ve2}+\lambda^{2\a-5+13\ve} )
     \lesssim \thq^{-14}\lambda^{-\frac \ve2}.\label{time derivative}
\end{align}

Regarding the viscosity term $(-\Delta)^\alpha w_{q+1}$,
we use \eqref{velocity perturbation} to estimate
\begin{align}
	\norm{\nu \mathcal{R}(-\Delta)^{\alpha} w_{q+1} }_{L_t^1L^\rho_x} \lesssim & \norm{\nu \mathcal{R}(-\Delta)^{\alpha} w_{q+1}^{(p)} }_{L_t^1L^\rho_x}+\norm{ \nu \mathcal{R}(-\Delta)^{\alpha} w_{q+1}^{(c)} }_{L_t^1L^\rho_x}\notag \\
	& +\norm{ \nu\mathcal{R}(-\Delta)^{\alpha} w_{q+1}^{(t)} }_{L_t^1L^\rho_x}+\norm{ \nu\mathcal{R}(-\Delta)^{\alpha} \dqo }_{L_t^1L^\rho_x}.\label{e5.17}
\end{align}
In order to estimate the
right-hand-side above,
we use the interpolation inequality (cf. \cite{BM18}), \eqref{uprinlp-endpt1} and the fact that $2-\alpha \geq  20\va$
to derive
\begin{align}
	\norm{ \nu\mathcal{R}(-\Delta)^{\alpha} w_{q+1}^{(p)} }_{L_t^1L^\rho_x}
    & \lesssim \norm{ |\na|^{2\a-1} w_{q+1}^{(p)} }_{L_t^1L^\rho_x}\notag\\
	& \lesssim  \norm{w_{q+1}^{(p)}}_{L_t^1L^\rho_x} ^{\frac{4-2\a}{3}} \norm{w_{q+1}^{(p)}}_{L_t^1W^{3,\rho}_x} ^{\frac{2\a-1}{3}}\notag\\
	& \lesssim \thq^{-1}\lbb^{2\alpha-1}\rs^{\frac{2}{\rho}-1}\rp^{\frac{1}{\rho}-\frac{1}{2}}\tau^{-\frac12}\lesssim \thq^{-1}\lambda^{-\frac \ve2}.\label{e5.18}
\end{align}
Similarly, by Lemma \ref{totalest},
\begin{align}
	&\norm{\nu \mathcal{R}(-\Delta)^{\alpha} w_{q+1}^{(c)} }_{L_t^1L^\rho_x}
     \lesssim \thq^{-1}\lbb^{2\alpha -1}\rs^{\frac{2}{\rho} }\rp^{\frac{1}{\rho}-\frac{3}{2}}\tau^{-\frac12}\lesssim \thq^{-1}\lambda^{-2\ve},\label{e5.19}\\
	&\norm{\nu \mathcal{R}(-\Delta)^{\alpha} w_{q+1}^{(t)} }_{L_t^1L^\rho_x}
      \lesssim \thq^{-2}\lbb^{2\alpha-1}\mu^{-1}\rs^{\frac{2}{\rho}-2}\rp^{\frac{1}{\rho}-1}\lesssim \thq^{-2}\lambda^{-\ve},\label{e5.20}\\
	&\norm{\nu \mathcal{R}(-\Delta)^{\alpha} \dqo }_{L_t^1L^\rho_x}
	\lesssim \thq^{-30}\sigma^{-1}\lesssim \thq^{-30}\lambda^{-2\ve} \lesssim \lambda_{q+1}^{-\ve}.\label{e5.21}
\end{align}
Hence, combining \eqref{e5.17}-\eqref{e5.21} and the fact that $\thq^{-30}\ll \lambda^{\ve}$ altogether we obtain
\begin{align}  \label{mag viscosity}
	\norm{\nu \mathcal{R}(-\Delta)^{\alpha} w_{q+1} }_{L_t^1L^\rho_x} \lesssim \thq^{-1}\lambda^{-\frac \ve2} .
\end{align}

It remains to treat the nonlinearity in \eqref{rup}.
By the Sobolev embedding $H^3_x\hookrightarrow L^\9_x$,
\eqref{nuh3} and Lemma \ref{totalest},
\begin{align} \label{linear estimate1}
	&\norm{ \mathcal{R}\P_H\div\(w_{q + 1} \otimes  \wt u_q +  \wt u_q \otimes w_{q+1}\) }_{L_t^1L^\rho_x}  \nonumber \\	
    \lesssim\,&\norm{w_{q + 1} \otimes \wt u_q +  \wt u_q \otimes w_{q+1} }_{L_t^1L^\rho_x}  \nonumber \\
	\lesssim\,& \norm{ \wt u_q}_{L^\9_{t}H^3_x} \norm{w_{q+1}}_{L_t^1L^\rho_x} \nonumber \\
	\lesssim\, &\lambda^5_q (\thq^{-1} \rs^{\frac{2}{\rho}-1}\rp^{\frac{1}{\rho}-\frac12} \tau^{-\frac 12}+\thq^{-2}\mu^{-1}\rs^{\frac{2}{\rho}-2}\rp^{\frac{1}{\rho}-1} +\thq^{-9}\sigma^{-1} )
    \lesssim \thq^{-10}\lambda^{-2\ve}.
\end{align}

Therefore, we conclude from \eqref{time derivative}, \eqref{mag viscosity} and \eqref{linear estimate1} that
\begin{align}   \label{linear estimate}
	\norm{\mathring{R}_{lin} }_{L_t^1L^\rho_x}
     & \lesssim \thq^{-14}\lambda^{-\frac \ve2} +\thq^{-1}\lambda^{-\frac \ve2}+\thq^{-10}\lambda^{-2\ve}
       \lesssim \thq^{-14}\lambda^{-\frac \ve2}.
\end{align}

\paragraph{\bf (ii) Oscillation error.}
Now let us treat the delicate oscillation error.
For this purpose, we decompose the oscillation error into three parts:
\begin{align*}
	\mathring{R}_{osc} = \mathring{R}_{osc.1} +  \mathring{R}_{osc.2}+  \mathring{R}_{osc.3},
\end{align*}
where the low-high spatial oscillation error
\begin{align*}
	\mathring{R}_{osc.1}
	&:=   \sum_{k \in \Lambda }\mathcal{R} \P_{H}\P_{\neq 0}\left(\g^2 \P_{\neq 0}(W_{(k)}\otimes W_{(k)} )\nabla (a_{(k)}^2) \right),
\end{align*}
the high temporal oscillation error
\begin{align*}
	\mathring{R}_{osc.2}
	&:= - \mu^{-1} \sum_{k \in \Lambda}\mathcal{R}\P_{H}\P_{\neq 0}\(\p_t (a_{(k)}^2\g^2) \psi_{(k_1)}^2\phi_{(k)}^2k_1\),
\end{align*}
and the low frequency error
\begin{align*}
	\mathring{R}_{osc.3} &
      := -\sigma^{-1}\sum_{k\in \Lambda}\mathcal{R}\P_{H}\P_{\neq 0}
       \(h_{(\tau)}\aint_{\T^3}W_{(k)}\otimes W_{(k)} \d x\, \p_t\nabla(a_{(k)}^{2})\).
\end{align*}

For the low-high spatial oscillation error $\mathring{R}_{osc.1}$, we note that the velocity flows are of
high oscillations
\begin{align*}
\P_{\neq 0}(W_{(k)}\otimes W_{(k)} )=\P_{\geq (\lambda \rs/2)}(W_{(k)}\otimes W_{(k)} ).
\end{align*}
Thus,
we use Lemmas \ref{buildingblockestlemma}, \ref{mae-endpt1}
and apply Lemma \ref{commutator estimate1}
with $a = \nabla (a_{(k)}^2)$ and $f =  \psi_{(k_1)}^2\phi_{(k)}^2$
to get
\begin{align}  \label{I1-esti-endpt1}
	\norm{\mathring{R}_{osc.1} }_{L^1_tL^\rho_x}
	&\lesssim  \sum_{ k \in \Lambda }
	\|\g\|_{L^2_t}^2\norm{|\nabla|^{-1} \P_{\not =0}
		\left(\P_{\geq (\lambda \rs/2)}(W_{(k)}\otimes W_{(k)} )\nabla (a_{(k)}^2)\right)}_{C_tL^\rho_x} \notag \nonumber  \\
	& \lesssim \sum_{ k \in \Lambda } \||\na|^3 (a^2_{(k)})\|_{C_{t,x}}
	\lambda^{-1} \rs^{-1} \norm{ \psi^2_{(k_1)}}_{C_tL^{\rho}_x} \norm{\phi^2_{(k)} }_{C_tL^{\rho}_x}  \nonumber  \\
	& \lesssim \thq^{-23}  \lambda^{-1}  \rs^{\frac{2}{\rho}-3}\rp^{\frac{1}{\rho}-1}.
\end{align}

Moreover, we apply Lemma~\ref{totalest}
and use the large temporal oscillation parameter $\mu$ to balance the high temporal oscillation error
$\mathring{R}_{osc.2}$:
\begin{align}  \label{I2-esti-endpt1}
	\norm{\mathring{R}_{osc.2} }_{L^1_tL_x^\rho}
	&\lesssim  {\mu}^{-1}  \sum_{k\in\Lambda }\norm{\mathcal{R} \P_{H} \P_{\neq 0}\(\p_t (a_{(k)}^2\g^2)\psi_{(k_1)}^2\phi_{(k)}^2k_1\)}_{L^1_tL_x^\rho} \nonumber  \\
	&\lesssim   {\mu}^{-1}  \sum_{k\in\Lambda}
	\( \norm{\p_t (a_{(k)}^2) }_{C_{t,x}}\norm{\g^2 }_{L^1_t}+ \norm{a_{(k)} }_{C_{t,x}}^2\norm{\p_t(\g^2)}_{ L_t^1 } \)
	\norm{\psi_{(k_1)}}_{C_tL^{2\rho}_x}^2\norm{\phi_{(k)}}_{L^{2\rho}_x}^2 \nonumber \\
	&\lesssim (\thq^{-9}+\thq^{-2}\tau\sigma)\mu^{-1}\rs^{\frac{2}{\rho}-2}\rp^{\frac{1}{\rho}-1}   \notag\\
    &  \lesssim  \thq^{-2}\tau\sigma\mu^{-1}\rs^{\frac{2}{\rho}-2}\rp^{\frac{1}{\rho}-1}.
\end{align}

The low frequency error $\mathring{R}_{osc.3} $
can be estimated easily by using \eqref{hk-esti} and \eqref{mag amp estimates},
\begin{align}  \label{I3-esti-endpt1}
	\norm{\mathring{R}_{osc.3} }_{L^1_tL^\rho_x}
    &\lesssim  \sigma^{-1} \sum_{k\in\Lambda}
     \left\|h_{(\tau)}\, \p_t\nabla(a_{(k)}^{2})\right\|_{L^1_tL^\rho_x} \nonumber  \\
	&\lesssim \sigma^{-1} \sum_{k\in\Lambda} \|h_{(\tau)}\|_{C_t}\( \norm{a_{(k)} }_{C_{t,x}} \norm{a_{(k)} }_{C_{t,x}^2} +\norm{a_{(k)} }_{C_{t,x}^1}^2\)\nonumber \\
	&\lesssim \thq^{-15} \sigma^{-1}.
\end{align}

Therefore, combing \eqref{I1-esti-endpt1}-\eqref{I3-esti-endpt1}
altogether and using \eqref{larsrp}, \eqref{rs-rp-p-ve}
and the bound $0<\ve<{(2-\a)}/{20}$
we conclude
\begin{align}
	\label{oscillation estimate}
	\norm{\mathring{R}_{osc}}_{L_t^1L^\rho_x}
    &\lesssim    \thq^{-23}  \lambda^{-1} \rs^{\frac{2}{\rho}-3}\rp^{\frac{1}{\rho}-1}
	+\thq^{-9}\tau\sigma\mu^{-1}\rs^{\frac{2}{\rho}-2}\rp^{\frac{1}{\rho}-1}+\thq^{-15} \sigma^{-1} \notag\\
	&\lesssim  \thq^{-23} \lbb^{-\ve}+\thq^{-2}\lbb^{2\a-4+12\ve}+\thq^{-15} \lbb^{-2\va}\notag \\
    & \lesssim  \thq^{-23}  \lambda^{-\ve}.
\end{align}

\paragraph{\bf (iii) Corrector error.}
As in \cite{lzz21},
we take $p_1,p_2\in(1,\9)$ such that
$$ \frac{1}{p_1}=1-\wt \eta,\quad \frac{1}{p_1}=\frac{1}{p_2}+\frac{1}{2},$$
with $\wt  \eta\leq \ve/(4(3-8\ve))$.
Using H\"older's inequality, Lemma \ref{totalest},
\eqref{Lp decorr vel} and \eqref{e3.41}
we derive
\begin{align}
	\norm{\mathring{R}_{cor} }_{L^1_{t}L^{p_1}_x}
	\lesssim& \norm{ w_{q+1}^{(p)} \otimes (w_{q+1}^{(c)}+ w_{q+1}^{(t)}+\wo) -(w_{q+1}^{(c)}+w_{q+1}^{(t)}+\wo) \otimes w_{q+1} }_{L^1_{t}L^{p_1}_x} \notag \\
	\lesssim& \norm{w_{q+1}^{(c)}+w_{q+1}^{(t)}+\wo }_{L^2_{t}L^{p_2}_x} (\norm{w^{(p)}_{q+1} }_{L^2_{t,x}} + \norm{w_{q+1} }_{L^2_{t,x}})\notag  \\
     \lesssim&  \delta_{q+1}^{\frac 12} ( \thq^{-1}\rs^{\frac{2}{p_2} }\rp^{\frac{1}{p_2}-\frac32}+ \thq^{-2}\mu^{-1} \rs^{\frac{2}{p_2}-2}\rp^{\frac{1}{p_2}-1} \tau^{\frac 12} +\thq^{-9}\sigma^{-1} )    \notag \\
	 \lesssim& \thq^{-9}\delta_{q+1}^\frac 12  (\lambda^{-2\ve+3\wt \eta-8\wt \eta\ve}+\lambda^{-\frac \ve2+3\wt \eta-8\wt \eta\ve}+ \lambda^{-2\ve} ) \lesssim \thq^{-9} \lambda^{-\frac \ve4}, \label{corrector estimate}
\end{align}
where the last step is due to the inequalities
$-\ve/2 + 3\wt \eta - 8\wt \eta \ve \leq -{\ve}/{4}$.

Therefore, combining the estimates \eqref{linear estimate},
\eqref{oscillation estimate},
\eqref{corrector estimate} and the fact that $\thq^{-9}\ll \lambda^{\frac \ve8}$ altogether we conclude
\begin{align} \label{rq1b}
	\|\mathring{R}_{q+1} \|_{L^1_{t,x}}
	&\leq \| \mathring{R}_{lin} \|_{L^1_tL^\rho_{x}} +  \| \mathring{R}_{osc}\|_{L^1_tL^\rho_{x}}
	+  \|\mathring{R}_{cor} \|_{L^1_tL^{p_1}_{x}} \nonumber  \\
	&\lesssim   \thq^{-14}\lambda^{-\frac \ve2}+\thq^{-23}  \lambda^{-\ve} + \thq^{-9} \lambda^{-\frac \ve4}\nonumber  \\
	& \leq  \lambda_{q+1}^{-\ve_R}\delta_{q+2}.
\end{align}
Thus, the $L^1$-estimate \eqref{rl1} of Reynolds stress is verified.

\section{The supercritical regime $\mathcal{A}_2$} \label{Sec-Endpt2}

In this section,
we mainly treat the supercritical regime $\mathcal{A}_2$ when $\alpha \in [1,2)$,
whose borderline in particular includes the other endpoint $(s,\gamma,p)=(2\alpha/\gamma+1-2\alpha, \gamma, \infty)$.

\subsection{Space-time building blocks} \label{Sec-Interm-Flow}
Unlike in the previous endpoint case $(s,\gamma,p)=(3/p+1-2\alpha,\infty,p)$,
$\alpha\in [5/4,2)$,
the building blocks in this case are indexed by four parameters $\rs$, $\lambda$, $\tau$ and $\sigma$:
\begin{equation}\label{larsrp-endpt2}
	\rs := \lambda_{q+1}^{-\a+1-8\varepsilon},\
	 \lambda := \lambda_{q+1},\      \tau:=\lambda_{q+1}^{2\a}, \ \sigma:=\lambda_{q+1}^{2\varepsilon},
\end{equation}
where $\varepsilon$ is given by \eqref{ne3.1}.

Instead of the intermittent jets,
inspired by \cite{cl20.2},
we choose the concentrated Mikado flows defined by
\begin{equation*}
	W_{(k)} :=  \phi_{\rs}( \lambda \rs N_{\Lambda}k\cdot (x-\alpha_k),\lambda \rs N_{\Lambda}k_2\cdot (x-\alpha_k))k_1,\ \  k \in \Lambda,
\end{equation*}
where we keep the same notations $\phi_{\rs}, r_{\rs}, N_{\Lambda}$
and $(k,k_1,k_2)$ as in \S \ref{Sec-Flow-Endpt1}.
We still use the same temporal building blocks $g_{(\tau)}$, $h_{(\tau)}$
as in \eqref{gk},
but with the different choice of parameters $\tau, \sigma$
given by \eqref{larsrp-endpt2}.

It should be mentioned that,
unlike the intermittent jets defined in \S~\ref{Sec-Flow-Endpt1}, the term $\mu t$
and the concentration parameter $\rp$ are not involved in the Mikado flows.
Hence, the Mikado flows provide at most 2D intermittency.
However, the temporal building blocks provide more intermittency,
which achieves even 4D spatial intermittency
when $\a$ is close to 2.

Then, setting
\begin{equation}\label{snp-endpt2}
	\begin{array}{ll}
		&\phi_{(k)}(x) := \phi_{\rs}( \lambda \rs N_{\Lambda}k\cdot (x-\a_k),\lambda \rs N_{\Lambda}k_2\cdot (x-\a_k)), \\
		&\Phi_{(k)}(x) := \Phi_{\rs}( \lambda \rs N_{\Lambda}k\cdot (x-\a_k),\lambda \rs N_{\Lambda}k_2\cdot (x-\a_k)),
	\end{array}
\end{equation}
we may rewrite
\begin{equation}\label{snwd-endpt2}
	W_{(k)} = \phi_{(k)} k_1,\quad  k\in \Lambda.
\end{equation}
The corresponding potential is then defined by
\begin{equation}
	\begin{aligned}
		\label{corrector vector-endpt2}
	 W_{(k)}^c := \frac{1}{\lambda^2N_{ \Lambda }^2}\Phi_{(k)} k_1 .
	\end{aligned}
\end{equation}

We summarize the estimates of spatial building blocks
in Lemma \ref{buildingblockestlemma-endpt2},
which follows from Lemma \ref{buildingblockestlemma}.

\begin{lemma} [Estimates of Mikado flows] \label{buildingblockestlemma-endpt2}
	For any $p \in [1,\infty]$ and $N \in \mathbb{N}$, we have
	\begin{align}
		&\left\|\nabla^{N} \phi_{(k)}\right\|_{L^{p}_{x}}+\left\|\nabla^{N} \Phi_{(k)}\right\|_{L^{p}_{x}}
		\lesssim r_{\perp}^{\frac 2p- 1}  \lambda^{N}, \label{intermittent estimates2-endpt2}
	\end{align}
	where the implicit constants are independent of $\rs,\,\rp,\,\lambda$ and $\mu$. Moreover, it holds that
	\begin{align}
		&\displaystyle \|\nabla^{N}  W_{(k)}\|_{C_t  L^{p}_{x}}
          +\lambda^{2} \|\nabla^{N}  W_{(k)}^c\|_{C_t L^{p}_{x}}\lesssim r_{\perp}^{\frac 2p- 1} \lambda^{N}, \ \ k\in \Lambda. \label{ew-endpt2}
	\end{align}
\end{lemma}

\subsection{Velocity perturbations}   \label{Sec-Pert}

We define the amplitudes of the velocity perturbations by
\begin{equation}\label{akb-endpt2}
	a_{(k)}(t,x):= \varrho^{\frac{1}{2} } (t,x) f (t)\gamma_{(k)}
       ({\rm Id}-\frac{\tr_q(t,x)}{\rho(t,x)}), \quad k \in \Lambda,
\end{equation}
where $\varrho, f, \gamma_{(k)}$ are defined as in \S \ref{Subsec-Velo-perturb}.
Note that,
the amplitudes $a_{(k)}$, $k\in \Lambda$,
obey the same estimates as in Lemma \ref{mae-endpt1}.
Namely, we have
\begin{lemma} \label{mae-endpt2}
	For $1\leq N\leq 9$, $k\in \Lambda$, we have
		\begin{align}
			\label{e3.15.2}
			&\norm{a_{(k)}}_{L^2_{t,x}} \lesssim \delta_{q+1}^{\frac{1}{2}} ,\\
			\label{mag amp estimates-endpt2}
			& \norm{ a_{(k)} }_{C_{t,x}} \lesssim \thq^{-1},\ \ \norm{ a_{(k)} }_{C_{t,x}^N} \lesssim \thq^{-7N},
		\end{align}
where the implicit constants are independent of $q$.
	\end{lemma}

Next, we define the principal part $w_{q+1}^{(p)}$ of the velocity perturbations by
\begin{align}
		w_{q+1}^{(p)} &:= \sum_{k \in \Lambda } a_{(k)}\g W_{(k)},
		\label{pv-endpt2}
\end{align}
which satisfies the same algebraic identity as in \eqref{mag oscillation cancellation calculation}.

The corresponding incompressibility corrector is then defined by
	\begin{align}\label{wqc-dqc-endpt2}
		w_{q+1}^{(c)}
		&:=   \sum_{k\in \Lambda }\g (\nabla a_{(k)} \times \curl W_{(k)}^c+ \curl (\nabla a_{(k)} \times W_{(k)}^c)).
	\end{align}
Note that, the incompressibility corrector \eqref{wqc-dqc-endpt2}
is different from the previous one in \eqref{wqc-dqc}.

By straightforward computations,
	\begin{align}
&  w_{q+1}^{(p)} + w_{q+1}^{(c)}=\sum_{k \in \Lambda} \curl \curl ( a_{(k)} \g W_{(k)}^c) , \label{div free velocity}
	\end{align}
which yields immediately that
\begin{align*}
	\div (w_{q+1}^{(p)} + w_{q +1}^{(c)})= 0.
\end{align*}

Regarding the temporal corrector,
because the new spatial building block \eqref{snwd-endpt2} satisfies $\div (W_{(k)}\otimes W_{(k)})=0$,
it is not necessary to introduce the temporal corrector $w_{q+1}^{(t)}$ as in \eqref{veltemcor}
to balance the spatial oscillation.
We will only need the temporal corrector $w_{q+1}^{(o)}$
to balance the high temporal frequency oscillation
in \eqref{mag oscillation cancellation calculation}:
\begin{align}
		& w_{q+1}^{(o)}:= -\sigma^{-1}\sum_{k\in\Lambda }\P_{H}\P_{\neq 0}\(h_{(\tau)}\aint_{\T^3} W_{(k)}\otimes W_{(k)}\d x\nabla (a_{(k)}^2) \) .   \label{wo-endpt2}
\end{align}
Then, by virtue of \eqref{hk} and \eqref{wo-endpt2}
and Leibniz's rule
we see that the algebraic identity \eqref{utemcom} is still valid.

Now, we are ready to define the velocity perturbation $w_{q+1}$ at level $q+1$ by
	\begin{align}
		w_{q+1} &:= w_{q+1}^{(p)} + w_{q+1}^{(c)}+\wo
		\label{velocity perturbation-endpt2}
	\end{align}
and the velocity field at level $q+1$  by
	\begin{align}
		& u_{q+1}:= \wt u_q + w_{q+1},
		\label{q+1 velocity-endpt2}
	\end{align}
where $\wt u_q$ is
the velocity field already prepared in
the gluing stage in \S \ref{Sec-Concen-Rey}.
By the above constructions, $w_{q+1}$ is mean-free and divergence-free.

Analogous to Lemma \ref{totalest}.
we have the estimates of the velocity perturbations below.

\begin{lemma}  [Estimates of perturbations] \label{totalest-endpt2}
	For any $\rho \in(1,\infty), \gamma \in [1,\infty]$ and
	every integer $0\leq N\leq 7$,
we have the following estimates:
	\begin{align}
	&\norm{\na^N w_{q+1}^{(p)} }_{L^ \gamma_tL^\rho_x }  \lesssim \thq^{-1} \lbb^N\rs^{\frac{2}{\rho}-1}\tau^{\frac12-\frac{1}{ \gamma}},\label{uprinlp-endpt2}\\
	&\norm{\na^N w_{q+1}^{(c)} }_{L^\gamma_tL^\rho_x   } \lesssim \thq^{-7}\lbb^{N-1}\rs^{\frac{2}{\rho}-1}\tau^{\frac12-\frac{1}{\gamma}}, \label{ucorlp-endpt2} \\
	&\norm{\na^N \wo }_{L^\gamma_tL^\rho_x  }\lesssim \thq^{-7N-9}\sigma^{-1} ,\label{dcorlp-endpt2}
	\end{align}
where the implicit constants depend only on $N$, $\gamma$ and $\rho$. In particular, for integers $1\leq N\leq 7$, we have
\begin{align}
& \norm{ w_{q+1}^{(p)} }_{L^\9_tH^N_x }  + \norm{ w_{q+1}^{(c)} }_{L^\9_tH^N_x}+\norm{ \wo }_{L^\9_tH^N_x}\lesssim \lambda^{N+2},\label{principal h3 est-endpt2}\\
& \norm{\p_t w_{q+1}^{(p)} }_{L^\9_tH^N_x }  + \norm{\p_t w_{q+1}^{(c)} }_{L^\9_tH^N_x}+\norm{\p_t \wo }_{L^\9_tH^N_x}\lesssim \lambda^{N+5},\label{pth2 est-endpt2}
\end{align}
where the implicit constants are independent of $\lambda$.
\end{lemma}

\begin{proof}
First, using \eqref{gk estimate}, \eqref{ew-endpt2},
\eqref{pv-endpt2} and Lemma~\ref{mae-endpt2}
we get that for any $\rho \in (1,\infty)$,
	\begin{align}\label{uplp}
			 \norm{\nabla^N w_{q+1}^{(p)} }_{L^\gamma_tL^\rho_x }
			\lesssim&  \sum_{k \in \Lambda}
             \sum\limits_{N_1+N_2 = N}
            \|a_{(k)}\|_{C^{N_1}_{t,x}}\|\g\|_{L_t^\gamma}
			\norm{ \nabla^{N_2} W_{(k)} }_{C_tL^\rho_x } \notag \\
		\lesssim&	 \thq^{-1}\lbb^N\rs^{\frac{2}{\rho}-1}\tau^{\frac12-\frac{1}{\gamma}},
	\end{align}
and thus \eqref{uprinlp-endpt2} follows.

Moreover, by \eqref{b-beta-ve},
\eqref{gk estimate}, \eqref{ew-endpt2}, \eqref{wqc-dqc-endpt2} and Lemma \ref{mae-endpt2},
	\begin{align*}
 \norm{\na^N w_{q+1}^{(c)} }_{L^\gamma_tL^\rho_x}	\lesssim&
		\sum\limits_{k\in \Lambda }\|\g\|_{L^\gamma_t}  \sum_{N_1+N_2=N}
         \( \norm{ a_{(k)} }_{C_{t,x}^{N_1+1}} \norm{\na^{N_2} W^c_{(k)}}_{C_tW^{1,\rho}_x }+ \norm{ a_{(k)} }_{C_{t,x}^{N_1+2}} \norm{\na^{N_2} W^c_{(k)}}_{C_tL^{\rho}_x } \)  \nonumber   \\
		 \lesssim & \thq^{-7}\lambda^{N-1}\rs^{\frac{2}{\rho}-1} \tau^{\frac12-\frac{1}{\gamma}},
	\end{align*}
which implies \eqref{ucorlp-endpt2}.
	
We then estimate the temporal corrector $\wo$
by using \eqref{wo-endpt2}, \eqref{hk-esti} and Lemmas \ref{mae-endpt2}:
		\begin{align*}
			\norm{ \na^N \wo }_{L^\gamma_tL^\rho_x }
			\lesssim \sigma^{-1}\sum_{k \in \Lambda}\|h_{(\tau)}\|_{C_{t}} \|\nabla^{N+1} (a^2_{(k)})\|_{C_{t,x}}
			\lesssim  \thq^{-7N-9} \sigma^{-1}.
		\end{align*}

Regarding the $L^\9_tH^N_x$-estimates of velocity perturbations,
using \eqref{ne3.1},
\eqref{larsrp-endpt2},
\eqref{uprinlp-endpt2}-\eqref{dcorlp-endpt2} we get
\begin{align*}
& \norm{ w_{q+1}^{(p)} }_{L^\9_tH^N_x }  + \norm{ w_{q+1}^{(c)} }_{L^\9_tH^N_x}+\norm{ \wo }_{L^\9_tH^N_x}\notag \\
\lesssim &\, \thq^{-1}\lambda^N \tau^{\frac12} +\thq^{-7}\lambda^{N-1} \tau^{\frac12} + \thq^{-7N-9}\sigma^{-1} \notag\\
\lesssim &\,\thq^{-1}\lambda^{\a+N} +\thq^{-7}\lambda^{\a+N-1}+ \thq^{-7N-9} \lambda^{-2\ve} \lesssim \lambda^{N+2},
\end{align*}
which verifies \eqref{principal h3 est-endpt2}.

It remains to prove \eqref{pth2 est-endpt2}. By virtue of \eqref{b-beta-ve}, \eqref{larsrp-endpt2} and Lemmas \ref{Lem-gk-esti}, \ref{buildingblockestlemma-endpt2} and \ref{mae-endpt2}, we get
	\begin{align} \label{wprincipal h2 est.2}
		\norm{\p_t w_{q+1}^{(p)} }_{L^\9_tH^N_x }
		\lesssim&   \sum_{k \in \Lambda }
		\|a_{(k)}\|_{C_{t,x}^{N+1} }
		\norm{ \p_t \g}_{L^\9_t}\norm{ W_{(k)} }_{L^\9_tH^N_x}
        \lesssim \thq^{-7N-7} \lbb^{N} \sigma  \tau^{\frac 32}
	\end{align}
	and
	\begin{align} \label{uc h2 est.2}
 \norm{\p_t w_{q+1}^{(c)} }_{L^\9_tH^N_x  }
		& \lesssim   \sum_{k \in \Lambda }
		\|a_{(k)}\|_{C_{t,x}^{N+3}}
		 \norm{\p_t \g}_{C_{t}}
         (\norm{ W^c_{(k)} }_{L^\9_tH^N_x} + \norm{ \nabla W^c_{(k)} }_{L^\9_tH^N_x} )   \nonumber \\
        & \lesssim \thq^{-7N-21} \sigma  \tau^{\frac 32}  (\lbb^{N-2}+\lbb^{N-1})  \notag \\
		& \lesssim \thq^{-7N-21} \lbb^{N-1}\sigma  \tau^{\frac 32}   .
	\end{align}
Since $\mathbb{P}_H \mathbb{P}_{\not =0}$ is bounded in $H^N_x$,
similarly to \eqref{wo h2 est},
we have
	\begin{align} \label{wo h2 est.2}
		\norm{\p_t w_{q+1}^{(o)} }_{L^\9_tH^N_x  }
        \lesssim \sigma^{-1} \sum_{k \in \Lambda } \|\p_t (h_{(\tau)} \na (a_{(k)}^2) )\|_{L^\9_tH^N_x}
		\lesssim  \thq^{-7N-9}\tau .
	\end{align}
Therefore, taking into account that $\thq^{-7N-21}\leq \lbb^{N\ve/2}$
and $0<\ve\leq (2-\a)/20$ we conclude  that
\begin{align*}
& \norm{\p_t w_{q+1}^{(p)} }_{L^\9_tH^N_x }  + \norm{\p_t  w_{q+1}^{(c)} }_{L^\9_tH^N_x}+\norm{ \p_t \wo }_{L^\9_tH^N_x}\notag \\
\lesssim &\, \thq^{-7N-7}\lambda^N\sigma \tau^{\frac32}
+\thq^{-7N-21}\lambda^{N-1} \sigma\tau^{\frac32}
+ \thq^{-7N-9}\tau \notag\\
\lesssim &\,\thq^{-7N-7}\lambda^{3\a+N+2\ve} +\thq^{-7N-21}\lambda^{3\a+N-1+2\ve}+ \thq^{-7N-9} \lambda^{2\a} \lesssim \lambda^{N+6}.
\end{align*}

Therefore, the proof of Lemma \ref{totalest-endpt2} is complete.
\end{proof}

\paragraph{\bf Verification of the inductive estimates for velocity.}
We apply the $L^p$ decorrelation Lemma~\ref{Decorrelation1} with $f= a_{(k)}$, $g = \g\phi_{(k)}$ and $\sigma = \lambda^{2\ve}$
and then using \eqref{la}, \eqref{b-beta-ve}
and Lemmas \ref{Lem-gk-esti}, \ref{buildingblockestlemma-endpt2}
and \ref{mae-endpt2}
to derive
\begin{align}
	\label{Lp decorr vel-endpt2}
	\norm{w^{(p)}_{q+1}}_{L^2_{t,x}}
	&\lesssim \sum\limits_{k\in \Lambda}
       \Big(\|a_{(k)}\|_{L^2_{t,x}}\norm{ \g }_{L^2_{t}} \norm{ \phi_{(k)}}_{C_tL^2_{x}} +\sigma^{-\frac12}\|a_{(k)}\|_{C^1_{t,x}}\norm{ \g }_{L^2_{t}} \norm{ \phi_{(k)}}_{C_tL^2_{x}}\Big) \notag\\
	&\lesssim  \delta_{q+1}^{\frac{1}{2}}+\thq^{-7}\lambda^{-\ve}_{q+1}   \lesssim \delta_{q+1}^{\frac{1}{2}}.
\end{align}
Then, using \eqref{b-beta-ve}, \eqref{Lp decorr vel-endpt2}
and Lemma \ref{totalest-endpt2} we obtain
\begin{align}  \label{e3.41}
	\norm{w_{q+1}}_{L^2_{t,x}} &\lesssim\norm{w_{q+1}^{(p)} }_{L^2_{t,x}} + \norm{ w_{q+1}^{(c)} }_{L^2_{t,x}}  +\norm{ \wo }_{L^2_{t,x}}\notag \\
	&\lesssim \delta_{q+1}^{\frac{1}{2}} +\thq^{-7}\lambda^{-1}+ \thq^{-9}\sigma^{-1}\lesssim \delta_{q+1}^{\frac{1}{2}},
\end{align}
and
\begin{align}  \label{wql1}
	\norm{w_{q+1}}_{L^1_tL^2_x} &\lesssim\norm{w_{q+1}^{(p)} }_{L^1_tL^2_x} + \norm{ w_{q+1}^{(c)} }_{L^1_tL^2_x}+\norm{ \wo }_{L^1_tL^2_x}\notag \\
	&\lesssim \thq^{-1}\tau^{-\frac12}+\thq^{-7} \lbb^{-1} \tau^{-\frac12} + \thq^{-9}\sigma^{-1}\lesssim \lambda_{q+1}^{-\ve}.
\end{align}

We are now ready to verify the iterative estimates for $u_{q+1}$.

In view of  \eqref{uh3}, \eqref{uuql2}, \eqref{pdvh3}, \eqref{q+1 velocity-endpt2} and \eqref{principal h3 est-endpt2},
we have that for $a$ large enough,
similarly to \eqref{uh6},
\begin{align}
	\norm{u_{q+1}}_{L^\9_tH^3_x}
	& \lesssim\norm{\wt u_q}_{L^\9_tH^3_x}+\norm{w_{q+1}}_{L^\9_tH^3_x} \notag \\
	&\lesssim  \lambda_{q}^5+ \lambda_{q+1}^{5}\lesssim \lambda_{q+1}^5, \label{verifyuc1-endpt2} \\
	 \norm{\p_t u_{q+1}}_{L^\9_tH^2_x}
	 & \lesssim\norm{\p_t \wt u_q}_{L^\9_tH^2_x}+\norm{\p_t w_{q+1}}_{L^\9_tH^2_x} \notag \\
	 &\lesssim   \thq^{-1}\la^{5}+\mq\la^5+ \lambda_{q+1}^{8}\lesssim \lambda_{q+1}^8.   \label{verifyupth2-endpt2}
\end{align}
Moreover, we derive from \eqref{uuql2}, \eqref{e3.41} and \eqref{wql1} that
\begin{align}
	\norm{u_{q} - u_{q+1}}_{L^2_{t,x}} & \leq \norm{ u_{q} -\wt u_q }_{L^2_{t,x}} + \norm{\wt u_q  - u_{q+1}}_{L^2_{t,x}} \nonumber   \\
 &\lesssim  \norm{ u_q - \wt u_q }_{L^\9_tL^2_x}+ \norm{w_{q+1}}_{L^2_{t,x}}  \nonumber  \\
  &\lesssim  \lambda_q^{-3}+\delta_{q+1}^{\frac{1}{2}}   \leq M^*\delta_{q+1}^{\frac{1}{2}}, \label{e3.43}
\end{align}
for $M^*$ sufficiently large and
\begin{align}  \label{uql1l2}
	\norm{u_{q} - u_{q+1}}_{L^1_tL^2_x} 	
	&\lesssim \norm{ u_q -\wt u_q }_{L^\9_tL^2_x}+ \norm{w_{q+1}}_{L^1_tL^2_x} \nonumber \\
	&\lesssim \lambda_q^{-3}+\lambda_{q+1}^{-\ve} \leq \delta_{q+2}^{\frac{1}{2}}.
\end{align}

Concerning the iterative estimate \eqref{u-B-Lw-conv},
let us first show the embedding
\begin{align}\label{sobolevem2}
H^3_x\hookrightarrow W^{s,p}_x.
\end{align}

In order to prove \eqref{sobolevem2},
when $(s,p,\gamma)\in \mathcal{A}_2$,
we see that for $1\leq p\leq 2$,
\begin{align}\label{e7.9}
0\leq s< 2\a+\frac{2\a-2}{p}+(1-2\a)\leq 2\a-1<3.
\end{align}
which, via the embedding $H^3_x\hookrightarrow H_x^s \hookrightarrow W^{s,p}_x$,
yields \eqref{sobolevem2}.
Moreover, for $p>2$,
since $\alpha<2$, $s\geq 0$,
\begin{align}\label{e7.10}
0<1+\frac{2\a-2}{p}-s<1+\frac{2}{p}-\frac{2}{3}s,
\end{align}
which implies that $3/2>s-3/p$,
thereby yielding \eqref{sobolevem2}
by the Sobolev embedding.
We thus prove \eqref{sobolevem2}.

Thus, similar to \eqref{wtu-u}, by virtue of \eqref{nrh3}, \eqref{est-vih3} and \eqref{sobolevem2}, we have
\begin{align*}
 \norm{\wt u_q-u_q}_{L^\gamma_tW^{s,p}_x} & \lesssim \norm{\sum_i\chi_i(v_i-u_q)}_{L^\9_tH^3_x}\lesssim \la^{-2}.
\end{align*}
Therefore, for any $\a\in [1,2)$,
using \eqref{b-beta-ve}, \eqref{larsrp-endpt2}, \eqref{sobolevem2}
and Lemma \ref{totalest-endpt2},
we derive
\begin{align}\label{lw-est.2}
 	\norm{ u_{q+1} - u_q }_{L^\gamma_tW^{s,p}_x}
	 &\lesssim  \norm{\wt u_q-u_q}_{L^\gamma_tW^{s,p}_x}
 	+\norm{w_{q+1}}_{L^\gamma_tW^{s,p}_x} \notag\\
 	&\lesssim  \la^{-2}+ \thq^{-1}\laq^{s}\rs^{\frac{2}{p}-1}\tau^{\frac12-\frac{1}{\gamma}}
 	+\thq^{-30}\sigma^{-1} \notag\\
    &\lesssim   \la^{-2}+ \lambda_{q+1}^{s+2\a-1-\frac{2\a}{\gamma}-\frac{2\a-2}{p}+\ve(9-\frac{16}{p}) }
	+ \lbb_{q+1}^{-\ve} .
\end{align}
Taking into account that, by \eqref{ne3.1},
\begin{align}\label{endpt2-condition}
s+2\a-1-\frac{2\a}{\gamma}-\frac{2\a-2}{p}+\ve(9-\frac{16}{p})
\leq s+2\a-1-\frac{2\a}{\gamma}-\frac{2\a-2}{p}+9\ve   <-10\ve,
\end{align}
we thus obtain
\begin{align}
\norm{ u_{q+1} - u_q }_{L^\gamma_tW^{s,p}_x} \leq \delta_{q+2}^{\frac12}. \label{nne6.6}
\end{align}

Therefore, the iterative estimates \eqref{uh3}, \eqref{upth2}, \eqref{u-B-L2tx-conv} and \eqref{u-B-Lw-conv} are verified.

\subsection{Reynolds stress} \label{Subsec-Reynolds-Endpt2}
Below we treat the Reynolds stress for the endpoint case
$(2\alpha/\gamma+1-2\alpha, \gamma, \infty)$.
We derive from equation \eqref{equa-nsr} at level $q+1$ that
the new Reynolds stress satisfies the equation
\begin{align}
		\displaystyle\div\mathring{R}_{q+1} - \nabla P_{q+1}
		&\displaystyle = \underbrace{\partial_t (w_{q+1}^{(p)}+w_{q+1}^{(c)}) +\nu(-\Delta)^{\alpha} w_{q+1} +\div\big(\wt u_q \otimes w_{q+1} + w_{q+ 1} \otimes \wt u_q \big) }_{ \div\mathring{R}_{lin} +\nabla P_{lin} }   \notag\\
		&\displaystyle\quad+ \underbrace{\div (w_{q+1}^{(p)} \otimes w_{q+1}^{(p)}+  \tr_q)+\partial_t \wo}_{\div\mathring{R}_{osc} +\nabla P_{osc}}  \notag\\
		&\displaystyle\quad+ \underbrace{\div\Big((w_{q+1}^{(c)} +\wo)\otimes w_{q+1}+ w_{q+1}^{(p)} \otimes (w_{q+1}^{(c)} +\wo) \Big)}_{\div\mathring{R}_{cor} +\nabla P_{cor}}. \label{ru-endpt2}
\end{align}
Then, using the inverse divergence operator $\mathcal{R}$
we can choose the Reynolds stress at level $q+1$ by
\begin{align}\label{rucom-endpt2}
	\mathring{R}_{q+1} := \mathring{R}_{lin} +   \mathring{R}_{osc}+ \mathring{R}_{cor},
\end{align}
where the linear error
\begin{align}
	\mathring{R}_{lin} & := \mathcal{R}\(\partial_t (w_{q+1}^{(p)} +w_{q+1}^{(c)}  )\)
	+ \nu \mathcal{R} (-\Delta)^{\a} w_{q+1} + \mathcal{R}\P_H \div \(\wt u_q \mathring{\otimes} w_{q+1} + w_{q+ 1}
	\mathring{\otimes} \wt u_q\), \label{rup-endpt2}
\end{align}
the oscillation error
\begin{align}\label{rou}
	\mathring{R}_{osc} :=& \sum_{k \in \Lambda } \mathcal{R} \P_H\P_{\neq 0}\left(\g^2 \P_{\neq 0}(W_{(k)}\otimes W_{(k)})\nabla (a_{(k)}^2)\right) \notag\\
	&-\sigma^{-1}\sum_{k\in \Lambda}\mathcal{R} \P_H \P_{\neq 0}\(h_{(\tau)}\aint_{\T^3}W_{(k)}\otimes W_{(k)}\d x\p_t\nabla(a_{(k)}^{2})\),
\end{align}
and the corrector error
\begin{align}
	\mathring{R}_{cor} &
	:= \mathcal{R} \P_H \div \bigg( w^{(p)}_{q+1} \mathring{\otimes} (w_{q+1}^{(c)} +\wo)
      + (w_{q+1}^{(c)} +\wo) \mathring{\otimes} w_{q+1} \bigg). \label{rup2}
\end{align}

In the following we verify the inductive estimates for the new Reynolds stress $\mathring R_{q+1}$. \\

\paragraph{\bf Verification of $L^\9_tH^N_x$-estimate of Reynolds stress.}
Since by \eqref{principal h3 est-endpt2} and \eqref{pth2 est-endpt2},
the velocity perturbations obey the same upper bounds as in
\eqref{principal h3 est-endpt1} and \eqref{pth2 est-endpt1},
we can argue as in a similar manner as in \eqref{ine-rq1h3}-\eqref{ul9}
and obtain \eqref{rh3} and \eqref{rh4} in the supercritical regime
$\mathcal{A}_2$. The details are omitted here. \\

\paragraph{\bf Verification of $L^1_{t,x}$-decay of Reynolds stress}
We aim to verify the $L^1_{t,x}$-decay \eqref{rl1} of the Reynolds stress $\mathring{R}_{q+1}$ at level $q+1$.
In this case, we choose
\begin{align}\label{defp}
\rho: =\frac{2\a-2+16\varepsilon}{2\a-2+14\varepsilon}\in (1,2),
\end{align}
where $\ve$ is given by \eqref{ne3.1}.
Then, we have
\begin{equation}\label{setp}
	(1-\a-8\ve)(\frac{2}{\rho}-1)=1-\a-6\ve,
\end{equation}
and
\begin{align}  \label{rs-rp-p-ve-endpt2}
\rs^{\frac 2\rho-1} =  \lambda^{1-\a-6\ve}.
\end{align}

\paragraph{\bf (i) Linear error.}
Note that, by Lemmas \ref{Lem-gk-esti},
\ref{buildingblockestlemma-endpt2}
and \ref{mae-endpt2}, \eqref{ne3.1}, \eqref{larsrp-endpt2}, \eqref{div free velocity}
and \eqref{rs-rp-p-ve-endpt2},
\begin{align}
	 & \| \mathcal{R}\partial_t( w_{q+1}^{(p)}+ w_{q+1}^{(c)})\|_{L_t^1L_x^\rho}  \nonumber \\
	\lesssim& \sum_{k \in \Lambda}\| \mathcal{R} \curl\curl\partial_t(\g a_{(k)} W^c_{(k)}) \|_{L_t^1L_x^\rho} \nonumber \\
	\lesssim& \sum_{k \in \Lambda}\Big(\| \g\|_{L^1_t}\|  a_{(k)} \|_{C_{t,x}^2}\| W^c_{(k)} \|_{C_t W_x^{1,\rho}}+\| \p_t\g\|_{L_t^1}\| a_{(k)} \|_{C_{t,x}^1}\| W^c_{(k)} \|_{C_t W_x^{1,\rho}}\Big)\nonumber \\
    \lesssim& \thq^{-14}\tau^{-\frac12}\rs^{\frac{2}{\rho}-1}\lambda^{-1} + \thq^{-7}\sigma\tau^{\frac12}\rs^{\frac{2}{\rho}-1}\lambda^{-1}
	\lesssim  \thq^{-7}\lambda^{-4\ve}.\label{time derivative-endpt2}
\end{align}
For the viscosity term, by \eqref{velocity perturbation-endpt2},
\begin{align}
	\norm{ \nu\mathcal{R}(-\Delta)^{\alpha} w_{q+1} }_{L_t^1L^\rho_x} \lesssim & \norm{ \nu\mathcal{R}(-\Delta)^{\alpha} w_{q+1}^{(p)} }_{L_t^1L^\rho_x}+\norm{ \nu \mathcal{R}(-\Delta)^{\alpha} w_{q+1}^{(c)} }_{L_t^1L^\rho_x}+\norm{ \nu\mathcal{R}(-\Delta)^{\alpha} \dqo }_{L_t^1L^\rho_x}.\label{e5.17.2}
\end{align}
Note that,
by the interpolation estimate, \eqref{larsrp-endpt2},
\eqref{uprinlp-endpt2}, \eqref{rs-rp-p-ve-endpt2}
and the fact that $2-\alpha \geq  5\va$,
\begin{align}
	\norm{ \nu\mathcal{R}(-\Delta)^{\alpha} w_{q+1}^{(p)} }_{L_t^1L^\rho_x}
    & \lesssim \norm{ |\na|^{2\a-1} w_{q+1}^{(p)} }_{L_t^1L^\rho_x}\notag\\
	& \lesssim  \norm{w_{q+1}^{(p)}}_{L_t^1L^\rho_x} ^{\frac{4-2\a}{3}} \norm{w_{q+1}^{(p)}}_{L_t^1W^{3,\rho}_x} ^{\frac{2\a-1}{3}}\notag\\
	& \lesssim \thq^{-1}\lbb^{2\alpha-1}\rs^{\frac{2}{\rho}-1}\tau^{-\frac12}\lesssim \thq^{-1}\lambda^{-6\ve}.\label{e5.18.2}
\end{align}
Similarly, by Lemma \ref{totalest-endpt2},
\begin{align}
	&\norm{ \nu\mathcal{R}(-\Delta)^{\alpha} w_{q+1}^{(c)} }_{L_t^1L^\rho_x}
     \lesssim \thq^{-7}\lbb^{2\alpha-2}\rs^{\frac{2}{\rho}-1}\tau^{-\frac12}\lesssim \thq^{-7}\lambda^{-1-6\ve},\label{e5.19.2}\\
	&\norm{\nu \mathcal{R}(-\Delta)^{\alpha} \dqo }_{L_t^1L^\rho_x}   \lesssim \thq^{-30}\sigma^{-1}\lesssim \thq^{-30}\lambda^{-2\ve}.\label{e5.21.2}
\end{align}
Hence, combining \eqref{e5.17.2}-\eqref{e5.21.2} altogether we obtain
\begin{align}  \label{mag viscosity-endpt2}
	\norm{ \nu\mathcal{R}(-\Delta)^{\alpha} w_{q+1} }_{L_t^1L^\rho_x} \lesssim \thq^{-30}\lambda^{-2\ve} .
\end{align}

Moreover,
using \eqref{nuh3}, Lemma \ref{totalest-endpt2}
and \eqref{rs-rp-p-ve-endpt2} we have
\begin{align} \label{linear estimate1-endpt2}
	&\norm{ \mathcal{R}\P_H\div\(w_{q + 1} \otimes \wt u_q + \wt u_q \otimes w_{q+1}\) }_{L_t^1L^\rho_x}  \nonumber \\	
    \lesssim\,&\norm{w_{q + 1} \otimes \wt u_q + \wt u_q \otimes w_{q+1} }_{L_t^1L^\rho_x}  \nonumber \\
	\lesssim\,& \norm{\wt u_q}_{L^\9_tH^3_x} \norm{w_{q+1}}_{L_t^1L^\rho_x} \nonumber \\
	\lesssim\, &\lambda^5_q (\thq^{-1} \rs^{\frac{2}{\rho}-1} \tau^{-\frac 12} +\thq^{-9}\sigma^{-1} )
    \lesssim \thq^{-10}\lambda^{-2\ve}.
\end{align}

Therefore,
we conclude from \eqref{time derivative-endpt2}, \eqref{mag viscosity-endpt2} and \eqref{linear estimate1-endpt2} that
\begin{align}   \label{linear estimate-endpt2}
	\norm{\mathring{R}_{lin} }_{L_t^1L^\rho_x}
     & \lesssim \thq^{-7}\lambda^{-4\ve} +\thq^{-30}\lambda^{-2\ve}+\thq^{-10}\lambda^{-2\ve}
       \lesssim \thq^{-30}\lambda^{-2\ve}.
\end{align}

\paragraph{\bf (ii) Oscillation error.}
Unlike in the previous endpoint case,
we only need to  decompose the oscillation error into two parts here:
\begin{align*}
	\mathring{R}_{osc} = \mathring{R}_{osc.1} +  \mathring{R}_{osc.2},
\end{align*}
where the low-high spatial oscillation error
\begin{align*}
	\mathring{R}_{osc.1}
	&:=   \sum_{k \in \Lambda }\mathcal{R} \P_{H}\P_{\neq 0}\left(\g^2 \P_{\neq 0}(W_{(k)}\otimes W_{(k)} )\nabla (a_{(k)}^2) \right),
\end{align*}
and the low frequency error
\begin{align*}
	\mathring{R}_{osc.2} &
      := -\sigma^{-1}\sum_{k\in \Lambda}\mathcal{R}\P_{H}\P_{\neq 0}
       \(h_{(\tau)}\aint_{\T^3}W_{(k)}\otimes W_{(k)} \d x\, \p_t\nabla(a_{(k)}^{2})\).
\end{align*}

Then, applying Lemmas \ref{buildingblockestlemma-endpt2}, \ref{mae-endpt2} and \ref{commutator estimate1}
with $a = \nabla (a_{(k)}^2)$ and $f =  \phi_{(k)}^2$
we get
\begin{align}  \label{I1-esti}
	\norm{\mathring{R}_{osc.1} }_{L^1_tL^\rho_x}
	&\lesssim  \sum_{ k \in \Lambda }
	\|\g\|_{L^2_t}^2\norm{|\nabla|^{-1} \P_{\not =0}
		\left(\P_{\geq (\lambda \rs/2)}(W_{(k)}\otimes W_{(k)} )\nabla (a_{(k)}^2)\right)}_{C_tL^\rho_x} \notag \nonumber  \\
	& \lesssim \sum_{ k \in \Lambda }
	   \||\na|^3 (a^2_{(k)})\|_{C_{t,x}}  \lambda^{-1} \rs^{-1}\norm{\phi^2_{(k)} }_{C_tL^{\rho}_x}  \nonumber  \\
	& \lesssim \thq^{-23}  \lambda^{-1}  \rs^{\frac{2}{\rho}-3}.
\end{align}

Moreover, as in \eqref{I3-esti-endpt1},
the low frequency part $\mathring{R}_{osc.2} $
can be estimated by using \eqref{hk-esti} and \eqref{mag amp estimates-endpt2}:
\begin{align}  \label{I3-esti}
	\norm{\mathring{R}_{osc.2} }_{L^1_tL^\rho_x}
    \lesssim \sigma^{-1} \sum_{k\in\Lambda} \|h_{(\tau)}\|_{C_t}\( \norm{a_{(k)} }_{C_{t,x}} \norm{a_{(k)} }_{C_{t,x}^2} +\norm{a_{(k)} }_{C_{t,x}^1}^2\)
	\lesssim \thq^{-15} \sigma^{-1}.
\end{align}

Therefore, combing \eqref{I1-esti} and \eqref{I3-esti} altogether and using \eqref{larsrp-endpt2} and \eqref{rs-rp-p-ve-endpt2}
we conclude
\begin{align}
	\label{oscillation estimate-endpt2}
	\norm{\mathring{R}_{osc}}_{L_t^1L^\rho_x}
    &\lesssim   \thq^{-23}  \lambda^{-1} \rs^{\frac{2}{\rho}-3}+\thq^{-15} \sigma^{-1} \notag \\
	&\lesssim \thq^{-23} \lbb^{\a-2+10\ve} +\thq^{-15} \lbb^{-2\va} \notag\\
	&\lesssim \thq^{-15} \lbb^{-2\va},
\end{align}
where the last step is due to \eqref{ne3.1}.

\paragraph{\bf (iii) Corrector error.}
We use H\"older's inequality, Lemma \ref{totalest-endpt2}
and \eqref{rs-rp-p-ve-endpt2} to get
\begin{align}
	\norm{\mathring{R}_{cor} }_{L^1_{t}L^{\rho}_x}
	\lesssim& \norm{ w_{q+1}^{(p)} \otimes (w_{q+1}^{(c)} +\wo) -(w_{q+1}^{(c)} +\wo) \otimes w_{q+1} }_{L^1_{t}L^{\rho}_x} \notag \\
	\lesssim& \norm{w_{q+1}^{(c)} +\wo }_{L^2_{t}L^{\9}_x} (\norm{w^{(p)}_{q+1} }_{L^2_{t}L^{\rho}_x} + \norm{w_{q+1} }_{L^2_{t}L^{\rho}_x})\notag  \\
     \lesssim&  \( \thq^{-7}\lbb^{-1}\rs^{-1 }+\thq^{-9}\sigma^{-1}\) \(\thq^{-1} \rs^{\frac{2}{\rho}-1}  +\thq^{-7}\lbb^{-1} \rs^{\frac{2}{\rho}-1}  + \thq^{-9} \sigma^{-1}\) \notag \\
      \lesssim&  \( \thq^{-7}\lbb^{\a-2+8\ve}+\thq^{-9}\lambda^{-2\ve}\) \(\thq^{-1} \lbb^{-\a+1-6\ve} + \thq^{-9}\lambda^{-2\ve}\) \notag \\
	 \lesssim& \, \thq^{-18} \lambda^{-4\ve}. \label{corrector estimate-endpt2}
\end{align}

Therefore,
we conclude from  estimates \eqref{linear estimate-endpt2},
\eqref{oscillation estimate-endpt2},
\eqref{corrector estimate-endpt2} that
\begin{align} \label{rq1b}
	\|\mathring{R}_{q+1} \|_{L^1_{t,x}}
	&\leq \| \mathring{R}_{lin} \|_{L^1_tL^\rho_{x}} +  \| \mathring{R}_{osc}\|_{L^1_tL^\rho_{x}}
	+  \|\mathring{R}_{cor} \|_{L^1_tL^\rho_{x}}  \nonumber  \\
	&\lesssim   \thq^{-30}\lambda^{-2\varepsilon}
	  +\thq^{-15}\lambda^{-2\ve} + \thq^{-18}\lambda^{-4\ve} \nonumber  \\
	& \leq \lambda^{-\ve_R} \delta_{q+2}.
\end{align}
This justifies the inductive estimate \eqref{rl1}
for the $L^1_{t,x}$-norm of the new Reynolds stress $\mathring{R}_{q+1}$.

\section{Proof of main results}  \label{Sub-Proof-Main}

This section contains the proofs of the main results,
i.e. Theorems \ref{Prop-Iterat} and \ref{Thm-Non-hyper-NSE},
Corollaries \ref{Cor-Strong-Nonuniq} and \ref{Cor-Nonuniq-Supercri},
and the strong vanishing viscosity result in Theorem \ref{Thm-hyperNSE-Euler-limit}.

\paragraph{\bf Proof of Theorem \ref{Prop-Iterat}}

Because the iterative estimates \eqref{uh3}-\eqref{rl1} and \eqref{u-B-L2tx-conv}-\eqref{u-B-Lw-conv}
have been verified in the previous sections,
we only need to prove the well-preparedness of $(u_{q+1}, \rr_{q+1})$ and the temporal inductive inclusion \eqref{suppru}.

Regarding the well-preparedness of $(u_{q+1}, \rr_{q+1})$,
we first note from the support of $a_{(k)}$ that
\begin{align*}
w_{q+1}(t)=0 \quad \text{if} \quad \operatorname{dist}(t,I_{q+1}^c)\leq \thq.
\end{align*}
Therefore, $u_{q+1}(t)=\wt u_q(t)$ if $\operatorname{dist}(t,I_{q+1}^c)\leq \thq$.
Then, by the well-preparedness of $(\wt u_q, \tr_q)$,
we infer that
\begin{align*}
\rr_{q+1}(t)=\tr_q(t)=0 \quad \text{if} \quad \operatorname{dist}(t,I_{q+1}^c)\leq \thq,
\end{align*}
which verifies the well-preparedness of $(u_{q+1}, \rr_{q+1})$.

Concerning the temporal inductive inclusion \eqref{suppru},
first note that
\begin{align}
	& \supp_t w_{q+1} \subseteq \bigcup_{k\in \Lambda }\supp_t a_{(k)} \subseteq N_{2\thq}(\supp_t \tr_{q}).  \label{e4.43}
\end{align}

Next we prove that
\begin{align}
	& I_{q+1}\subseteq N_{4T/\mq}(\supp_t \rr_{q}). \label{ne4.43}
\end{align}

To this end, for any $t\in I_{q+1}$,
we have $t\in [t_i-2\thq,t_i+3\thq]$ for some $i\in \mathcal{C}$.
Hence, there exists $t_*\in [t_{i-1}, t_i+\thq]$
such that $\rr_q(t_*)\neq 0$. Since
\begin{align*}
  |t-t_*|\leq \frac{T}{\mq}+3\thq < \frac{4T}{\mq},
\end{align*}
we infer that $t\in N_{4T/\mq}(\supp_t \rr_{q})$,
which proves \eqref{ne4.43}, as claimed.

Combining \eqref{e4.43} and \eqref{ne4.43} together we obtain
\begin{align}
	& \supp_t w_{q+1} \subseteq  N_{2\thq}(\supp_t \tr_{q})
	\subseteq  N_{2\thq}(I_{q+1})\subseteq  N_{6T/\mq}(\supp_t \rr_{q}).  \label{e4.45}
\end{align}

Regarding the temporal support of $\wt u_{q}$, we claim that
\begin{align}\label{suppwtuq}
\supp_t \wt u_{q}\subseteq N_{2T/\mq}(\supp_tu_{q}).
\end{align}
To this end, for any $t\in [0,T]$ such that $\wt u_q(t)\neq 0$,
there exists $0\leq i\leq \mq-1$ such that $t\in [t_i,t_{i+1}]$.
If $t\in [t_i+\thq,t_{i+1}]$, then we have $u_q(t_i)\neq 0$,
otherwise $\wt u_q(t)=v_i(t)=0$.
Since $|t-t_i|\leq T/\mq$,
we see that \eqref{suppwtuq} is valid.
Moreover, if $t\in [t_i, t_i+\thq]$,
we have $u_q(t_i)\neq 0$ or $u_q(t_{i-1})\neq 0$.
Actually, if $u_q(t_i) = u_q(t_i) =0$,
the uniqueness in Proposition \ref{Prop-LWP-Hyper-NLSE}
yields that
$v_i(t) = v_{i-1}(t)=0$,
and so
\begin{align*}
\wt u_q(t) =\chi_i(t) v_i(t)+(1-\chi_i(t))v_{i-1}(t)=0,
\end{align*}
which contradicts the fact that $\wt u_q(t)\not = 0$.
Hence, taking into account
$|t-t_i|\leq T/\mq$ and $|t-t_{i-1}|\leq 2(T/\mq)$
we prove \eqref{suppwtuq}, as claimed.

Thus, we deduce from \eqref{e4.45} and \eqref{suppwtuq} that
\begin{align}
	 \supp_t u_{q+1}
	 \subseteq \supp_t \wt u_{q} \cup \supp_t w_{q+1}
	\subseteq N_{6T/\mq}( \supp_t (u_{q}, \rr_{q}))
	\subseteq N_{\delta_{q+2}^{\frac 12}} ((\supp_t(u_{q}, \rr_{q})), \label{suppbq}
\end{align}
where the last step is due to
$6T/\mq \ll \delta_{q+2}^{1/2}$.

Moreover,
in view of  \eqref{e4.43}, \eqref{e4.45} and \eqref{suppwtuq},
we get
\begin{align}
	& \supp_t \rr_{q+1}\subseteq \bigcup\limits_{k\in \Lambda}\supp_t a_{(k)}
             \cup \supp_t \wt u_q \subseteq  N_{\delta_{q+2}^{\frac12}}( \supp_t (u_{q}, \rr_{q})). \label{supp-Ru-RB-q+1}
\end{align}
Therefore, putting \eqref{suppbq} and \eqref{supp-Ru-RB-q+1}
altogether we prove Theorem \ref{Prop-Iterat}.
\hfill $\square$ \\

\paragraph{\bf Proof of Theorem \ref{Thm-Non-hyper-NSE}}  \label{Proof-Nonuniq}
Below we prove the statements $(i)$-$(iv)$ in Theorem~\ref{Thm-Non-hyper-NSE} below.

$(i)$. Take $u_0=\tilde{u}$ and set
\begin{align}
	&\mathring{R}_0 :=\mathcal{R}\(\p_t u_0+\nu(-\Delta)^{\alpha} u_0\) + u_0\mathring\otimes u_0, \label{r0u}  \\
	& P_0 := -\frac{1}{3} |u_0|^2.
\end{align}
Thus, $(u_0, \rr_0)$ is a well-prepared solution to \eqref{equa-nsr}
with the set $I_0 = [0,T]$ and the length scale $\theta_0=T$.
Let $\delta_{1}:= \|\mathring{R}_0 \|_{L^1_{t,x}}$
and choose $a$ sufficiently large such that \eqref{uh3}-\eqref{rl1} are satisfied at level $q=0$.
Then, in view of Theorem~\ref{Prop-Iterat}, there exists a sequence of solutions $\{u_{q},\rr_{q}\}_{q}$ to \eqref{equa-nsr} satisfying the inductive estimates \eqref{uh3}-\eqref{suppru} for all $q\geq 0$.

Note that, by \eqref{uh3} and \eqref{upth2},
\begin{align}\label{ine-uuqh1}
\norm{ u_{q+1} - u_q }_{H^1_{t,x}} & \leq \norm{\p_t( u_{q+1} - u_q) }_{L^\9_tH^2_{x}}+ \norm{ u_{q+1} - u_q }_{L^\9_tH^3_{x}} \notag\\
 & \leq \norm{\p_t u_{q+1}}_{L^\9_tH^2_{x}}+\norm{\p_t u_{q}}_{L^\9_tH^2_{x}}+\norm{u_{q+1}}_{L^\9_tH^3_{x}}+ \norm{ u_{q}}_{L^\9_tH^3_{x}}\notag\\
  &\lesssim \laq^{8}+\la^8+\laq^5+\la^5\lesssim \laq^{8}.
\end{align}
Then, using the interpolation, \eqref{la}, \eqref{uh3},
\eqref{u-B-L2tx-conv} and \eqref{ine-uuqh1}
we infer that for any $\beta'\in (0,\frac{\beta}{8+\beta})$,
\begin{align}
\sum_{q \geq 0} \norm{ u_{q+1} - u_q }_{H^{\beta'}_{t,x}}
	\leq  & \, \sum_{q \geq 0} \norm{ u_{q+1} - u_q }_{L^2_{t,x}}^{1- \beta'}\norm{ u_{q+1} - u_q }_{H^1_{t,x}}^{\beta'}\notag\\
	\lesssim  &\,  \sum_{q \geq 0} (M^*)^{1-\beta'} \delta_{q+1}^{\frac{1-\beta'}{2}}\lambda_{q+1}^{8 \beta' } \notag\\
	\lesssim &\, (M^*)^{1-\beta'} \delta_{1}^{\frac{1-\beta'}{2}}\lambda_{1}^{8 \beta' } +
	\sum_{ q \geq 1} (M^*)^{1-\beta'} \lambda_{q+1}^{-\beta(1 - \beta')  + 8\beta'  } <\9, \label{interpo}
\end{align}
where the last step is due to the inequality $-\beta(1 - \beta')  + 8\beta' <0$.

Therefore, $\{u_q\}_{q\geq 0}$ is a Cauchy sequence in $H^{\beta'}_{t,x}$,
and thus there exists $u\in H^{\beta'}_{t,x}$ such that $\lim_{q\rightarrow\infty}(u_q)=u$ in $H^{\beta'}_{t,x}$.
Taking into account the fact that
$\lim_{q \to \infty} \mathring{R}_{q} = 0 $ in $L^1_{t,x}$
and $u_q(0)=\wt u(0)$ for all $q\geq 0$,
we conclude that $u$ is a weak solution to \eqref{equa-NS}
with the initial datum $\wt u(0)$.

$(ii)$. Regarding the regularity of the weak solution $u$,
by virtue of \eqref{ne6.6} and \eqref{nne6.6}, we have
\begin{align}
	\sum_{q \geq 0}\norm{ u_{q+1} - u_q }_{L^\gamma_tW^{s,p}_x} < \9,   \label{result-lw}
\end{align}
which yields that $\{u_q\}_{q\geq 0}$
is also a Cauchy sequence in $L^\gamma_tW^{s,p}_x$.
Thus, using the uniqueness of weak limits
we obtain
$$u \in H^{\beta^\prime}_{t,x}  \cap  L^\gamma_tW^{s,p}_x,$$
thereby proving the regularity statement $(ii)$.

$(iii).$ Regarding the Hausdorff measure of the singular set,
we set
\[
\mathcal{G} =   \bigcup_{q \geq 0} I_q^c \setminus \{0,T\}, \ \
\mathcal{B}:=[0,T] \setminus \mathcal{G}.
\]
By construction, $u_q$ is a smooth solution to the hyperdissipative Navier-Stokes equation \eqref{equa-NS}
on $\mathcal{G}$,
and  $u \equiv  u_q $ on $I_{q}^c$ for each $q$.
Thus, $\mathcal{B}$ contains the singular set of time.
Since by \eqref{Iq1-C-def},
each $I_q$ is covered by at most $m_{q}=\theta_q^{-{\eta}}$ many balls of radius $5\theta_q$,
the Hausdorff dimension of the potential singular set $\mathcal{B}$ can be estimated by
\begin{align*}
d_{\mathcal{H}} ( \mathcal{B}) = d_{\mathcal{H}} (\bigcap_{q\geq 0}I_q )=d_{\mathcal{H}} (\limsup_q I_q )\leq \eta<\eta_* .
\end{align*}

$(iv).$ Concerning the small deviations of temporal supports,
we note that
\begin{align}
	& \supp_t \mathring{R}_0
	\subseteq K_0 := \supp_t  u_0
	= \supp_t \wt u.
\end{align}
Then, set
\begin{align}
   K_q:= \supp_t u_q\cup \supp_t \mathring{R}_q, \ \ q\geq 1.
\end{align}
Using \eqref{suppru} we have
\begin{align}
	 K_{q+1} \subseteq N_{\delta_{q+2}^\frac 12} K_{q}
      \subseteq \cdots
      \subseteq N_{\sum\limits_{j=2}^{q+2}\delta_{j}^\frac 12} K_{0}.
\end{align}
Thus, taking into account
$\sum_{q\geq 0}\delta_{q+2}^{1/2}\leq \ve_*$
for $a$ large enough
we get
\begin{align}
	&\supp_t u	\subseteq \bigcup_{q\geq 0} K_q
	\subseteq N_{\ve_*} ( \supp_t \tilde{u}).
\end{align}
This verifies the temporal support statement $(iii)$.

$(v).$ Finally, for the small deviations on average,
we infer from \eqref{u-B-Lw-conv} that
\begin{align}
	\norm{ u - \tilde{u} }_{L^1_tL^2_x}+\norm{ u - \tilde{u} }_{L^\gamma_tW^{s,p}_x}
	\leq &\,\sum_{q \geq 0}(\norm{  u_{q+1} - u_q }_{L^1_tL^2_x}+ \norm{ u_{q+1} - u_q }_{L^\gamma_tW^{s,p}_x})\notag \\
	\leq &\,2\sum_{q\geq 0} \delta_{q+2} ^{\frac12} \leq 2\sum_{q\geq 2} a^{-\beta b^q} \leq 2\sum_{q\geq 2} a^{-\beta bq}
	=\frac{2a^{-2\beta b}}{1-a^{-\beta b}}\leq \ve_*, \label{e6.5}
\end{align}
where the last inequality holds
for $a$ large enough (depending on $\ve_*$).

Therefore, the proof of Theorem~\ref{Thm-Non-hyper-NSE} is complete.
\hfill $\square$  \\

\paragraph{\bf Proof of Corollary \ref{Cor-Strong-Nonuniq}.}
Let $\wt u$ be a weak solution to \eqref{equa-NS} with
divergence-free $\wt u(0)=\wt u_0\in L^2$.

If $\wt u$ is not a Leray-Hopf solution,
then due to Lions \cite{lions69},
there exists a unique smooth Leray-Hopf solution $u$ to \eqref{equa-NS} on $[0,T]$.
Hence, $u$ is different from $\wt u$.

If $\wt u$ is a Leray-Hopf solution to \eqref{equa-NS},
we may choose any smooth, divergence-free and mean-free vector field $v$ on $[0,T]$
and set
\begin{align*}
  u_m^*:=\wt u+\frac{m}{c_0}v,
\end{align*}
where $m \in \mathbb{N}_+$ and $c_0:= \|v\|_{L^1([T/2,T];L^2_x)} (>0)$.
Then, we glue $\wt u$ and $u^*_m$ together by
\begin{align}\label{def-wtu1}
\wt u_m:= \chi \wt u+(1-\chi)u^*_m,
\end{align}
where $\chi\in C^\9([0,T])$ satisfies
\begin{align}\label{def-chi2}
\chi_{i}= \begin{cases}1 & \text { if } 0 \leq t \leq \frac{T}{4},
\\ 0 & \text { if }  \frac{T}{2} \leq t\leq T.\end{cases}
\end{align}
Note that,
$\wt u_m$ is a smooth, divergence-free and mean-free vector field on $[0,T]$,
such that
\begin{align*}
  \wt u_m(0)=\wt u_0, \ \
  \wt u_m|_{[\frac T 2, T]} = u_m^*|_{[\frac T 2, T]}.
\end{align*}

Then, for any $\ve_* \in (0,1/4)$,
Theorem~\ref{Thm-Non-hyper-NSE}
gives weak solutions $u_m\in H^{\beta'}_{t,x} \cap L^\gamma_tW^{s,p}_x$
to \eqref{equa-NS} on $[0,T]$,
$m\geq 1$,
which are smooth outside a null set in time and
satisfy
\begin{align}\label{condition-u1}
u_m(0)=\wt u_m(0)=\wt u_0 \quad \text{and}\quad \|u_m-\wt u_m\|_{L^1([0,T];L^2_x)}\leq \va_*.
\end{align}

Hence, taking into account
\begin{align*}
  \|u_m-\wt u\|_{L^1(T/2,T;L^2_x)}&\geq\|\wt u_m - \wt u\|_{L^1(T/2,T;L^2_x)}- \|\wt u_m -u_m\|_{L^1(T/2,T;L^2_x)}\\
  &= \| u^*_m - \wt u\|_{L^1(T/2,T;L^2_x)}-\|\wt u_m -u_m\|_{L^1(T/2,T;L^2_x)}\\
  &\geq m-\ve_*>\frac12,
\end{align*}
which yields that $u_m\not = \wt u$ on $[0,T]$ for every $m\geq 1$.

Thus, we conlcude that for any weak solution $\wt u$ to \eqref{equa-NS},
there exists another different weak solution to \eqref{equa-NS}
in the space $L^\gamma_tW^{s,p}_x$
with the same initial datum $\wt u(0)$,
where $(s,\gamma, p) \in \mathcal{A}_1 \cup \mathcal{A}_2$.

Moreover, for any $m,m'\geq 1$, $m\neq m'$,
we have
\begin{align*}
  \|u_m-u_{m'}\|_{L^1(T/2,T;L^2_x)}&\geq\|\wt u_m - \wt u_{m'}\|_{L^1(T/2,T;L^2_x)}- \|\wt u_m -u_m\|_{L^1(T/2,T;L^2_x)}
  -\|\wt u_{m'} -u_{m'}\|_{L^1(T/2,T;L^2_x)}\\
  &= \| u^*_m - u^*_{m'}\|_{L^1(T/2,T;L^2_x)}- \|\wt u_m -u_m\|_{L^1(T/2,T;L^2_x)}
  -\|\wt u_{m'} -u_{m'}\|_{L^1(T/2,T;L^2_x)}\\
  &\geq |m-m'|-2\ve_*> \frac12,
\end{align*}
which yields that $u_m\not = u_{m'}$ on $[0,T]$.
Therefore, there exist infinitely many different weak solutions to \eqref{equa-NS}
with the same initial datum $\wt u(0)$.
We finish the proof of Corollary~\ref{Cor-Strong-Nonuniq}.
\hfill $\square$ \\

\paragraph{\bf Proof of Corollary \ref{Cor-Nonuniq-Supercri}.}
Taking $s=0$ and $\gamma =\infty$ in the supercritical regime $\mathcal{A}_1$
and then applying Theorem \ref{Thm-Non-hyper-NSE} we obtain the non-uniqueness
in the supercritical spaces $L^p_x$
for any $1\leq p< 3/(2\alpha-1)$.
Note that,
the initial data of non-unique solutions in Theorem \ref{Thm-Non-hyper-NSE}
are in $L^2_x$, and hence also in $L^p_x$,
due to the embeddings on torus:
$L^2_x\hookrightarrow L^{3/(2\alpha-1)}_x \hookrightarrow L^p_x$
as $\alpha \in [5/4,2)$.

Regarding $(ii)$, for any $s<s_p:=3/p+1-2\alpha$,
we may find a small constant $\eta>0$ such that
$s<s_p-\eta$.
Then, using the embedding theorems (cf.\cite[p.164]{ST87}) we infer
\begin{align*}
	F^{s_p-\eta}_{p,q}
	\hookrightarrow B^{s_p-\eta}_{p,p\vee q}
	\hookrightarrow B^{s}_{p,q}.
\end{align*}
Then, letting $\wt p := \frac{3}{2\alpha-1+\eta}< \frac{3}{2\alpha-1}$
(we may take $\eta$ even smaller such that $\wt p>1$
and using the embedding of Triebel-Lizorkin spaces
(cf. \cite[p.170]{ST87}) we get
\begin{align*}
	L^{\wt p} = F^{0}_{\wt p,2}
	\hookrightarrow F^{s_p-\eta}_{p,q}.
\end{align*}
Thus, we obtain
\begin{align*}
	L^{\wt p} \hookrightarrow B^{s}_{p,q},
\end{align*}
which along with the statement $(i)$ yields
the existence of non-unique weak solutions in ${B}^{s}_{p,q}$.

Similarly, the last statement $(iii)$ follows from $(i)$
and the following embedding (cf. \cite[p.165]{ST87})
\begin{align*}
	L^{\wt p}
	\hookrightarrow F^{s_p-\eta}_{p,q}
	\hookrightarrow F^{s}_{p,q}.
\end{align*}
The proof is therefore complete.
\hfill $\square$ \\

\paragraph{\bf Proof of Theorem~\ref{Thm-hyperNSE-Euler-limit}}
We choose two families of standard compactly support Friedrichs mollifiers $\left\{\phi_{\varepsilon}\right\}_{\varepsilon>0}$
and $\left\{\varphi_{\varepsilon}\right\}_{\varepsilon>0}$ on $\T^{3}$ and $\R$, respectively.
Set
\begin{align} \label{un-u-Bn-B}
	u_{n} :=\left(u *_{x} \phi_{\lambda_{n}^{-1}}\right) *_{t} \varphi_{\lambda_{n}^{-1}},
\end{align}
for some $n>0$, restricted to $[0,T]$.

Since $u$ is a weak solution to the Euler equation \eqref{equa-Euler},
we infer that $u_n$ satisfies
\begin{equation}\label{mhd2}
	\left\{\aligned
	&\p_t u_n+\lambda_{n}^{-2\alpha}(-\Delta)^{\alpha} u_n+ \div(u_n\otimes u_n)+\nabla P_n=\div  \mathring{R}_n ,  \\
	&\div u_n=0  , \\
	\endaligned
	\right.
\end{equation}
where the Reynolds stress
\begin{align}
	\mathring{R}_n
	:=& u_{n} \mathring\otimes u_{n}
	-((u \mathring\otimes u) *_{x} \phi_{\lambda_{n}^{-1}}) *_{t} \varphi_{\lambda_{n}^{-1}} +\lambda_{n}^{-2\alpha}\mathcal{R} (-\Delta)^{\alpha} u_{n}, \label{rnu}
\end{align}
and the pressure
\begin{align*}
	P_n:= P*_x \phi_{\lbb_n^{-1}} *_t \vf_{\lbb_n^{-1}}
	-|u_n|^2
	+ |u|^2 *_x \phi_{\lbb_n^{-1}} *_t \vf_{\lbb_n^{-1}}.
\end{align*}

Let $ \nu:=\nu_{n}:=\lambda_n^{-2\alpha}$
and $\widetilde{M}:= \|u \|_{H^{\widetilde{\beta}}_{t,x}}$.
We claim that for $a$ sufficiently large,
$(u_n,\mathring{R}_n)$ satisfy the iterative estimates \eqref{uh3}-\eqref{rl1} at level $q=n(\geq 1)$.

To this end, let us first consider the most delicate estimate \eqref{rl1}.
Using the Minkowski inequality and the Slobodetskii-type norm of Sobolev spaces
we have
(see, e.g., \cite[(6.35)]{lzz21})
\begin{align}  \label{u-un-lbbn}
	\|u-u_{n}\|_{L^2_{t,x}}
	\lesssim&\, \lambda_{n}^{- \widetilde{\beta}} \|u\|_{H^{\widetilde{\beta}}_{t,x}}\lesssim \lambda_{n}^{- \widetilde{\beta}}\widetilde{M}.
\end{align}

Moreover, we note that
\begin{align} \label{uu-uun-wtM}
	& \|u_{n} \otimes u_{n}-((u \otimes u) *_{x} \phi_{\lambda_{n}^{-1}}) *_{t} \varphi_{\lambda_{n}^{-1}}  \|_{L^{1}_{t,x}}  \notag \\
	\lesssim& \|u-u_n\|_{L^2_{t,x}}^2
	+ \|(|s|+|y|)^{4+2\wt \beta} \phi_{\lambda_{n}^{-1}} \varphi_{\lambda_{n}^{-1}} \|_{L^\9_{s,y}}
	\left\|\frac{u(t,x)-u(t-s,x-y)}{(|s|+|y|)^{2+\wt \beta}} \right\|_{L^2_{t,x}L^2_{s,y}}^2 \notag \\
	\lesssim&   \lbb_n^{-2\wt \beta} \|u\|_{H^{\wt \beta}_{t,x}}^2
	\lesssim    \lbb_n^{-2\wt \beta}  \wt M^2,
\end{align}
where the last step is due to \eqref{u-un-lbbn}.

Estimating as in \eqref{e5.18} we also get
\begin{align} \label{DeltaRu-L1-wtM}
	\|\lbb_n^{-2\alpha}\mathcal{R} (-\Delta)^{\alpha} u_n\|_{L^1_{t,x}}
	\lesssim \lbb_n^{-2\alpha} \( \|u_n\|_{L^1_{t}L^2_x}^{\frac{4-2\alpha}{3}} \|u_n\|_{L^1_{t}H^3_x}^{\frac{2\alpha-1}{3}} \)
	\lesssim \lbb_n^{-1} \|u\|_{L^2_{t,x}} 	
	\lesssim  \lbb_n^{-1}\wt M.
\end{align}

Thus, combing \eqref{rnu}, \eqref{uu-uun-wtM} and \eqref{DeltaRu-L1-wtM} altogether
we conclude that
\begin{align} \label{Ru-wtM-L1}
	\|\mathring{R}_n\|_{L^{1}_{t,x}}
	& \lesssim \|u_{n} \otimes u_{n}- ((u \otimes u) *_{x} \phi_{\lambda_{n}^{-1}}) *_{t} \varphi_{\lambda_{n}^{-1}}  \|_{L^{1}_{t,x}}
	+ \|\lambda_{n}^{-2\alpha}\mathcal{R} (-\Delta)^{\alpha} u_{n}\|_{L^1_{t,x}}  \notag  \\
	&  \lesssim \lambda_{n}^{-1}\widetilde{M}
	+ \lambda_{n}^{-2\widetilde{\beta}} \widetilde{M}^{2},
\end{align}
which verifies \eqref{rl1} at level $n$ by choosing $\beta$ and $\ve_R$ sufficiently small,
such that $\wt \beta>\ve_R/2 + \beta b$,
where $\ve_R$ and $\beta$ are as in the proof of Theorem~\ref{Prop-Iterat}.

Regarding the inductive estimate \eqref{uh3},
by Sobolev's embedding $H^1_{t} \hookrightarrow L^\9_{t}$
and Young's inequality,
\begin{align} \label{un-Bn-C1}
	\left\|u_{n}\right\|_{L^\9_tH^3_x} &\leq \sum_{0\leq |N|\leq 3}\|\nabla^{N} u_n\|_{L^\9_{t}L^2_x}\notag\\
	&\lesssim \sum_{0\leq |N|\leq 3}\|\|u*_t\varphi_{\lambda_n^{-1}}\|_{L^\9_t}\|_{L^2_x}\|\nabla^{N} \phi_{\lambda_n^{-1}}\|_{L^1_x}  \notag\\
	&\lesssim  \sum_{0\leq |N|\leq 3}\sum_{0\leq M\leq 1}\|u\|_{L^2_{t,x}}\|\p_t^{M}\varphi_{\lambda_n^{-1}}\|_{L^1_t}
	\|\nabla^{N} \phi_{\lambda_n^{-1}}\|_{L^1_x}  \notag \\
	&\lesssim \lbb_n^4 \wt M,
\end{align}
which verifies \eqref{uh3} at level $n$.

Moreover, by the Sobolev embedding $H^3_{t,x} \hookrightarrow L^\9_{t,x}$,
\begin{align} \label{un-pth3}
	\left\|\p_t u_{n}\right\|_{L^\9_tH^2_x}
	\lesssim \sum_{0\leq |N|\leq 2}\|\p_t \nabla^{N}u_n\|_{L^\9_{t,x}}
	\lesssim  \|u_n\|_{H^6_{t,x}}
	\lesssim \lbb_n^6 \wt M,
\end{align}
which verifies \eqref{upth2} at level $n$.

Finally, for the estimate \eqref{rh3},
by the Sobolev embedding $W^{4,1}_{t,x}\hookrightarrow L^\9_{t,x}$, we obtain
\begin{align}
	\|\mathring{R}_n\|_{L^\9_tH^3_x}\leq &  \sum_{0\leq|N|\leq 3}\|\nabla^{N}\mathring{R}_n\|_{L^\9_{t,x}} \lesssim \|\mathring{R}_n\|_{W^{7,1}_{t,x}} \notag\\
	\leq&  \|u_{n} \otimes u_{n}   - (u \otimes u) *_{x} \phi_{\lambda_{n}^{-1}} *_{t} \varphi_{\lambda_{n}^{-1}} \|_{W^{7,1}_{t,x}}
	+\lambda_{n}^{-2\alpha} \|\mathcal{R} (-\Delta)^{\alpha} u_{n}\|_{W^{7,1}_{t,x}} \notag \\
	\lesssim&\sum\limits_{0\leq M_{1}+M_{2}+N_1+N_2\leq 7}
	\|\partial_t^{M_{1}}\na^{N_1}u_n\|_{L^2_{t,x}}  \|\partial_t^{M_{2}}\na^{N_2}u_n\|_{L^2_{t,x}} \notag \\
	& \quad  + \sum\limits_{0\leq M+ N\leq 7}   \|u\|_{L^2_{t,x}}^2
	\|\na^{N} \phi_{\lambda_{n}^{-1}} \|_{L^1_x}   \|\partial_t^{M} \varphi_{\lambda_{n}^{-1}} \|_{L^1_t} + \lbb_n^{-2\alpha} \||\na|^{2\alpha-1} u_n\|_{H^7_{t,x}} \notag \\
	\lesssim& \lbb_n^7 \|u\|_{L^2_{t,x}}^2+\lbb_n^{-2\alpha} \|u_n\|_{L^2_{t,x}}^{\frac{2-\alpha}{5}}\|u_n\|_{H^{10}_{t,x}}^{\frac{3+\alpha}{5}}  \notag \\
	\lesssim&  \lbb_n^7 (\wt M+ \wt M^2 ).
\end{align}
Therefore, taking $a$ sufficiently large, we verify the inductive estimate \eqref{rh3}.

Thus, we can apply Theorem~\ref{Prop-Iterat} to
the approximate equation \eqref{mhd2} and
then let $q\rightarrow\infty$ to obtain a weak solution $u^{(\nu_{n})} \in H^{\beta'}_{t,x}$ to \eqref{equa-NS}
for some $\beta'\in (0,\beta/(8+\beta))$.

Furthermore,
estimating as in \eqref{interpo},
using \eqref{u-un-lbbn} and taking $\beta'$ sufficiently small such that
$0<\beta'<\min\{\wt \beta, \beta/(8+\beta)\}$
we deduce that for any $n\geq 1$,
\begin{align*}
   \|u^{(\nu_{n})}-u\|_{H^{\beta^{\prime}}_{t,x}}
		&\leq\|u^{(\nu_{n})}-u_{n}\|_{H^{\beta^{\prime}}_{t,x}}
		+\|u-u_{n}\|_{H^{\beta^{\prime}}_{t,x}} \notag\\
	&\leq C \(\sum_{q=n}^{\infty}  \lambda_{q+1}^{-\beta(1-\beta^{\prime})} \lambda_{q+1}^{8 \beta^{\prime}}+ \|u-u_n\|_{L^2_{t,x}}^{1-\frac{\beta'}{\wt \beta}}  \|u-u_n\|_{H^{\wt \beta}_{t,x}}^{\frac{\beta'}{\wt \beta}}\)\notag\\
		&\leq C \(\sum_{q=n}^{\infty}  \lambda_{q+1}^{-\beta(1-\beta^{\prime})} \lambda_{q+1}^{8 \beta^{\prime}}+  \lbb_n^{-(\wt \beta- \beta')} \wt M\)
	\leq \frac{1}{n},
\end{align*}
where the last step is valid for $a$ sufficiently large.
This verifies the strong convergence \eqref{convergence} in $H^{\beta^{\prime}}_{t,x}$.

Therefore, the proof of Theorem \ref{Thm-hyperNSE-Euler-limit} is complete.
\hfill $\square$

\section{Appendix} \label{Set-App}

In this section,
we collect some preliminary results used in the previous sections
and the well-posedness result in the critical space $L^{3/(2\alpha-1)}_x$.

\begin{lemma} ({\bf Geometric Lemma}, \cite[Lemma 4.1]{bcv21})
	\label{geometric lem 2}
	There exists a set $\Lambda \subset \mathbb{S}^2 \cap \mathbb{Q}^3$ that consists of vectors $k$
	with associated orthonormal bases $(k, k_1, k_2)$,  $\varepsilon_u> 0$,
	and smooth positive functions $\gamma_{(k)}: B_{\varepsilon_u}(\Id) \to \mathbb{R}$,
	where $B_{\varepsilon_u}(\Id)$ is the ball of radius $\varepsilon_u$ centered at the identity
	in the space of $3 \times 3$ symmetric matrices,
	such that for  $S \in B_{\varepsilon_u}(\Id)$ we have the following identity:
	\begin{equation}
		\label{sym}
		S = \sum_{k \in \Lambda} \gamma_{(k)}^2(S) k_1 \otimes k_1.
	\end{equation}
\end{lemma}

As pointed out in \cite{bcv21},
there exists $N_{\Lambda} \in \mathbb{N}$ such that
\begin{equation} \label{NLambda}
	\{ N_{\Lambda} k,N_{\Lambda}k_1 , N_{\Lambda}k_2 \} \subseteq N_{\Lambda} \mathbb{S}^2 \cap \mathbb{Z}^3.
\end{equation}
We denote by $M_*$ the geometric constant such that
\begin{align}
	\sum_{k \in \Lambda} \norm{\gamma_{(k)}}_{C^4(B_{\varepsilon_u}(\Id))} \leq M_*.
	\label{M bound}
\end{align}
This parameter  is universal and will be used later in the estimates of the size of perturbations.

\medskip
Then, we recall the $L^p$ decorrelation lemma introduced by \cite[Lemma 2.4]{cl21} (see also \cite[Lemma 3.7]{bv19b}), which is the key lemma to obtain the $L_{t,x}^2$ estimates of the perturbations.
\begin{lemma}[\cite{cl21}, Lemma 2.4]   \label{Decorrelation1}
	Let $\sigma\in \mathbb{N}$ and $f,g:\mathbb{T}^d\rightarrow \R$ be smooth functions. Then for every $p\in[1,\infty]$,
	\begin{equation}\label{lpdecor}
		\big|\|fg(\sigma\cdot)\|_{L^p(\T^d)}-\|f\|_{L^p(\T^d)}\|g\|_{L^p(\T^d)} \big|\lesssim \sigma^{-\frac{1}{p}}\|f\|_{C^1(\T^d)}\|g\|_{L^p(\T^d)}.
	\end{equation}
\end{lemma}

The following stationary phase lemma is a main tool to handle the errors of Reynolds stress.
\begin{lemma}[\cite{lt20}, Lemma 6; see also \cite{bv19b}, Lemma B.1] \label{commutator estimate1}
	Let $a \in C^{2}\left(\mathbb{T}^{3}\right)$. For all $1<p<\infty$ we have
	$$
	\left\||\nabla|^{-1} \P_{\neq 0}\left(a \P_{\geq k} f\right)\right\|_{L^{p}\left(\mathbb{T}^{3}\right)}
	\lesssim k^{-1}\left\|\nabla^{2} a\right\|_{L^{\infty}\left(\mathbb{T}^{3}\right)}\|f\|_{L^{p}\left(\mathbb{T}^{3}\right)}
	$$
 for any smooth function $f \in L^{p}\left(\mathbb{T}^{3}\right)$.
\end{lemma}

We close this section
with the well-posedness result for equation \eqref{equa-NS}
in the critical space $C_tL^{\frac{3}{2\a-1}}_x$.
The proof follows closely from the strategy
by Kato \cite{K84} and Cannone \cite{can97}.

\begin{theorem}  (Well-posedness in critical space $L^{\frac{3}{2\alpha-1}}_x$) \label{Thm-GWP-HNSE-Lp}
Let $\a\in (1,2)$.
Then, there exists $\delta>0$
such that for any $u_0\in  L^{\frac{3}{2\alpha-1}}_x$,
$\|u_0\|_{ L^{\frac{3}{2\alpha-1}}_x} \leq \delta$,
there exists a unique global solution $u$
to \eqref{equa-NS} satisfying
$u(0)=u_0$,
\begin{align*}
	&u \in C ([0, T]; L^{\frac{3}{2\a-1}} (\T^{3} ) ),\ \
	t^{\frac{2\a-1-3/p}{2\a}} u  \in C\left([0, T] ; L^{p}\left(\T^{3}\right)\right),
\end{align*}
where $\frac{3}{2\a-1} < p \leq \frac{6}{2\a-1}$,
and in addition
\begin{align*}
	\lim _{t \rightarrow 0} t^{\frac{2\a-1-3/p}{2\a}}\|u(t)\|_{L^p_x}=0.
\end{align*}
\end{theorem}

\begin{proof}
Let us formulate equation \eqref{equa-NS} in the mild form
\begin{align}\label{equa-mild-HNS}
	u(t)=e^{-t(-\Delta)^\a}u_0-\int_{0}^{t} e^{-(t-s)(-\Delta)^\a} \P_H \div(u \otimes u)(s)\d s.
  \end{align}
We choose the Banach space $X$
which consist of functions $v$ satisfying
\begin{align}
&v\in C ([0, T] ; L^{\frac{3}{2\a-1}} (\T^{3} ) ),\label{est-ul3.8}\\
&t^{\frac{2\a-1-3/p}{2\a}} v  \in C\left([0,T] ; L^{p}\left(\T^{3}\right)\right),\label{est-ulq.8}\\
&\lim _{t \rightarrow 0} t^{\frac{2\a-1-3/p}{2\a}}\|v(t)\|_{L^p_x}=0, \label{est-ulimit.8}
\end{align}
and are equipped with the norm
\begin{align*}
\|u\|_X:= \|u\|_{C_tL^{\frac{3}{2\a-1}}_x}
   +\sup_{t>0}t^{\frac{2\a-1-3/p}{2p}}\|u(t)\|_{L^{p}_x}.
\end{align*}

In order to prove Theorem \ref{Thm-GWP-HNSE-Lp},
by virtue of Lemma 1.5 of \cite{can97},
it suffices to prove that
\begin{enumerate}
	\item [$(i)$] If $u_0\in X$, then $e^{-t(\Delta)^\alpha} u_0 \in X$.
    \item [$(ii)$] The bilinear operator defined by
    \begin{align*}
		B(u,v)(t) :=-\int_{0}^{t} e^{-(t-s)(-\Delta)^\a}\mathbb{P}_H \div( u \otimes v)(s) \d s
	\end{align*}
	is bicontinuous in $X\times X \to X$.
\end{enumerate}
The important ingredients of proof are the following
standard estimates for the
semigroup $\{e^{-t(-\Delta)^\alpha}\}$,
that is,
for any $\rho, \in [1, \infty)$ and $1\leq \rho_2\leq \rho_1 <\9$,
\begin{align}
   & \|e^{-t(-\Delta)^\alpha} v\|_{L^\rho_x} \leq C \|v\|_{L^\rho_x}, \label{semigroup-bdd} \\
   & \|e^{-t(-\Delta)^\alpha} v\|_{L^{\rho_1}_x} \leq C
   t^{-\frac{3}{2\alpha}(\frac{1}{\rho_2}-\frac{1}{\rho_1})} \|v\|_{L^{\rho_2}_x}, \label{semigroup-rho1-rho2} \\
   &   \|\na e^{-t(-\Delta)^\alpha} v\|_{L^{\rho_1}_x} \leq C
   t^{-\frac{1}{2\alpha}-\frac{3}{2\alpha}(\frac{1}{\rho_2}-\frac{1}{\rho_1})} \|v\|_{L^{\rho_2}_x}.  \label{semigroup-nabla}
\end{align}

The  property $(i)$ follows immediately
from estimates \eqref{semigroup-bdd} and \eqref{semigroup-rho1-rho2}
with $\rho_1 = p$ and $\rho=\rho_2=3/(2\alpha-1)$.

Regarding the second property $(ii)$,
we use estimate \eqref{semigroup-nabla} with
$\rho_1=3/(2\alpha-1)$ and $\rho_2=p/2$
to get
\begin{align}\label{verify-ul3.8}
\| B(v_1,v_2)(t)\|_{L^{\frac{3}{2\a-1}}_x}
& \leq C\int_{0}^{t} (t-s)^{-\frac{1-\a+3/p}{\a}} \|v_1(s)\|_{L^p_x}\|v_2(s)\|_{L^p_x}\d s \notag\\
&\leq C \int_{0}^{t} (t-s)^{-\frac{1-\a+3/p}{\a}} s^{-\frac{2\a-1-3/p}{\a}}\d s
 \sup_{0\leq s\leq t} (s^{\frac{2\a-1-3/p}{2\a}}\|v_1(s)\|_{L^p_x} )
 \sup_{0\leq s\leq t} (s^{\frac{2\a-1-3/p}{2\a}}\|v_2(s)\|_{L^p_x})  \notag\\
&\leq C \|v_1\|_{X} \|v_2\|_{X}.
\end{align}
Similarly, applying \eqref{semigroup-nabla} with $\rho_1=p$
and $\rho_2=p/2$ we get
\begin{align}\label{verify-ulq.8}
\| B(v_1,v_2)(t)\|_{L^{p}_x}
& \leq C\int_{0}^{t} (t-s)^{-\frac{1+3/p}{2\a}} \|v_1(s)\|_{L^p_x}\|v_2(s)\|_{L^p_x}\d s \notag\\
&\leq C \int_{0}^{t} (t-s)^{-\frac{1+3/p}{2\a}} s^{-\frac{2\a-1-3/p}{\a}}\d s
\sup_{0\leq s\leq t} (s^{\frac{2\a-1-3/p}{2\a}}\|v_1(s)\|_{L^p_x} )
\sup_{0\leq s\leq t} (s^{\frac{2\a-1-3/p}{2\a}}\|v_2(s)\|_{L^p_x})  \notag\\
&\leq Ct^{-\frac{2\a-1-3/p}{2\a}} \|v_1\|_{X} \|v_2\|_{X}.
\end{align}
Hence, estimates \eqref{verify-ul3.8} and \eqref{verify-ulq.8}
together yield that
\begin{align}
   \| B(v_1,v_2)\|_{X} \leq C \|v_1\|_{X} \|v_2\|_{X}.
\end{align}
Moreover,
arguing as in the \eqref{verify-ulq.8}
we also infer from that,
for any $v_1,v_2\in X$,
\begin{align*}
    t^{\frac{2\a-1-3/p}{2\a}}
\| B(v_1, v_2)(t)\|_{L^{p}_x}
\leq C
      \sup\limits_{0\leq s\leq t}(s^{\frac{2\a-1-3/p}{2\a}}\| v_1(s)\|_{L^{p}_x})
	  \sup\limits_{0\leq s\leq t}(s^{\frac{2\a-1-3/p}{2\a}}\| v_2(s)\|_{L^{p}_x})
	 \to 0, \ as\ \ t\to 0.
\end{align*}
Thus, the second property $(ii)$ is verified.
The proof is complete.
\end{proof}

\noindent{\bf Acknowledgment.}
Yachun Li thanks the support by NSFC (No. 11831011, 12161141004).
Peng Qu thanks the supports by  NSFC (No. 12122104, 11831011)
and Shanghai Science and Technology Programs 21ZR1406000, 21JC1400600, 19JC1420101.
Deng Zhang  thanks the supports by NSFC (No. 11871337, 12161141004)
and Shanghai Rising-Star Program 21QA1404500.
Yachun Li and Deng Zhang are also grateful for the supports by
Institute of Modern Analysis--A Shanghai Frontier Research Center.


\begin{thebibliography}{99}

    \bibitem{ABC21}
	D. Albritton, E. Bru\'e, and M. Colombo.
	\newblock Non-uniqueness of Leray solutions of the forced Navier-Stokes equations.
    \newblock  arXiv:2112.03116, 2021.

	\bibitem{bbv20}
	R. Beekie, T. Buckmaster, and V. Vicol.
	\newblock Weak solutions of ideal {MHD} which do not conserve magnetic
	helicity.
	\newblock {\em Ann. PDE}, 6(1):Paper No. 1, 40, 2020.

	\bibitem{BP08}
    J. Bourgain and N. Pavlovi\'c.
	\newblock Ill-posedness of the Navier-Stokes equations in a critical space in 3D.
	\newblock {\em J. Funct. Anal.}, 255(9):2233-2247, 2008.

    \bibitem{BM18}
	H. Brezis and P. Mironescu.
	\newblock Gagliardo-Nirenberg inequalities and non-inequalities: the full story.
	\newblock {\em Ann. Inst. H. Poincar\'{e} Anal. Non Lin\'{e}aire}, 35(5):1355--1376, 2018.
	
    \bibitem{B15}
	T. Buckmaster.
	\newblock Onsager's conjecture almost everywhere in time.
	\newblock {\em Comm. Math. Phys.}, 333(3):1175-1198, 2015.

	\bibitem{bcv21}
	T. Buckmaster, M. Colombo, and V. Vicol.
	\newblock Wild solutions of the Navier-Stokes equations whose singular sets in time have Hausdorff dimension strictly less than 1.
	\newblock {\em J. Eur. Math. Soc.}, 24:3333-3378, 2021.
	
	\bibitem{bdis15}
	T. Buckmaster, C. De~Lellis, P. Isett, and L. Sz\'{e}kelyhidi, Jr.
	\newblock Anomalous dissipation for {$1/5$}-{H}\"{o}lder {E}uler flows.
	\newblock {\em Ann. of Math. (2)}, 182(1):127--172, 2015.
	
	\bibitem{bdls16}
	T. Buckmaster,  C. De Lellis, and L. Sz\'{e}kelyhidi, Jr.
	\newblock Dissipative Euler flows with Onsager-critical spatial regularity.
	\newblock {\em Comm. Pure Appl. Math.} 69(9):1613-1670, 2016.
	
	\bibitem{bdsv19}
	T. Buckmaster, C. De~Lellis, L. Sz\'{e}kelyhidi, Jr., and
	V. Vicol.
	\newblock Onsager's conjecture for admissible weak solutions.
	\newblock {\em Comm. Pure Appl. Math.}, 72(2):229--274, 2019.
	
	\bibitem{bv19b}
	T. Buckmaster and V. Vicol.
	\newblock Nonuniqueness of weak solutions to the {N}avier-{S}tokes equation.
	\newblock {\em Ann. of Math. (2)}, 189(1):101--144, 2019.
	
	
	\bibitem{bv19r}
	T. Buckmaster and V. Vicol.
	\newblock Convex integration and phenomenologies in turbulence.
	\newblock {\em EMS Surv. Math. Sci.}, 6(1-2):173--263, 2019.
	
	\bibitem{bv21}
	T. Buckmaster and V. Vicol.
	\newblock Convex integration constructions in hydrodynamics.
	\newblock {\em Bull. Amer. Math. Soc. (N.S.)}, 58(1):1--44, 2021.
	
   \bibitem{BT08}
   N. Burq and N. Tzvetkov.
   \newblock Random data Cauchy theory for supercritical wave equations. II. A global existence result.
   \newblock {\em Invent. Math.}, 173(3):477-496, 2008.

   \bibitem{CKN82}
   L. Caffarelli, R. Kohn, and L. Nirenberg.
  \newblock  Partial regularity of suitable weak solutions of the Navier-Stokes equations.
   \newblock {\em Comm. Pure Appl. Math.}, 35(6):771-831, 1982.

   \bibitem{can97}
	M. Cannone.
   \newblock  A generalization of a theorem by Kato on Navier-Stokes equations.
   \newblock  {\em Rev. Mat. Iberoam.}, 13(3):515-541, 1997.

	\bibitem{C04}
	M. Cannone.
	\newblock {\em Harmonic analysis tools for solving the incompressible Navier-Stokes equations.}
	\newblock Handbook of mathematical fluid dynamics. Vol. III, 161-244, North-Holland, Amsterdam, 2004.
	
	\bibitem{CD14}
	A. Cheskidov and M. Dai.
	\newblock Norm inflation for generalized Navier-Stokes equations.
	\newblock {\em  Indiana Univ. Math. J.}, 63(3):869-884, 2014.
	
	\bibitem{cl21}
	A. Cheskidov and X. Luo.
	\newblock Nonuniqueness of weak solutions for the transport equation at
	critical space regularity.
	\newblock {\em Ann. PDE}, 7(1):Paper No. 2, 45, 2021.

	\bibitem{cl20.2}
	A. Cheskidov and X. Luo.
	\newblock Sharp uniqueness for the Navier-Stokes equations.
	\newblock {\em Invent. math.}, https://doi.org/10.1007/s00222-022-01116-x, 2022.

	\bibitem{cl21.2}
	A. Cheskidov and X. Luo.
	\newblock $L^2$-critical nonuniqueness for the 2D Navier-Stokes equations.
	\newblock arXiv:2105.12117, 2021.

	\bibitem{cl22}
	A. Cheskidov and X. Luo.
	\newblock Extreme temporal intermittency in the linear Sobolev transport: almost smooth nonunique solutions.
	\newblock  arXiv:2204.0895, 2022.
	
	\bibitem{CS12}
	A. Cheskidov and R. Shvydkoy.
	\newblock Ill-posedness for subcritical hyperdissipative Navier-Stokes equations
	in the largest critical spaces.
	\newblock {\em J. Math. Phys.}, 53(7), 115620, 7 pp, 2012.
	
	\bibitem{CCT03}
	M.Christ, J.Colliander, and T.Tao.
	\newblock Asymptotics, frequency modulation, and low regularity ill-posedness for canonical defocusing equations.
	\newblock {\em Amer. J. Math.}, 125(6):1235-1293, 2003.
	
	\bibitem{CDR18}
	M. Colombo, C. De Lellis, and L. De Rosa.
	\newblock Ill-posedness of Leray solutions for the hypodissipative Navier-Stokes equations.
	\newblock {\em Comm. Math. Phys.}, 362(2):659--688, 2018.
	
	\bibitem{CDM20}
	M. Colombo, C. De Lellis, and A. Massaccesi.
    \newblock The generalized Caffarelli-Kohn-Nirenberg theorem for the hyperdissipative Navier-Stokes system.
    \newblock {\em Comm. Pure Appl. Math.}, 73(3):609-663, 2020.

	
	\bibitem{dls09}
	C. De~Lellis and L. Sz\'{e}kelyhidi, Jr.
	\newblock The {E}uler equations as a differential inclusion.
	\newblock {\em Ann. of Math. (2)}, 170(3):1417--1436, 2009.
	
	\bibitem{dls10}
	C. De~Lellis and L. Sz\'{e}kelyhidi, Jr.
	\newblock On admissibility criteria for weak solutions of the {E}uler
	equations.
	\newblock {\em Arch. Ration. Mech. Anal.}, 195(1):225--260, 2010.
	
	\bibitem{dls13}
	C. De~Lellis and L. Sz\'{e}kelyhidi, Jr.
	\newblock Dissipative continuous {E}uler flows.
	\newblock {\em Invent. Math.}, 193(2):377--407, 2013.
	
	\bibitem{dls14}
	C. De~Lellis and L. Sz\'{e}kelyhidi, Jr.
	\newblock Dissipative Euler flows and Onsager's conjecture.
	\newblock {\em J. Eur. Math. Soc.}, 16(7):1467--1505, 2014.
	
	\bibitem{dls17}
	C. De~Lellis and L. Sz\'{e}kelyhidi, Jr.
	\newblock High dimensionality and h-principle in {PDE}.
	\newblock {\em Bull. Amer. Math. Soc. (N.S.)}, 54(2):247--282, 2017.
	
	\bibitem{DR19}
	L. De Rosa.
	\newblock Infinitely many Leray-Hopf solutions for the fractional Navier-Stokes equations.
	\newblock {\em Comm. Partial Differential Equations}, 44(4): 335--365, 2019.
	
	\bibitem{G08}
	P. Germain.
	\newblock The second iterate for the Navier-Stokes equation.
	\newblock {\em J. Funct. Anal.}, 255(9):2248-2264, 2008.
	
	\bibitem{iss03}
	L.~Escauriaza, G.~A. Seregin, and V.~\v{S}ver\'{a}k.
	\newblock {$L_{3,\infty}$}-solutions of {N}avier-{S}tokes equations and
	backward uniqueness.
	\newblock {\em Uspekhi Mat. Nauk}, 58(2(350)):3--44, 2003.
	
	\bibitem{FJR72}
	E.B. Fabes, B.F. Jones, and N.M. Rivière.
	\newblock The initial value problem for the Navier-Stokes equations with data in $L^p$.
	\newblock {\em Arch. Ration. Mech. Anal.}, 45:222-240, 1972.
	
	\bibitem{fls21}
	D. Faraco, S. Lindberg, and L. Sz\'{e}kelyhidi, Jr.
	\newblock Bounded solutions of ideal {MHD} with compact support in space-time.
	\newblock {\em Arch. Ration. Mech. Anal.}, 239(1):51--93, 2021.

	\bibitem{fls21.2}
	D. Faraco, S. Lindberg, and L. Sz\'{e}kelyhidi, Jr.
	\newblock Magnetic helicity, weak solutions and relaxation of ideal MHD.
	\newblock arXiv:2109.09106, 2021.

	\bibitem{FK64}
	H. Fujita and T. Kato.
	\newblock On the Navier-Stokes initial value problem. I.
	\newblock {\em Arch. Ration. Mech. Anal.}, 16:269-315, 1964.
	
	\bibitem{FLRT00}
	G. Furioli,  P.G. Lemari\'e-Rieusset, and E. Terraneo.
	\newblock Unicit\'e dans $L^3(\mathbb{R}^3)$ et d'autres espaces fonctionnels limites pour Navier-Stokes.
	\newblock {\em Rev. Mat. Iberoam.}, 16(3):605-667, 2000.
	
    \bibitem{GIP03}
	I. Gallagher, D. Iftimie, and F. Planchon.
	\newblock Asymptotics and stability for global solutions to the Navier-Stokes equations.
	\newblock {\em Ann. Inst. Fourier (Grenoble)}, 53(5):1387-1424, 2003.


	\bibitem{hopf1951}
	E. Hopf.
	\newblock \"{U}ber die {A}nfangswertaufgabe f\"{u}r die hydrodynamischen
	{G}rundgleichungen.
	\newblock {\em Math. Nachr.}, 4:213--231, 1951.
	
	\bibitem{I18}
	P. Isett.
	\newblock A proof of {O}nsager's conjecture.
	\newblock {\em Ann. of Math. (2)}, 188(3):871--963, 2018.
	
	\bibitem{js14}
	H. Jia and V. \v{S}ver\'{a}k.
	\newblock Local-in-space estimates near initial time for weak solutions of the
	{N}avier-{S}tokes equations and forward self-similar solutions.
	\newblock {\em Invent. Math.}, 196(1):233--265, 2014.
	
	\bibitem{js15}
	H. Jia and V. \v{S}ver\'{a}k.
	\newblock Are the incompressible 3d {N}avier-{S}tokes equations locally
	ill-posed in the natural energy space?
	\newblock {\em J. Funct. Anal.}, 268(12):3734--3766, 2015.


	\bibitem{K84}
	T. Kato.
    \newblock Strong $L^p$-solutions of the Navier-Stokes equation in $\R^m$, with applications to weak solutions.
	\newblock {\em Math. Z.}, 187(4):471-480, 1984.
	
	\bibitem{KP02}
	N.H. Katz and N. Pavlovi\'c.
	\newblock A cheap Caffarelli-Kohn-Nirenberg inequality for the Navier-Stokes equation with hyper-dissipation.
	\newblock {\em Geom. Funct. Anal.}, 12(2):355-379, 2002.
	
	\bibitem{KPV01}
	C.E. Kenig, G. Ponce, and L. Vega.
	\newblock On the ill-posedness of some canonical dispersive equations.
	\newblock {\em Duke Math. J.}, 106(3):617-633, 2001.
	

	\bibitem{K00}
    S. Klainerman.
    \newblock P{DE} as a unified subject.
    \newblock {\em Geom. Funct. Anal.}, Number Special Volume, Part I, pages 279--315, 2000.
    \newblock GAFA 2000 (Tel Aviv, 1999).

	\bibitem{K17}
	S. Klainerman.
	\newblock On Nash's unique contribution to analysis in just three of his papers.
	\newblock {\em Bull. Amer. Math. Soc.} (N.S.), 54(2):283--305, 2017.
	
	\bibitem{kt01}
	H. Koch and D. Tataru.
	\newblock Well-posedness for the {N}avier-{S}tokes equations.
	\newblock {\em Adv. Math.}, 157(1):22--35, 2001.
	
	\bibitem{L67}
	O.A. Lady\v{z}enskaja.
	\newblock Uniqueness and smoothness of generalized solutions of Navier-Stokes equations. (Russian)
	\newblock {\em Zap. Nau\v{c}n. Sem. Leningrad. Otdel. Mat. Inst. Steklov. (LOMI)}, 5:169-185, 1967.
	
	\bibitem{LR16}
	P.G. Lemari\'e-Rieusset.
	\newblock {\em The Navier-Stokes problem in the 21st century.}
    \newblock CRC Press, Boca Raton, FL, xxii+718 pp, 2016.
	
	
	\bibitem{leray1934}
	J. Leray.
	\newblock Sur le mouvement d'un liquide visqueux emplissant l'espace.
	\newblock {\em Acta Math.}, 63(1):193--248, 1934.

	\bibitem{lzz21}
	Y. Li, Z. Zeng, and D. Zhang.
	\newblock Non-uniqueness of weak solutions to 3D magnetohydrodynamic equations.
	\newblock arXiv:2112.10515, 2021.
	
	\bibitem{Lin98}
	F. Lin.
	\newblock A new proof of the {C}affarelli-{K}ohn-{N}irenberg theorem.
	\newblock {\em Comm. Pure Appl. Math.}, 51(3):241--257, 1998.
	
	\bibitem{lions69}
	J.-L. Lions.
	\newblock {\em Quelques m\'{e}thodes de r\'{e}solution des probl\`emes aux
		limites non lin\'{e}aires}.
	\newblock Dunod; Gauthier-Villars, Paris, 1969.
	
	\bibitem{LM01}
	P.-L. Lions and N. Masmoudi.
	\newblock Uniqueness of mild solutions of the Navier-Stokes system in $L^N$.
   \newblock {\em Comm. Partial Differential Equations}, 26(11-12):2211-2226, 2001.
	
	\bibitem{lq20}
	T. Luo and P. Qu.
	\newblock Non-uniqueness of weak solutions to 2{D} hypoviscous
	{N}avier-{S}tokes equations.
	\newblock {\em J. Differential Equations}, 269(4):2896--2919, 2020.
	
	
	\bibitem{lt20}
	T. Luo and E.S. Titi.
	\newblock Non-uniqueness of weak solutions to hyperviscous {N}avier-{S}tokes
	equations: on sharpness of {J}.-{L}. {L}ions exponent.
	\newblock {\em Calc. Var. Partial Differential Equations}, 59(3):Paper No. 92, 15, 2020.
	
	\bibitem{luo19}
	X. Luo.
	\newblock Stationary solutions and nonuniqueness of weak solutions for the
	{N}avier-{S}tokes equations in high dimensions.
	\newblock {\em Arch. Ration. Mech. Anal.}, 233(2):701--747, 2019.
	
	\bibitem{MS99}
	J.C. Mattingly and Ya.G. Sinai.
	\newblock An elementary proof of the existence and uniqueness theorem for the Navier-Stokes equations.
	\newblock {\em Commun. Contemp. Math.}, 1(4): 497-516, 1999.
	
	\bibitem{M99}
	Y. Meyer.
	\newblock {\em Wavelets, paraproducts, and Navier-Stokes equations.}
	\newblock Current developments in mathematics, 1996 (Cambridge, MA), 105-212, Int. Press, Boston, MA, 1997.
	
	\bibitem{prodi59}
	G. Prodi.
	\newblock Un teorema di unicit\`a per le equazioni di {N}avier-{S}tokes.
	\newblock {\em Ann. Mat. Pura Appl. (4)}, 48:173--182, 1959.
	
	\bibitem{R02}
	F. Ribaud.
    \newblock  A remark on the uniqueness problem for the weak solutions of Navier-Stokes equations.
	\newblock {\em Ann. Fac. Sci. Toulouse Math.}, 6(2):225-238, 2002.

    \bibitem{ST87}
	H. Schmeisser and H. Triebel.
	\newblock {\em Topics in Fourier analysis and function spaces.}
	\newblock A Wiley-Interscience Publication. John Wiley \& Sons, Ltd., Chichester, 300 pp, 1987.

	\bibitem{serrin62}
	J. Serrin.
	\newblock On the interior regularity of weak solutions of the {N}avier-{S}tokes
	equations.
	\newblock {\em Arch. Ration. Mech. Anal.}, 9:187--195, 1962.
	
	\bibitem{S76}
    V. Scheffer.
	\newblock Partial regularity of solutions to the Navier-Stokes equations.
	\newblock {\em Pacific J. Math.}, 66(2):535-552, 1976.

    \bibitem{S77}
	V. Scheffer.
	\newblock  Hausdorff measure and the Navier-Stokes equations.
	\newblock {\em Comm. Math. Phys.}, 55 (1977), no. 2, 97-112.

	\bibitem{SvW84}
	H. Sohr and W. von Wahl.
	\newblock On the singular set and the uniqueness of weak solutions of the Navier-Stokes equations.
	\newblock {\em Manuscripta Math.}, 49(1):27-59, 1984.

	\bibitem{tao09}
	T. Tao.
	\newblock Global regularity for a logarithmically supercritical
	hyperdissipative {N}avier-{S}tokes equation.
	\newblock {\em Anal. PDE}, 2(3):361--366, 2009.
	
	\bibitem{Vasseur07}
	A. Vasseur.
	\newblock A new proof of partial regularity of solutions to {N}avier-{S}tokes
	equations.
	\newblock {\em NoDEA Nonlinear Differential Equations Appl.}, 14(5-6):753--785, 2007.
	
	\bibitem{W15}
	B. Wang.
    \newblock Ill-posedness for the Navier-Stokes equations in critical Besov spaces $\dot{B}^{-1}_{\infty,q}$.
	\newblock {\em Adv. Math.}, 268:350-372, 2015.

	\bibitem{W06}
	J. Wu.
    \newblock Lower bounds for an integral involving fractional Laplacians and the generalized Navier-Stokes equations in Besov spaces.
   \newblock  {\em Comm. Math. Phys.}, 263(3):803-831, 2006.
	
	\bibitem{XZ22}
	B. Xia and D. Zhang.
	\newblock Ill-posedness of quintic fourth order Schr\"odinger equation.
	\newblock arXiv:2202.03020, 2022.
	
	\bibitem{Y10}
	T. Yoneda.
	\newblock Ill-posedness of the 3D-Navier-Stokes equations in a generalized Besov space near $BMO^{-1}$.
	\newblock {\em J. Funct. Anal.}, 258(10):3376-3387, 2010.

	\bibitem{Z07}
	Y. Zhou.
    \newblock Regularity criteria for the generalized viscous MHD equations.
	\newblock {\em Ann. Inst. H. Poincaré C Anal. Non Linéaire}, 24(3):491-505, 2007.	
\end{thebibliography}
\end{document}